\documentclass[reqno]{amsart}
\usepackage[margin = 1in]{geometry}
\usepackage{amsmath, amssymb, amsthm, fancyhdr, verbatim, graphicx, relsize, mathtools, xfrac}
\usepackage{enumerate}
\usepackage{enumitem} 
\usepackage[all]{xy}
\usepackage[dvipsnames]{xcolor}
\usepackage{mathrsfs}
\usepackage{tikz}
\usetikzlibrary{arrows.meta, positioning, calc, patterns, decorations.pathmorphing, decorations.pathreplacing}
\usepackage{tikz-cd}
\usetikzlibrary{calc}
\usepackage[T1]{fontenc} 
\usepackage{framed}
\usepackage[hypertexnames=false]{hyperref}
\usepackage[OT2, T1]{fontenc}  
\usepackage[titletoc]{appendix}
\usepackage{bbm}
\usepackage{adjustbox}
\usepackage{amssymb} 
\usepackage{float}

\numberwithin{equation}{subsection}

\DeclareSymbolFont{cyrletters}{OT2}{wncyr}{m}{n}
\DeclareMathSymbol{\Sha}{\mathalpha}{cyrletters}{"58}

\usepackage{color}

\newcommand{\F}{\mathbf{F}}

\newcommand{\CC}{\mathbf{C}}
\newcommand{\G}{\mathbf{G}}

\newcommand{\wt}[1]{\widetilde{#1}}

\newcommand{\Q}{\mathbf{Q}}
\newcommand{\Z}{\mathbf{Z}}

\newcommand{\Gal}{\operatorname{Gal}}
\newcommand{\Mat}{\operatorname{Mat}}

\newcommand{\ol}[1]{\overline{#1}}
\newcommand{\wh}[1]{\widehat{#1}}

\newcommand{\co}{\colon}

\newcommand{\ld}{{}^L}

\newcommand{\inj}{\hookrightarrow}

\newcommand\cA{\mathcal{A}}
\newcommand\cB{\mathcal{B}}

\newcommand\cE{\mathcal{E}}
\newcommand\cF{\mathcal{F}}
\newcommand\cG{\mathcal{G}}

\newcommand\cL{\mathcal{L}}

\newcommand\cO{\mathcal{O}}

\newcommand\cS{\mathcal{S}}
\newcommand\cT{\mathcal{T}}

\newcommand\cZ{\mathcal{Z}}

\newcommand\sH{\mathscr{H}}

\newcommand\fm{\mathfrak{m}}

\DeclareMathOperator{\GL}{GL}
\DeclareMathOperator{\SL}{SL}

\DeclareMathOperator{\Tr}{Tr}

\DeclareMathOperator{\PGL}{PGL}
\DeclareMathOperator{\Sp}{Sp}

\DeclareMathOperator{\Hom}{Hom}

\DeclareMathOperator{\ord}{ord}
\DeclareMathOperator{\Aut}{Aut}
\DeclareMathOperator{\Rep}{Rep}
\DeclareMathOperator{\Irr}{Irr}
\DeclareMathOperator{\Nm}{Nm}

\DeclareMathOperator{\End}{End}

\DeclareMathOperator{\unr}{unr}
\DeclareMathOperator{\Res}{Res}

\DeclareMathOperator{\Span}{Span}

\DeclareMathOperator{\tors}{tors}

\DeclareMathOperator{\ind}{ind}

\DeclareMathOperator{\FS}{FS}
\DeclareMathOperator{\cInd}{c-Ind}

\DeclareMathOperator{\red}{red}

\DeclareMathOperator{\der}{der}

\DeclareMathOperator{\rk}{rk}
\DeclareMathOperator{\Gla}{Gla}

\DeclareMathOperator{\DR}{DR}

\newcommand{\tw}[1]{\langle #1 \rangle}

\DeclareMathOperator{\Fr}{Fr}

\RequirePackage{xspace}

\newcommand{\rH}{\ensuremath{\mathrm{H}}\xspace}

\newcommand{\rT}{\ensuremath{\mathrm{T}}\xspace}

\newcommand{\bF}{\mathbf{F}}
\newcommand{\bT}{\mathbf{T}}

\renewcommand{\ss}{\mathrm{ss}}

\newtheorem{thm}{Theorem}[subsection]
\newtheorem{lemma}[thm]{Lemma}
\newtheorem{prop}[thm]{Proposition}
\newtheorem{cor}[thm]{Corollary}

\theoremstyle{remark}
\newtheorem{remark}[thm]{Remark} 
 
\newtheorem{defn}[thm]{Definition}

\newtheorem{hypothesis}[thm]{Hypothesis}

\newtheorem{example}[thm]{Example}

\newtheorem{question}[thm]{Question}

\makeatletter
\def\th@remark{%
  \thm@headfont{\bfseries}%
  \normalfont 
  \thm@preskip \thm@preskip 
  \thm@postskip\thm@preskip
}
\def\imod#1{\allowbreak\mkern5mu({\operator@font mod}\,\,#1)}
\makeatother

\numberwithin{equation}{subsection}
\numberwithin{figure}{subsection}

\title[Modular functoriality for finite groups]{Modular functoriality for finite groups}

\author{Sean Cotner}

\begin{document}

\begin{abstract}
We develop an extension of Deligne--Lusztig theory to certain (possibly infinite type) disconnected reductive groups arising from the special fibers of point stabilizers in the Bruhat--Tits building, which we call \emph{paraductive}. We then compute explicit lower bounds for the Tate cohomology of representations of paraductive groups, relating these to Shintani descent, Lusztig restriction, and the Glauberman correspondence. As an application, using Feng's modular functoriality and Scholze's independence of $\ell$, we compute the Fargues--Scholze L-parameters of non-singular depth $0$ cuspidal representations of a (possibly wildly ramified) reductive group over a nonarchimedean local field.
\end{abstract}

\maketitle

\tableofcontents


\section{Introduction}

\subsection{The main theorem}

This paper forms one of the technical cores of the author's work with Tony Feng \cite{CF26a}, \cite{CF26b} comparing the Fargues--Scholze Local Langlands Correspondence \cite{FS} to Kaletha's Local Langlands Correspondence for non-singular supercuspidal representations \cite{Kal21b}, as well as an inertial extension of the latter to singular cuspidal representations. The two recent innovations making this comparison possible are Feng's ``modular functoriality'' \cite[Theorem~1.3.1]{F24} and Scholze's ``independence of $\ell$'' \cite[Theorem~1.1]{Sch25}; the primary motivation of this paper is to set up enough machinery to make modular functoriality easy to apply in practice. We will illustrate our results by proving the following theorem, which is related to, but is neither a strict generalization nor a strict specialization of, the main theorem of \cite{CF26b}.

\begin{thm}[Theorem~\ref{thm:main-depth-0-comparison}]\label{thm:intro-debacker-reeder}
    If $G$ is a connected reductive group over a non-archi\-medean local field $F$ and $\pi$ is a depth $0$ non-singular supercuspidal $\ol\Q_\ell$-representation of $G(F)$, then the Fargues--Scholze L-parameter $\rho^{\FS}(\pi)$ is equal to the L-parameter $\rho^{\DR}(\pi)$ associated to $\pi$ by the work of DeBacker--Reeder \cite{DR09} and Kaletha \cite{Kal19}, \cite{Kal21b}.
\end{thm}

Notably, Theorem~\ref{thm:intro-debacker-reeder} allows $G$ to be wildly ramified and $p$ to be arbitrary; this is why it is not a special case of \cite[Theorem 1.1.1]{CF26b}. If $G$ is quasi-split, then it was known from \cite[Corollary 1.1.1]{DL26} that the parameter $\rho^{\FS}(\pi)$ in Theorem~\ref{thm:intro-debacker-reeder} is tamely ramified. Theorem~\ref{thm:intro-debacker-reeder} was obtained in Eteve's thesis \cite[Theorem~2.3.9]{Ete23} when $G$ is split and $F$ is a function field. Building on this, it has also recently been obtained conditionally by Fu \cite{Fu26} when $G$ is unramified and $F$ is arbitrary.\footnote{In fact, neither \cite{DL26}, \cite{Ete23}, nor \cite{Fu26} requires that $\pi$ is non-singular, but in general their results only concern $\rho^{\FS}(\pi)|_{I_F}$. In Remark~\ref{rmk:singular-generalization}, we briefly describe two methods of similarly extending Theorem~\ref{thm:intro-debacker-reeder} beyond the non-singular case.} The analogue of Theorem~\ref{thm:intro-debacker-reeder} for Zhu's depth $0$ Local Langlands Correspondence is proven in \cite[\S 5.3.4]{Zhu25} when $G$ is unramified, and in view of \cite{GHILZ26} it seems plausible that it will eventually be possible (with considerably more work) to use this result to deduce Theorem~\ref{thm:intro-debacker-reeder} when $G$ is unramified. Our proof is independent of, and bears little resemblance to, these prior arguments.

\subsection{Fargues--Scholze and DeBacker--Reeder}

Let $F$ be a non-archimedean local field with residue field $\F_q$, let $W_F$ be the Weil group of $F$, and let $G$ be a connected reductive $F$-group. Let $\wh G$ be the Langlands dual group of $G$, defined over $\Z$ and equipped with a pinning-preserving action of $W_F$, and let $\ld G = \wh G \rtimes W_F$ be the L-group of $G$. Choose a prime number $\ell$ not dividing $q$. Let $k$ be a field among $\ol\Q_\ell$ and $\ol\F_\ell$, let $\Pi_k(G)$ be the set of irreducible smooth $k$-representations of $G(F)$ up to isomorphism, and let $\Phi_k^{\ss}(G)$ denote the set of semisimple L-parameters $W_F \to \ld G(k)$.

\subsubsection{Fargues--Scholze}

The Fargues--Scholze Local Langlands Correspondence \cite[\S I.9]{FS} is a map of sets
\[
\rho^{\FS}\co \Pi_k(G) \to \Phi_k^{\ss}(G),
\]
which is widely believed to be the ``true'' (semisimple) Local Langlands Correspondence.

The semisimple Local Langlands Correspondence is expected to satisfy many properties, a modern list of which can be found in \cite[\S 6]{Tai25}. Many of these properties are known for $\rho^{\FS}$, including compatibility with the usual Local Langlands Correspondence for tori, parabolic induction, and the classical case of $G = \GL_n$ \cite[Theorem~I.9.6]{FS}; the latter has since been extended to many other groups by many authors \cite{HKW22}, \cite{BMHN24}, \cite{Ham25}, \cite{Pen26}, \cite{Han26}, \cite{DvHKZ26}. However, beyond classical groups and their forms, little is known about $\rho^{\FS}$: for instance, David Hansen has informed us that if $F$ is a $p$-adic field then the literature does not exhibit a single supercuspidal representation $\pi$ of $\mathrm{E}_8(F)$ such that $\rho^{\FS}(\pi)$ is nontrivial.

\subsubsection{DeBacker--Reeder}\label{sss:intro-dr}

Another approach to constructing the ``true'' (semisimple) local Langlands correspondence was introduced in \cite{DR09} and developed in \cite{Kal19}, \cite{Kal21b}; we will briefly recall it here in the special case that $G$ is semisimple and simply connected. If $\pi$ is as in Theorem~\ref{thm:intro-debacker-reeder}, then \cite[Proposition 6.8]{MP96} shows that there is a vertex $[x]$ in the (reduced) Bruhat--Tits building $\cB(G)$ and an irreducible cuspidal $\ol\Q_\ell$-representation $\tau$ of $G(F)_{[x]}/G(F)_{x,0+}$ such that 
\begin{equation}\label{eqn:intro-depth-0-rep}
    \pi \cong \cInd_{G(F)_{[x]}}^{G(F)}(\tau).
\end{equation}
Recall that there exists a connected reductive $\F_q$-group $\ol G_{[x]}$ such that $G(F)_{[x]}/G(F)_{[x],0+} \cong \ol G_{[x]}(\F_q)$. Using Deligne--Lusztig theory \cite{DL76} and deformation theory for tori, one extracts a maximally unramified anisotropic maximal $F$-torus $T \subset G$ and a depth $0$ character $\theta$ of $T(F)$. There is a canonical L-embedding $\ld j_{T,G}\co \ld T \to \ld G$, and one defines
\[
\rho^{\DR}(\pi) = \ld j_{T,G} \circ \ld\theta,
\]
where $\ld\theta\co W_F \to \ld T(\ol\Q_\ell)$ is the L-homomorphism arising from the local Langlands correspondence for tori.

If $G$ is not semisimple and simply connected, then one can still define $\rho^{\DR}$ in a similar manner. The main problem is that the group $\ol G_{[x]}$ appearing above may no longer be connected, nor even of finite type; one therefore requires a version of Deligne--Lusztig theory for disconnected reductive groups. Such a theory was developed in \cite{DM94} and \cite{Kal21b}, but neither reference goes quite as far as we need in practice. We will discuss this further below after describing the key tool which allows us to compare $\rho^{\FS}$ and $\rho^{\DR}$.

\subsection{Modular functoriality in the local Langlands program}

Recently, Feng \cite{F24} has given a new local method for studying $\rho^{\FS}$ when $k = \ol\F_\ell$. We briefly recall the set-up, which is inspired by Treumann--Venkatesh \cite{TV}. 

Let $\sigma$ be an $F$-automorphism of $G$ of order $\ell$, and let $H = (G^\sigma)^\circ$ denote the identity component of the $\sigma$-fixed subgroup of $G$. Let $\ol\pi$ be a smooth irreducible $\ol\F_\ell$-representation of $G(F) \rtimes \langle\sigma\rangle$. Define the \emph{Tate cohomology} groups $\rT^a(\sigma, \ol\pi)$ for $a \in \Z/2$ by
\[
\rT^0(\sigma, \ol\pi) \coloneqq \frac{\ker(1-\sigma | \ol\pi)}{N_\sigma(\ol\pi)} \quad \text{ and } \quad 
\rT^1(\sigma, \ol\pi) \coloneqq \frac{\ker(N_\sigma | \ol\pi)}{( 1-\sigma  ) \ol\pi},
\]
where $N_\sigma = \sum_{i=0}^{\ell-1} \sigma^i$. Observe that $\rT^a(\sigma, \ol\pi)$ is a smooth $\ol\F_\ell$-representation of $H(F)$. In fact, recent work of Dhar--Nadimpalli \cite[Theorem~1.1]{DN25} proves a conjecture of Treumann--Venkatesh asserting that $\rT^a(\sigma, \ol\pi)$ is of finite length.

\begin{thm}[{\cite[Theorem~1.3.1]{F24}}, ``modular functoriality'']\label{thm:feng-modular-functoriality}
    There are constants $b(\wh H)$ and $b(\wh G)$ such that if $\ell > \max(b(\wh G), b(\wh H))$, then for every $\sigma$ as above there exists an L-homomorphism $\ld\psi\co \ld H_{\ol\F_\ell} \to \ld G_{\ol\F_\ell}$ such that for every $\ol\pi$ as above, both $a \in \Z/2$, and every irreducible constituent $\ol\pi_H$ of $\rT^a(\sigma, \ol\pi)$, we have
    \[
    \rho^{\FS}(\ol\pi^{(\ell)}) \sim (\ld\psi \circ \rho^{\FS}(\ol\pi_H))^{\ss},
    \]
    where $\ol\pi^{(\ell)}$ denotes the $\ell$-Frobenius twist of $\ol\pi$ and the superscript $\ss$ denotes the semisimplification in $\ld G$.
\end{thm}

The importance of Theorem~\ref{thm:feng-modular-functoriality} from the point of view of the classical Local Langlands Correspondence (with characteristic $0$ coefficients) comes from the fact that, if $\pi$ is a smooth irreducible $\ol\Q_\ell$-representation of $G(F)$ which admits a $\ol\Z_\ell$-lattice $\Lambda$, then for every irreducible $\ol\F_\ell$-subquotient $\ol\pi$ of $\Lambda \otimes_{\ol\Z_\ell} \ol\F_\ell$, the L-parameter $\rho^{\FS}(\ol\pi)$ is in a precise sense the $\ell$-modular reduction of $\rho^{\FS}(\pi)$, as will be discussed in \cite[\S 9.1]{CF26b}. Thus Theorem~\ref{thm:feng-modular-functoriality} can be viewed as providing a principled way to construct congruences between L-parameters, which one may then hope to ``propagate'' to equalities in character $0$. This will be discussed further below.

This discussion suggests possible inductive methods for studying $\rho^{\FS}$. However, to apply Theorem~\ref{thm:feng-modular-functoriality} in practice for a given $\sigma$, one needs to answer two basic questions:
\begin{enumerate}
    \item What is $\ld\psi$?
    \item What is $\rT^a(\sigma, \ol\pi)$? (For instance, is it nonzero?)
\end{enumerate}
The first question is addressed in some generality in \cite[Proposition~10.2.1]{F24}, and it will be addressed in further generality in \cite{CF26b}. The main goal of this paper is to address the second question precisely enough to prove Theorem~\ref{thm:intro-debacker-reeder}. Our proof will rely on a small number of technical results concerning L-embeddings, which are proven in a self-contained manner in \cite{CF26a} and \cite{CF26b}. Only three lemmas (Lemmas~\ref{lemma:finite-groups-mod-ell}, \ref{lemma:conjugate-weyl-elements}, and \ref{lemma:torus-torsion-grows}) of \S\ref{section:debacker-reeder} will be used in the sequel papers, and since their proofs are self-contained there is no circularity.

\subsection{Tate cohomology for finite reductive groups}

We now outline the proof of Theorem~\ref{thm:intro-debacker-reeder}. A simple argument reduces one to showing that $\rho^{\FS}(\pi)$ and $\rho^{\DR}(\pi)$ have the same restrictions to the inertia subgroup $I_F \subset W_F$.

\subsubsection{Paraductive group schemes}

Tate cohomology commutes with compact inductions in a suitable sense (Lemma~\ref{lemma:tate-cohom-compact-induction}), so the key point in view of \eqref{eqn:intro-depth-0-rep} is to study the Tate cohomology of $\ol\tau$. This representation is inflated from an irreducible representation of the group $\ol G_{[x]}(\F_q)$ of $\F_q$-points of a certain smooth $\F_q$-group scheme $\ol G_{[x]}$ with reductive identity component. Notably, $\ol G_{[x]}$ is not typically connected, nor even of finite type; we call the group schemes appearing in this way \emph{paraductive} (see Definition~\ref{def:paraductive} for an actual list of conditions).

We begin in \S\ref{section:extended-dl-theory} by developing Deligne--Lusztig theory \cite{DL76} for paraductive group schemes, extending the theory for disconnected reductive groups introduced in \cite{DM94} and \cite{Kal21b}. Importantly, we develop analogues of Lusztig induction $R_{\ol L}^{\ol G}$ and Lusztig restriction $^*R^{\ol G}_{\ol L}$ for ``twisted Levi subgroups'' $\ol L \subset \ol G$, as well as ``semi-rational'' Lusztig series $\cE(\ol G, [\ol T, \theta])$ associated to ``generalized maximal tori'' $\ol T \subset \ol G$ and characters $\theta$ of $\ol T(\F_q)$. The definition of $\rho^{\DR}$ can be phrased in terms of the Lusztig series to which $\tau$ belongs: indeed, the fact that the pair $(\ol T, \theta)$ defined in \S\ref{sss:intro-dr} is associated to $\tau$ means precisely that $\tau \in \cE(\ol G, [\ol T, \theta])$. In general, one should think of $\cE(\ol G, [\ol T, \theta])$ as a (subset of a) ``semisimple inertial L-packet'', in the sense that the inertial restriction of the semisimple L-parameter attached to $\pi$ should depend only on $\cE(\ol G, [\ol T, \theta])$, even if $\pi$ is singular.

We will use the language of paraductive group schemes for the remainder of the introduction, but the reader will not lose much on a first pass by interpreting ``paraductive'' as ``connected reductive''. In particular, if $G$ is semisimple and simply connected then all paraductive group schemes which arise in the proof of Theorem~\ref{thm:intro-debacker-reeder} are actually connected reductive.

\subsubsection{Lower bounds on Tate cohomology}

If $\Gamma$ is a group and $U$ and $V$ are two semisimple finite-dimensional representations of $\Gamma$ over a field, then we write $U \leq V$ if $U$ is isomorphic to a subrepresentation of $V$. The following theorem, which will be improved in Theorem~\ref{thm:eigenvalues-and-jordan-blocks-variant}, is one of our main calculations. If $U = \sum_{i=1}^m n_i V_i$ is a virtual representation of $\Gamma$, where $n_i \in \Z$ and the $V_i$ are pairwise non-isomorphic irreducible representations of $\Gamma$, then the \emph{absolute value} $|U|$ is defined to be $\sum_{i=1}^m |n_i| V_i$.

\begin{thm}[Theorem~\ref{thm:eigenvalues-and-jordan-blocks-variant}]\label{thm:intro-tate-lower-bound}
    Let $\Gamma$ be a finite group, let $\sigma$ be an automorphism of $\Gamma$ of order $\ell$, and let $V$ be a $\ol\Z_\ell[\Gamma \rtimes \langle\sigma\rangle]$-module which is finite free as a $\ol\Z_\ell$-module. Suppose that $V \otimes_{\ol\Z_\ell} \ol\Q_\ell$ is defined over the maximal unramified extension $\Q_\ell^{\unr}$ as a $\Gamma$-representation. Then
    \[
    \rT^a(\sigma, V \otimes_{\ol\Z_\ell} \ol\F_\ell)^{\ss} \geq \ol U
    \]
    for both $a \in \Z/2$, where $\ol U$ is the semisimple representation of $\Gamma^\sigma$ which is the absolute value of the virtual representation which is the reduction modulo $\ell$ of the $\ol\Q_\ell$-representation with character $\gamma \mapsto \Tr(\gamma \rtimes \sigma|V_{\ol\Q_\ell})$.
\end{thm}

In the case $a=0$, Tate cohomology is also known in the finite group theory literature as the \emph{Brauer construction} or \emph{Brauer map}. In this connection, Brauer's second main theorem (see for instance \cite[Th\'eor\`eme 4.3]{Dig86b}) is reminiscent of Theorem~\ref{thm:intro-tate-lower-bound}, and it requires no assumption on fields of definition. The main additional content of Theorem~\ref{thm:intro-tate-lower-bound} is that it gives a concrete computational tool for checking that Tate cohomology is \emph{nonzero} (as is needed for Theorem~\ref{thm:feng-modular-functoriality} to have content).

The proof of Theorem~\ref{thm:intro-tate-lower-bound} involves using the eigenspaces of $\sigma$ on $V_{\ol\Q_\ell}$ to build many $\sigma$-stable flags in $V \otimes_{\ol\Z_\ell} \ol\F_\ell$ on whose subquotients $\sigma$ acts trivially. These are then played against each other to yield lower bounds on $V^\sigma$ and upper bounds on $N_\sigma(V)$, which imply a lower bound on $\rT^0(\sigma, V \otimes_{\ol\Z_\ell} \ol\F_\ell)$. To illustrate the strength of Theorem~\ref{thm:intro-tate-lower-bound}, we note two special cases.

\begin{cor}\label{cor:intro-tate-lower-bounds}
    Let $\ol G$ be a connected reductive $\F_q$-group scheme.
    \begin{enumerate}
        \item \emph{(Proposition~\ref{prop:tate-cohomology-shintani-descent})} Let $\sigma$ denote the $\F_q$-automorphism of the Weil restriction $\Res_{\F_{q^\ell}/\F_q}(\ol G_{\F_{q^\ell}})$ induced by a generator of $\Gal(\F_{q^\ell}/\F_q)$. If $V_\ell$ is a $\ol\Z_\ell[\ol G(\F_{q^\ell}) \rtimes \langle\sigma\rangle]$-module which is finite free as a $\ol\Z_\ell$-module, and $V_\ell \otimes_{\ol\Z_\ell} \ol\Q_\ell$ is defined over $\Q_\ell^{\unr}$, then
        \[
            \rT^a(\sigma, V_\ell \otimes_{\ol\Z_\ell} \ol\F_\ell)^{\ss} \geq \ol V,
        \]
        where $\ol V$ is the $\ell$-Frobenius twist of the absolute value of the $\ell$-modular reduction of the Shintani descent\footnote{An important technical point is that Shintani descent requires as input a convention for a ``norm map'', and this convention is not completely standard in the literature. We discuss this in Remark~\ref{remark:glauberman-shintani} (following \cite{Dig86b}), and we note that Langlands functoriality (in the form of this corollary and Theorem~\ref{thm:feng-modular-functoriality}) suggests the ``correct'' choice of norm. It turns out that this matches the choice in \cite{Kaw87}.} of $V_\ell \otimes_{\ol\Z_\ell} \ol\Q_\ell$.
        \item \emph{(Proposition~\ref{prop:tate-cohom-dl-restriction})} Let $t \in \ol G(\F_q)$ be an element of order $\ell$ such that $\ol H = Z_{\ol G}(t)$ is a twisted Levi $\F_q$-subgroup of $\ol G$, and let $\sigma$ denote the automorphism of $\ol G$ induced by $t$-conjugation. If $V$ is a $\ol\Z_\ell[\ol G(\F_q)]$-module which is finite free as a $\ol\Z_\ell$-module, and $V \otimes_{\ol\Z_\ell} \ol\Q_\ell$ is defined over $\Q_\ell^{\unr}$ and lies in a prime-to-$\ell$ Lusztig series, then
        \[
            \rT^a(\sigma, V \otimes_{\ol\Z_\ell} \ol\F_\ell)^{\ss} \geq \left|\ol{{}^*R^{\ol G}_{\ol H}(V)}\right|
        \]
        for both $a \in \Z/2$. If $\ell$ is good for $\ol G^\circ$, then the $\ell$-modular reduction $\ol{{}^*R^{\ol G}_{\ol H}(V)}$ of $^*R^{\ol G}_{\ol H}(V)$ is nonzero.
    \end{enumerate}
\end{cor}

In the case $\ol G = \GL_n$ and $(\ell p, n) = 1$, a sharper version of Corollary~\ref{cor:intro-tate-lower-bounds}(1) can be found in \cite[Theorem~13]{Ron16}. A technical refinement of Corollary~\ref{cor:intro-tate-lower-bounds}(2) can be found in Proposition~\ref{prop:tate-cohom-dl-restriction}. For our purposes, the content of Corollary~\ref{cor:intro-tate-lower-bounds}(2) is twofold:
\begin{enumerate}[label=(\alph*)]
    \item By a variation (Corollary~\ref{cor:bonnafe-11.11}) on a theorem of Lusztig \cite[Corollaire~11.11]{Bon06}, it implies that Tate cohomology respects Lusztig series in a certain sense.
    \item It implies that Tate cohomology is nonzero in this case.
\end{enumerate}
Note that when $\ol G$ is connected, (a) is implied by the remarkable \cite[Th\'eor\`eme 3.2]{BM89} with no assumptions on fields of definition. For applications to non-singular cuspidal representations, point (b) also follows from \cite[Th\'eor\`eme 3.2]{BM89}. However, for applications to singular cuspidal representations in \cite{CF26b} (and for the potential generalization of Theorem~\ref{thm:intro-debacker-reeder} described in Remark~\ref{rmk:singular-generalization}), point (b) is crucial.

Shintani descent is usually regarded as realizing ``base change functoriality for finite reductive groups''. In view of Theorem~\ref{thm:feng-modular-functoriality}, Corollary~\ref{cor:intro-tate-lower-bounds} can be regarded as a shadow of this statement, as well as the statement that Lusztig restriction realizes ``twisted Levi functoriality for finite reductive groups''. There are three important reasons we use the word ``shadow'' here:
\begin{enumerate}[label=(\Alph*)]
    \item Theorem~\ref{thm:feng-modular-functoriality} only refers to $\ol\F_\ell$-representations,
    \item Both Corollary~\ref{cor:intro-tate-lower-bounds}(1) and (2) require the $\ol\F_\ell$-representations to admit models over $\Z_\ell^{\unr}$,
    \item For a given twisted Levi $\F_q$-subgroup $\ol H$ there is typically no element $t$ as in (2).
\end{enumerate}
For Shintani descent, we do not see a way to deal with these issues in general, but for Lusztig restriction it can be done, as we now explain.

\subsubsection{Base change of large prime degree}

If $\pi$ is a $\ol\Q_\ell$-representation of $G(F)$, then because $\rho^{\FS}(\pi)|_{I_F}$ has finite image, Lemma~\ref{lemma:finite-groups-mod-ell} shows that one can compute it by computing it modulo $\ell$ if $\ell$ is ``large'' (e.g., larger than the order of the image of $\rho^{\FS}(\pi)|_{I_F}$). We aim to perform this calculation by induction on the semisimple rank of $G$, the base case that $G$ is a torus following from the fact that $\rho^{\FS}$ and $\rho^{\DR}$ are both compatible with the Local Langlands Correspondence for tori.

The induction step is principally based on Corollary~\ref{cor:intro-tate-lower-bounds}(2). To pass to a situation to which this applies, the idea is to first pass to an unramified extension of $F$ of ``large'' prime degree so that $G(F)$ has torsion elements $t$ of large prime degree, and then to use such $t$ in Corollary~\ref{cor:intro-tate-lower-bounds}. As mentioned above, issues (A) and (B) are serious when dealing with general base change; however, in a special case, they both disappear.

Recall the \emph{Glauberman correspondence} from \cite{Gla68}: in a special case, this states that if $\Gamma$ is a finite group of order prime to $\ell$ and $\sigma$ is an automorphism of $\Gamma$ of order $\ell$, then there exists a natural bijection
\[
\Irr_{\ol\Q_\ell}(\Gamma)^\sigma \cong \Irr_{\ol\Q_\ell}(\Gamma^\sigma)
\]
which is characterized by a character identity \eqref{eqn:glauberman-condition}. Since $\ell$ does not divide the order of $\Gamma$, every $\ol\Q_\ell$-representation of $\Gamma$ is defined over $\Q_\ell^{\unr}$, and there are canonical bijections
\[
\Irr_{\ol\Q_\ell}(\Gamma)^\sigma \cong \Irr_{\ol\F_\ell}(\Gamma)^\sigma \qquad \text{and} \qquad \Irr_{\ol\Q_\ell}(\Gamma^\sigma) \cong \Irr_{\ol\F_\ell}(\Gamma^\sigma)
\]
by \cite[Part~III, no.\ 15.5, Proposition~43]{Serre77}.

In fact, the Glauberman correspondence can be extended slightly to a certain (very restricted) class of infinite groups, such as $\ol G(\F_{q^\ell})$, as can Theorem~\ref{thm:intro-glauberman}; see \S\ref{sss:infinite-glauberman} and especially Hypothesis~\ref{hypothesis:infinite-glauberman}. In Remark~\ref{remark:glauberman-shintani}, we will explain (following \cite{Dig86b}) that if $\Gamma = \ol G(\F_{q^\ell})$ as in the setting of Shintani descent, where $\ol G$ is connected, then the Glauberman correspondence realizes the $\ell$-Frobenius twist of the $\ell$-modular reduction of Shintani descent. Thus issues (A) and (B) above do not appear for large $\ell$.

The following theorem is a sharper version of Corollary~\ref{cor:intro-tate-lower-bounds}(1) in this case.

\begin{thm}[Alperin \cite{Alp76}, Dade \cite{Dade78}, Corollary~\ref{cor:glauberman-correspondence}]\label{thm:intro-glauberman}
    Let $\Gamma$ be a finite group, and let $\sigma$ be an automorphism of $\Gamma$ of order $\ell$ prime to $|\Gamma|$. Then $\rT^a(\sigma, -)$ induces the Glauberman correspondence for each $a \in \Z/2$.
\end{thm}

Theorem~\ref{thm:intro-glauberman} is not new as stated; it also appears in \cite{Alp76} and the last sentence of \cite{Dade78}. In practice, we need a slightly sharper version, which incorporates some cases in which $\ell$ divides $|\Gamma|$ and has a slightly more precise conclusion. This somewhat technical strengthening appears as Theorem~\ref{theorem:eigenvalues-and-jordan-blocks}, and it is proven independently of previous results in the literature. Our proof is similar in spirit to the proof of Theorem~\ref{thm:intro-tate-lower-bound}.


Partially using Theorem~\ref{thm:intro-glauberman}, we establish the following further results.

\begin{cor}\label{cor:intro-glauberman-properties}
    Let $\ol G$ be a paraductive $\F_q$-group, let $\ell$ be a prime not dividing $|\ol G(\F_q)|$, let $\ol\tau$ be an irreducible $\ol\F_\ell$-representation of $\ol G(\F_q)$, and let $\wt\tau$ be the irreducible $\ol\F_\ell$-representation of $\ol G(\F_{q^\ell})$ corresponding to $\ol\tau$ under the Glauberman correspondence.
    \begin{enumerate}
        \item (Proposition~\ref{prop:glauberman-cuspidal}) $\ol\tau$ is cuspidal if and only if $\wt\tau$ is cuspidal.
        \item (Proposition~\ref{prop:digne-3.5}) If $\ol T \subset \ol G$ is a generalized maximal torus and $\theta\co \ol T(\F_q) \to \ol\F_\ell^\times$ is a character such that $\ol\tau$ lies in the Lusztig series $\cE(\ol G, [\ol T, \theta])$, then $\wt\tau$ lies in the Lusztig series $\cE(\ol G_{\F_{q^\ell}}, [\ol T_{\F_{q^\ell}}, \theta_\ell])$, where $\theta_\ell\co \ol T(\F_{q^\ell}) \to \ol\F_\ell^\times$ is the unique $\Gal(\F_{q^\ell}/\F_q)$-stable character extending $\theta$.
        \item (Corollary~\ref{cor:glauberman-lusztig-bijection}) If $\ell$ is large enough, then the induced map $\cE(\ol G, [\ol T, \theta]) \to \cE(\ol G_{\F_{q^\ell}}, [\ol T_{\F_{q^\ell}}, \theta_\ell])$ is a bijection.
    \end{enumerate}
\end{cor}

One direction of Corollary~\ref{cor:intro-glauberman-properties}(1), namely the fact that Tate cohomology sends cuspidal $\ol\F_\ell$-representations to cuspidal $\ol\F_\ell$-representations, was proven in \cite[Corollary~3.3.3]{DN25a}; the converse is a special feature of the Glauberman correspondence. When $\ol G$ is connected and $p$ is \emph{good} for $\ol G$, Corollary~\ref{cor:intro-glauberman-properties}(2) was proven in the more general setting of Shintani descent in \cite[Corollaire~3.5]{Dig99}, under the hypothesis that Deligne--Lusztig induction is independent of the choice of parabolic (now known by \cite[Theorem~8.7.2]{GRV26}). Note that Corollary~\ref{cor:intro-glauberman-properties}(3) shows that base change along an unramified extension of degree $\ell$ is ``harmless''; this will be important in \cite{CF26b} but is not necessary for Theorem~\ref{thm:intro-debacker-reeder}. In the case that $\ol G$ is connected, Corollary~\ref{cor:intro-glauberman-properties}(3) was established in \cite[Lemma~3.3]{Cot26a}, and the proof of Corollary~\ref{cor:intro-glauberman-properties}(3) is an amplification of that proof.

The identity \eqref{eqn:glauberman-condition} characterizing the Glauberman correspondence shows that it preserves the field of definition of a character, and since $\Q_\ell^{\unr}$ has trivial Brauer group, it preserves the field of definition of a representation. Thus one can apply the Glauberman correspondence to pass to a setting in which issues (A), (B), and (C) do not intervene in the ``twisted Levi functoriality'' situation (see Lemmas~\ref{lemma:integer-valued-characters} and \ref{lemma:character-values-irreducible-factors}). Thus we first use ``independence of $\ell$'' \cite[Theorem~1.1]{Sch25} and Theorem~\ref{thm:intro-glauberman} to pass to an unramified extension of $F$ of large prime degree $\ell_1$ so that $T(F)$ has a torsion element of large prime order $\ell_2$, and then we apply Corollary~\ref{cor:intro-tate-lower-bounds}(2) modulo $\ell_2$ to perform the induction step described above and thereby prove Theorem~\ref{thm:intro-debacker-reeder}.

\subsection{Outline of the paper}

In \S\ref{section:extended-dl-theory}, we develop the theory of paraductive $\F_q$-group schemes and extend Deligne--Lusztig theory to such groups. In \S\ref{section:rosetta-stone}, we provide various general methods for computing (or at least providing lower bounds for) Tate cohomology, and we use these methods to partially calculate Tate cohomology in various settings of interest. In \S\ref{sec:large-prime-degree-base-change}, we use our Tate cohomology calculations to establish the properties of the Glauberman correspondence described above. Finally, in \S\ref{section:debacker-reeder}, we recall the definition of $\rho^{\DR}$ and prove a few basic results about it, use most of the preceding theory to prove results about $\rho^{\FS}$, and finally prove Theorem~\ref{thm:intro-debacker-reeder}.

\subsection{Notation and conventions}

The symbol $q$ will always denote a power of the prime number $p$, and $\ell$ will always denote a prime number distinct from $p$.

If $H$ is a group scheme over a field $k$, then $H^\circ$ denotes the identity component of $H$. If $M \subset H$ is a closed subscheme, then $N_H(M)$ (resp.\ $Z_H(M)$) denotes the functorial normalizer (resp.\ centralizer) of $M$ in $H$, which will be representable by a closed $k$-subgroup scheme of $H$ in all situations in which it appears. If $H$ is smooth and connected, then $H_{\der}$ is the derived group of $H$ (in the sense of algebraic groups).

\subsection{Acknowledgements}

I thank Jeff Adler, Ad\`ele Bourgeois, Charlotte Chan, Stephen DeBacker, Tony Feng, Jessica Fintzen, Alex Hazeltine, Josh Lansky, Santosh Nadimpalli, Monica Nevins, David Schwein, Jack Sempliner, Loren Spice, and Jay Taylor for helpful conversations. I especially thank Tony Feng for suggesting various edits, as well as permission to include the proof of Theorem~\ref{thm:intro-debacker-reeder} in \S\ref{section:debacker-reeder}, which was developed as a variant of our joint work. The author did not employ AI tools in the preparation of this paper.
I acknowledge support from the National Science Foundation under Award No.\ 2402231 and the European Research Council (ERC) under the
European Union’s Horizon 2020 research and innovation programme (grant agreement no.\ 950326).

\section{Deligne--Lusztig theory for paraductive groups}\label{section:extended-dl-theory}

In this section, we generalize elements of Deligne--Lusztig theory to disconnected groups, extending the work done in \cite[\S 2]{Kal21b} and \cite{DM94}. The motivation for this generalization comes from the theory of $p$-adic reductive groups, where disconnected groups arise naturally as the special fibers of integral models, as we will explain below. Starting in \S2.3, the symbol $\ol G$ will be used to denote a paraductive $\F_q$-group scheme.

\subsection{Paraductive groups}\label{ss:main-paraductive-example}

The following definition describes one of the main classes of objects of interest to this paper.\footnote{The name is a shortened version of \emph{parareductive}, intended to sound similar to the Bruhat--Tits-sanctioned shortening \emph{parahoric}. It is meant to emphasize that these groups are closely related to reductive groups, but that they go ``beyond'' reductivity.}

\begin{defn}\label{def:paraductive}
    Let $k$ be a field. A smooth $k$-group scheme $\ol G$ is \textit{paraductive} if
    \begin{enumerate}
        \item $\ol G^\circ$ is reductive,
        \item $\pi_0(\ol G)(\ol k)$ is a finitely generated abelian group,
        \item $\ol G^\circ \cdot Z(\ol G)$ is of finite index in $\ol G$,
        \item if $\ol T^\circ \subset \ol G^\circ$ is a maximal $k$-torus, then $Z_{\ol G}(\ol T^\circ) = \ol T^\circ \cdot Z(\ol G)$.
    \end{enumerate}
\end{defn}

\begin{example}
    Paraductive group schemes form a slightly smaller class of groups than the one described by \cite[Assumption 2.1.1]{Kal21b}, which includes only assumptions (1), (2), and (3) of Definition~\ref{def:paraductive}. For example, let $x \in \F_q^\times$ and let $\ol G$ be the smooth $\F_q$-group scheme with underlying scheme $\G_m \times \underline{\Z}^2$ (where $\underline{\Z}^2$ is the constant $\F_q$-scheme corresponding to $\Z^2$) and multiplication
    \[
    (g_1, (m_1, n_1)) \cdot (g_2, (m_2, n_2)) = (ghx^{m_1n_2}, (m_1 + n_1, m_2 + n_2)).
    \]
    Then $\ol G$ satisfies \cite[Assumption 2.1.1]{Kal21b}, but it is not paraductive: indeed, $\ol G^\circ$ is a torus but $\ol G$ is not commutative, so it does not satisfy assumption (4) of Definition~\ref{def:paraductive}. In some sense, assumption (4) is designed to eliminate such central extensions from consideration.
\end{example}

We are interested in generalizing Deligne--Lusztig theory to paraductive $\F_q$-group schemes.\footnote{The reason for assumption (4) above is that it makes various definitions and results about Lusztig series cleaner, and, as we will see, it holds in all situations of interest.} Before doing so, we explain the main example. Although we are primarily interested in the case $k = \F_q$, it will occasionally be useful to allow $k = \ol\F_q$.

\begin{lemma}\label{lemma:surj-centralizers}
    Let $k$ be a field, and let $f\co H' \to H$ be a surjective $k$-homomorphism of smooth $k$-group schemes whose kernel is a unipotent $k$-group scheme. Let $S' \subset H'$ be a $k$-torus, and let $S = f(S')$. The map $f_0\co Z_{H'}(S') \to Z_H(S)$ is surjective.
\end{lemma}

\begin{proof}
    We may and do assume $k = \ol k$, so it suffices to show that $f_0$ is surjective on $k$-points. Let $h \in Z_H(S)(k)$ and choose $h'_0 \in H'(k)$ such that $f(h'_0) = h$. Since $\ker f$ is unipotent, the map $f|_{S'}\co S' \to S$ is an isomorphism. Note that $S'$ and $h'_0S'(h'_0)^{-1}$ are both maximal tori of $f^{-1}(S)$, so there exists $v \in f^{-1}(S)(k)$ such that
    \[
    v\bigl(h'_0S'(h'_0)^{-1}\bigr)v^{-1}=S'.
    \]
    Thus $vh'_0$ normalizes $S'$. The induced automorphism of $S'$ maps under $f|_{S'}$ to conjugation by $f(v)h$ on $S$, which is trivial because $f(v) \in S(k)$ and $h \in Z_H(S)(k)$. Hence $vh'_0$ centralizes $S'$. Choose $s' \in S'(k)$ such that $f(s')=f(v)$. $(s')^{-1}vh'_0 \in Z_{H'}(S')(k)$ maps to $h$ under $f_0$, as desired.
\end{proof}

Let $F$ be a non-archimedean local field with residue field $\F_q$, and let $G$ be a connected reductive $F$-group. Let $x$ be a point of the enlarged Bruhat--Tits building $\cB(G)$, and let $[x]$ denote its image in $\cB(G_{\der})$. According to Bruhat--Tits theory \cite[Remark 8.3.4]{KP}, there exists a canonical smooth separated $\cO_F$-group scheme $\cG_{[x]}$ with generic fiber $G$ and which satisfies $\cG_{[x]}(\cO_F) = G(F)_{[x]}$. Let $\ol\cG_{[x]}$ denote the special fiber of $\cG_{[x]}$, and let $\ol G_{[x]}$ denote the quotient of $\ol\cG_{[x]}$ by the unipotent radical of $\ol\cG_{[x]}^\circ$.

\begin{prop}\label{prop:main-paraductive-example}
    The $\F_q$-group scheme $\ol G_{[x]}$ is paraductive.
\end{prop}

\begin{proof}
    Conditions (1)-(3) in Definition~\ref{def:paraductive} have been verified in \cite[\S 3.2]{Kal21b}; thus it suffices to prove condition (4). It is clear that if $\ol S^\circ \subset \ol G_{[x]}$ is a maximal $\F_q$-torus, then $Z_{\ol G_{[x]}}(\ol S^\circ)$ contains $\ol S^\circ \cdot Z(\ol G_{[x]})$, so we need only show the reverse containment.

    Let $S \subset G$ be a maximally split maximal unramified maximal $F$-torus such that $x$ lies in the apartment $\cA(S)$. By \cite[Axiom 4.1.20]{KP}, if $\cS$ is the $\cO_F$-torus with generic fiber $S$, then there is a natural monic $\cO_F$-homomorphism $\cS \to \cG_{[x]}$ such that the special fiber $\ol\cS$ is a maximal $\F_q$-torus of $\ol\cG_{[x]}$. Let $\ol S^\circ$ denote the image of $\ol\cS$ in $\ol G_{[x]}$, which is a maximal $\F_q$-torus by \cite[Proposition~11.14(1)]{Bor91}. Recall that the centralizer $T = Z_G(S)$ is a maximal $F$-torus of $G$ by \cite[Remark 16.4]{KP}. If $\cT$ denotes the schematic closure of $T$ in $\cG_{[x]}$ and $\cZ$ denotes the schematic closure of $Z(G)$ in $\cG_{[x]}$, then the justification of \cite[Notation 2.1.2]{Kal21b} in \cite[\S 3.2]{Kal21b} shows that the image of $\ol\cT$ in $\ol G_{[x]}$ is equal to $\ol S^\circ \cdot \ol Z$, where $\ol Z$ is the image of $\ol\cZ$ in $\ol G_{[x]}$. But Lemma~\ref{lemma:surj-centralizers} shows that the map $\ol\cT = Z_{\ol\cG_{[x]}}(\ol\cS) \to Z_{\ol G_{[x]}}(\ol S^\circ)$ is surjective, so indeed $Z_{\ol G_{[x]}}(\ol S^\circ) \subset \ol S^\circ \cdot Z(\ol G_{[x]})$, as desired.
\end{proof}

\subsection{Twisted Levis for paraductive groups}

Recall that if $k$ is a field, then a closed $k$-subgroup scheme $\ol L^\circ \subset \ol G^\circ$ of a connected reductive $k$-group $\ol G^\circ$ is called a \textit{twisted Levi subgroup} if $\ol L^\circ_{\ol k}$ is a Levi factor of a parabolic $\ol k$-subgroup of $\ol G^\circ_{\ol k}$. We need a definition of twisted Levi subgroups of disconnected reductive groups. 

\begin{defn}\label{def:twisted-levi}
    If $\ol G$ is a paraductive group scheme over a field $k$ and $\ol S^\circ \subset \ol G^\circ$ is a $k$-torus, then $\ol H\coloneqq Z_{\ol G}(\ol S^\circ)$ is called a \textit{twisted Levi subgroup} of $\ol G$. If $\ol S^\circ$ is a maximal torus of $\ol G^\circ$, then we will call $\ol H$ a \textit{generalized maximal torus}.
\end{defn}

\begin{example}
    In the setting of Proposition~\ref{prop:main-paraductive-example}, if $L \subset G$ is a twisted Levi $F$-subgroup such that $x \in \cB(L)$ (under any choice of embedding $\cB(L) \subset \cB(G)$), the group $\ol L_{[x]}$ (where $[x]$ is still the image of $x$ in $\cB(G_{\der})$, not $\cB(L_{\der})$) is naturally a twisted Levi $\F_q$-subgroup of $\ol G_{[x]}$: this follows from Lemma~\ref{lemma:surj-centralizers}.
    
    If $G = \PGL_2$ and $x$ is the midpoint of an alcove in $\cB(\PGL_2)$, then $\ol G_{[x]} \cong \G_m \rtimes \Z/2$, where $\Z/2$ acts on $\G_m$ nontrivially. Note that $\ol G_{[x]}$ is a twisted Levi subgroup of itself, but it is not a generalized maximal torus despite having torus identity component.
\end{example}

\begin{lemma}\label{lemma:paraductive-permanence}
    If $\ol G$ is a paraductive $k$-group scheme and $\ol L \subset \ol G$ is a twisted Levi $k$-subgroup, then $\ol L$ is paraductive.
\end{lemma}

\begin{proof}
    Let $\ol S^\circ$ be the maximal central $k$-torus of $\ol L$, so $\ol L = Z_{\ol G}(\ol S^\circ)$. Note that $Z_{\ol G^\circ}(\ol S^\circ)$ is a twisted Levi $k$-subgroup of $\ol G^\circ$, so $\ol L^\circ$ is reductive and $\pi_0(\ol L)(\ol k) \subset \pi_0(\ol G)(\ol k)$, proving assumptions (1) and (2) of Definition~\ref{def:paraductive}. Assumptions (3) and (4) are immediate.
\end{proof}

\subsection{Parabolic Deligne--Lusztig varieties}\label{ss:ns-dl-reps} 
Let $\ol L \subset \ol G$ be a twisted Levi subgroup, let $\ol P^\circ \subset \ol G^\circ_{\ol\F_q}$ be a parabolic subgroup of $\ol G^\circ_{\ol\F_q}$ with Levi factor $\ol L^\circ_{\ol\F_q}$, and let $\ol U$ denote the unipotent radical of $\ol P^\circ$. We define the associated \emph{Deligne--Lusztig variety}
\[
Y_{\ol U}^{\ol G} \coloneqq \{g \in \ol G_{\ol\F_q}\co g^{-1}\Fr_q(g) \in \ol U \cdot \Fr_q(\ol U)\}/\ol U.
\]
Note that $Y_{\ol U}^{\ol G}$ admits commuting left actions of $\ol G(\F_q)$ by left multiplication and $\ol L(\F_q)$ by inverted right multiplication, and these two actions agree on $Z(\ol G)(\F_q) \subset \ol L(\F_q)$. Previously, \cite{Kal21b} developed Deligne--Lusztig theory for $\F_q$-group schemes $\ol G$ satisfying every hypothesis in Definition~\ref{def:paraductive} except (4), in the case where $\ol L = \ol T^\circ \cdot \ol Z$ for some central $\F_q$-subgroup scheme $\ol Z \subset \ol G$ such that $\ol G^\circ \cdot \ol Z$ is of finite index in $\ol G$. Our definitions clearly agree in cases of overlap. Also, \cite{DM94} developed a form of Deligne--Lusztig theory for disconnected (finite type) reductive groups over $\F_q$ which is slightly different than ours, because they use a slightly different generalization of parabolic subgroups than we (implicitly) use. We will see soon (Lemma~\ref{lemma:agreement-with-digne-michel}) that their definition agrees with ours in some cases of overlap; the following lemma implies that it does not agree in \emph{all} cases of overlap.

\begin{lemma}\label{lemma:disc-dl-induction}
    Let $\ol G$ be a paraductive $\F_q$-group scheme, let $\ol L \subset \ol G$ be a twisted Levi subgroup, and let $i \in \Z$.
    \begin{enumerate}
        \item As a $1 \times \ol L(\F_q)$-representation, we have 
        \[
        \rH_c^i(Y_{\ol U}^{\ol G^\circ\cdot \ol L}, \ol\Q_\ell) \cong \rH_c^i(Y_{\ol U}^{\ol G^\circ}, \ol\Q_\ell) \otimes_{\ol\Q_\ell[1 \times \ol L^\circ(\F_q)]} \ol\Q_\ell[1 \times \ol L(\F_q)],
        \]
        \item As a $\ol G(\F_q) \times 1$-representation, we have 
        \[
        \rH_c^i(Y_{\ol U}^{\ol G}, \ol\Q_\ell) \cong \ind_{(\ol G^\circ \cdot \ol L)(\F_q) \times 1}^{\ol G(\F_q) \times 1} \rH_c^i(Y_{\ol U}^{\ol G^\circ \cdot \ol L}, \ol\Q_\ell).
        \]
    \end{enumerate}
    In particular, if $\rho$ is an irreducible $\ol\Q_\ell$-representation of $\ol L(\F_q)$ then $\rH_c^i(Y_{\ol U}^{\ol G}, \ol\Q_\ell) \otimes_{\ol\Q_\ell[\ol L(\F_q)]} \rho$ is a finite-dimensional $\ol\Q_\ell$-representation of $\ol G(\F_q)$.
\end{lemma}

\begin{proof}
    Both (1) and (2) are clear from the definition of $Y_{\ol U}^{\ol G}$: the key point is just to decompose the cohomology of $Y_{\ol U}^{\ol G}$ into the sum of the cohomology of its connected components and keep track of the actions. For the final claim, observe that
    \begin{align*}
    \rH_c^i(Y_{\ol U}^{\ol G}, \ol\Q_\ell) \otimes_{\ol\Q_\ell[\ol L(\F_q)]} \rho &\cong \ind_{\ol G^\circ(\F_q) \cdot \ol L(\F_q)}^{\ol G(\F_q)} \left(\rH_c^i(Y_{\ol U}^{\ol G^\circ}, \ol\Q_\ell) \otimes_{\ol\Q_\ell[\ol L^\circ(\F_q)]} \rho\right)
    \end{align*}
    by (1). Since $\ol G^\circ(\F_q) \cdot \ol L(\F_q)$ is of finite index in $\ol G(\F_q)$, the claim follows.
\end{proof}

\subsection{Lusztig induction}\label{ss:lusztig-ind}
In this section, let $\ol G$ be a paraductive $\F_q$-group scheme and let $\ol L \subset \ol G$ be a twisted Levi $\F_q$-subgroup. Note that because $\ol G^\circ(\F_q)$ is finite and $\ol G^\circ(\F_q) \cdot Z(\ol G)(\F_q)$ is of finite index in $\ol G(\F_q)$, every irreducible $\ol\Q_\ell$-representation of $\ol G(\F_q)$ is finite-dimensional. In particular, the K-group $K_0(\Rep_{\ol\Q_\ell}(\ol G(\F_q)))$ of the category of finite-dimensional $\ol\Q_\ell$-representations of $\ol G(\F_q)$ has basis consisting of the irreducible representations. We define the \emph{Lusztig induction} 
\[
R_{\ol L}^{\ol G}\co K_0(\Rep_{\ol\Q_\ell}(\ol L(\F_q))) \to K_0(\Rep_{\ol\Q_\ell}(\ol G(\F_q)))
\]
as follows: if $\rho$ is an irreducible $\ol\Q_\ell$-representation of $\ol L(\F_q)$, then we define
\[
\rH_c^i(Y_{\ol U}^{\ol G}, \ol\Q_\ell)_\rho \coloneqq \rH_c^i(Y_{\ol U}^{\ol G}, \ol\Q_\ell) \otimes_{\ol\Q_\ell[\ol L(\F_q)]} \rho
\]
and we set
\[
R_{\ol L}^{\ol G}(\rho) = \sum_{i \geq 0} (-1)^i \rH_c^i(Y_{\ol U}^{\ol G}, \ol\Q_\ell)_\rho,
\]
This alternating sum is well-defined by the final claim of Lemma~\ref{lemma:disc-dl-induction}. Moreover, Lemma~\ref{lemma:disc-dl-induction}(2) shows that
\[
R_{\ol L}^{\ol G}(\rho) \cong \ind_{(\ol G^\circ\cdot \ol L)(\F_q) \times 1}^{\ol G(\F_q) \times 1} R_{\ol L}^{\ol G^\circ \cdot \ol L}(\rho).
\]
Note that $(Z(\ol G) \cap \ol L)(\bF_q)$ acts on $R_{\ol L}^{\ol G}(\rho)$ through the same character as the one through which it acts on $\rho$. In the special case that $\ol L = \ol T$ is a generalized maximal torus of $\ol G$, we will refer to $R_{\ol T}^{\ol G}$ as \emph{Deligne--Lusztig induction}.

\begin{lemma}\label{lemma:independence-of-parabolic}
    $R_{\ol L}^{\ol G}$ is independent of the choice of $\ol U$.
\end{lemma}

\begin{proof}
    Let $\rho$ be an irreducible $\ol\Q_\ell$-representation of $\ol L(\F_q)$. By Lemma~\ref{lemma:disc-dl-induction}, we may and do assume that $\ol G = \ol G^\circ \cdot \ol L$. If $\ol Z = Z(\ol G)$, then $\ol L^\circ(\F_q) \cdot \ol Z(\F_q)$ is of finite index in $\ol L(\F_q)$. The Lusztig induction $R_{\ol L^\circ \cdot \ol Z(\F_q)}^{\ol G^\circ \cdot \ol Z(\F_q)}(\rho|_{\ol L^\circ(\F_q) \cdot \ol Z(\F_q)})$ has the same underlying $\ol G^\circ(\F_q)$-action as $R_{\ol L^\circ}^{\ol G^\circ}(\rho|_{\ol L^\circ(\F_q)})$, and its $\ol Z(\F_q)$-action is given by $\rho|_{\ol Z(\F_q)}$ (as in \cite[Remark 2.6.5]{Kal21b}). Note that $\ind_{\ol G^\circ(\F_q) \cdot \ol Z(\F_q)}^{\ol G(\F_q)}(R_{\ol L^\circ \cdot \ol Z(\F_q)}^{\ol G^\circ \cdot \ol Z(\F_q)}(\rho|_{\ol L^\circ(\F_q) \cdot \ol Z(\F_q)}))$ is a semisimple representation of $\ol G(\F_q)$ since it has a central character and $\ol G(\F_q)/\ol Z(\F_q)$ is finite. It follows that $R_{\ol L}^{\ol G}(\rho)$ is the $\rho$-isotypic component of $\ind_{\ol G^\circ(\F_q) \cdot \ol Z(\F_q)}^{\ol G(\F_q)}(R_{\ol L^\circ \cdot \ol Z(\F_q)}^{\ol G^\circ \cdot \ol Z(\F_q)}(\rho|_{\ol L^\circ(\F_q) \cdot \ol Z(\F_q)}))$, and these considerations thereby reduce us to the case that $\ol G$ is connected.
    In this case, the result follows from \cite[Theorem~8.7.2]{GRV26}.
\end{proof}

\begin{lemma}\label{lemma:agreement-with-digne-michel}
    Suppose that $\ol G$ is of finite type and $\ol G = \ol G^\circ \cdot \ol L$. Then $R_{\ol L}^{\ol G}$ agrees with the definition in \cite[D\'efinition 2.2]{DM94}.
\end{lemma}

\begin{proof}
    Let $\ol P^\circ$ be a parabolic $\F_q$-subgroup of $\ol G^\circ$ with Levi factor $\ol L^\circ$. We first claim that $\ol L$ is equal to the normalizer $N_{\ol G}(\ol P^\circ, \ol L^\circ)$. For this, let $Z_{\ol G}$ and $Z_{\ol L}$ be the maximal central tori of $\ol G$ and $\ol L$, respectively. We may pass from $\ol G$ to $\ol G/Z_{\ol G}$ to assume that $\ol G^\circ$ is semisimple. In this case, there is an $\F_q$-cocharacter $\lambda\co \G_m \to Z_{\ol L}$ such that $\ol P^\circ = P_{\ol G^\circ}(\lambda)$, with notation as in the dynamic method. By definition, the torus $Z_{\ol L}$ is central in $\ol L$, so it follows that $\ol L \subset N_{\ol G}(\ol P^\circ, \ol L^\circ)$. On the other hand, $N_{\ol G^\circ}(\ol P^\circ, \ol L^\circ)$ is equal to $\ol L^\circ$, so the fact $\ol G = \ol G^\circ \cdot \ol L$ shows that the map $\pi_0(\ol L) \to \pi_0(N_{\ol G}(\ol P^\circ, \ol L^\circ))$ is surjective and thus $\ol L = N_{\ol G}(\ol P^\circ, \ol L^\circ)$. But now $\ol P \coloneqq \ol P^\circ \cdot \ol L$ is a ``parabolic'' of $\ol G$ with ``Levi'' $\ol L$ in the sense of \cite[D\'efinition 1.4]{DM94}, so if $\ol U$ is the unipotent radical of $\ol P^\circ$ then the variety $Y_{\ol U}^{\ol G}$ defined above is the quotient of the variety $Y_{\ol U}$ of \cite[D\'efinition 2.1]{DM94} by $\ol U$. Since $\ol U$ is scheme-theoretically isomorphic to affine space, whose compactly supported \'etale cohomology is concentrated in top degree, the claim follows from the definitions (see for instance \cite[Proposition 10.12]{DM20}).
\end{proof}

\subsection{Lusztig restriction}\label{ss:lusztig-restriction}

If $\ol G$ is a paraductive $\F_q$-group scheme, then there is a canonical bilinear form on $K_0(\Rep_{\ol\Q_\ell}(\ol G(\F_q)))$ defined as follows: if $\chi = \sum_{i=1}^n a_i\chi_i$ and $\eta = \sum_{i=1}^n b_i \chi_i$ for irreducible characters $\chi_i$ of $\ol G(\F_q)$, then we define
\[
\langle \chi, \eta\rangle_{\ol G(\F_q)} \coloneqq \sum_{i=1}^n a_ib_i.
\]

\begin{lemma}\label{lemma:dl-restriction}
    Let $\ol G$ be a paraductive $\F_q$-group scheme, and let $\ol L \subset \ol G$ be a twisted Levi subgroup. There exists a homomorphism
    \[
    ^*R^{\ol G}_{\ol L}\co K_0(\Rep_{\ol\Q_\ell}(\ol G(\F_q))) \to K_0(\Rep_{\ol\Q_\ell}(\ol L(\F_q))),
    \]
    which we will call \textit{Lusztig restriction}, uniquely characterized by the property that
    \begin{equation}\label{eqn:dl-restriction-adjoint}
    \langle R_{\ol L}^{\ol G}(\rho), \chi\rangle_{\ol G(\F_q)} = \langle \rho, {}^*R^{\ol G}_{\ol L}(\chi)\rangle_{\ol L(\F_q)}
    \end{equation}
    for all irreducible characters $\chi$ of $\ol G(\F_q)$ and $\rho$ of $\ol L(\F_q)$.
\end{lemma}

\begin{proof}
    It is clear that ${}^*R^{\ol G}_{\ol L}$ is uniquely characterized by \eqref{eqn:dl-restriction-adjoint}, so it suffices to show that this equation makes sense, i.e., that for a given irreducible character $\chi$ of $\ol G(\F_q)$ there are only finitely many irreducible characters $\rho$ of $\ol L(\F_q)$ such that $\langle R_{\ol L}^{\ol G}(\rho), \chi\rangle_{\ol G(\F_q)} \neq 0$. Recalling that $\ol G^\circ \cdot Z(\ol G)$ is of finite index in $\ol G$ by definition, we may twist $\chi$ by a character of $\ol G(\F_q)/\ol G^\circ(\F_q)$ to reduce to the case that $\chi$ has finite order central character $\eta$. Note that if $\langle R_{\ol L}^{\ol G}(\rho), \chi\rangle_{\ol G(\F_q)} \neq 0$ then $\eta$ and $\rho$ restrict to the same character of $Z(\ol G)(\F_q)$.
    
    We claim that there exists a constant closed $\F_q$-subgroup scheme $\ol Z_0 \subset Z(\ol G)$ such that $\ol Z_0(\F_q)$ is torsion-free and $\ol G^\circ(\F_q) \cdot \ol Z_0(\F_q)$ is of finite index in $\ol G(\F_q)$. By Lemma~\ref{lemma:paraductive-permanence}, the group $\ol L^\circ \cdot Z(\ol G)$ is of finite index in $\ol L$; this implies that $\ol L^\circ(\F_q) \cdot Z(\ol G)(\F_q)$ is of finite index in $\ol L(\ol\F_q)$. Since $Z(\ol G)(\F_q)$ is a finitely generated abelian group by hypothesis, we may take $\ol Z_0$ to be a constant $\F_q$-group scheme whose $\F_q$-points are a torsion-free finite index subgroup of $Z(\ol G)(\F_q)$. By passing to a further finite index subgroup, we may also arrange that $\eta|_{\ol Z_0(\F_q)} = 1$. Note that $\rH^1(\F_q, \ol Z_0) = 0$ since $\ol Z_0$ is constant and torsion-free, so the maps $\ol G(\F_q) \to (\ol G/\ol Z_0)(\F_q)$ and $\ol L(\F_q) \to (\ol L/\ol Z_0)(\F_q)$ are surjective. Therefore we may pass from $(\ol G, \ol L)$ to $(\ol G/\ol Z_0, \ol L/\ol Z_0)$ to assume that $\ol G$ is of finite type, in which case the result is obvious.
\end{proof}

In the special case that $\ol L = \ol T$ is a generalized maximal torus of $\ol G$, we will refer to $^*R^{\ol G}_{\ol T}$ as \emph{Deligne--Lusztig restriction}.

The following result is an analogue of a classical result for connected reductive groups; see for example \cite[Lemma 13.3]{TT20} and the references given there. We remark that this is a minor extension of \cite[Corollaire~2.9]{DM94}, which is another disconnected version of this result.

\begin{lemma}\label{lemma:dl-restriction-character}
    Let $\ol G$ be a paraductive $\F_q$-group scheme, and suppose that $\ol G$ is of finite type. Let $\ol L \subset \ol G$ be a twisted Levi subgroup, let $\chi$ be a character of $\ol G(\F_q)$, and let $g \in \ol L(\F_q)$. Suppose that if $g = tu$ is the Jordan decomposition of $g$, then $Z_{\ol G^\circ}(t) \subset \ol L$. Then
    \[
    \chi(g) = ({}^*R^{\ol G}_{\ol L}\chi)(g)
    \]
\end{lemma}

\begin{proof}
    Since ${}^*R^{\ol G}_{\ol L} = {}^*R^{\ol G^\circ \cdot \ol L}_{\ol L} \circ \Res^{\ol G(\F_q)}_{\ol G^\circ(\F_q) \cdot \ol L(\F_q)}$ (this is the adjoint of Lemma~\ref{lemma:disc-dl-induction}(2)), we may assume that $\ol G = \ol G^\circ \cdot \ol L$. In this case, the result follows immediately from \cite[Corollaire~2.9]{DM94} (and Lemma~\ref{lemma:agreement-with-digne-michel}).
\end{proof}

\subsection{Exhaustion}\label{sss:exhaustion}

Our present goal is to prove an analogue of \cite[Corollary~7.7]{DL76} for paraductive $\F_q$-group schemes, i.e., that every irreducible representation of $\ol G(\F_q)$ occurs in some $R_{\ol T}^{\ol G}(\theta)$.

\begin{lemma}\label{lemma:dl-7.5}
    Let $\ol G$ be a finite type paraductive $\F_q$-group scheme. The character of the regular representation of $\ol G(\F_q)$ over $\ol\Q_\ell$ is a $\Q$-linear combination of Deligne--Lusztig inductions $R_{\ol T}^{\ol G}(\theta)$ for generalized tori $\ol T \subset \ol G$ and characters $\theta\co \ol T(\F_q) \to \ol\Q_\ell^\times$.
\end{lemma}

\begin{proof}
    Let $\ol Z = Z(\ol G)$. Let $\eta_{\ol G}$ denote the class of the regular representation of $\ol G(\F_q)$ in the Grothendieck group over $\ol\Q_\ell$. Note that $\eta_{\ol G} = \ind_{(\ol G^\circ \cdot \ol Z)(\F_q)}^{\ol G(\F_q)} \eta_{\ol G^\circ \cdot \ol Z}$, so using Lemma~\ref{lemma:disc-dl-induction} we immediately reduce to the case that $\ol G = \ol G^\circ \cdot \ol Z$.
    
    Next we reduce to the case that $\ol Z \cap \ol G^\circ$ is connected. Note that $\ol Z \cap \ol G^\circ$ is of multiplicative type, so there is an $\F_q$-torus $\wt Z^\circ$ and an embedding $\ol Z \cap \ol G^\circ \subset \wt Z^\circ$. Let $\wt G = \ol G \times^{\ol Z \cap \ol G^\circ} \wt Z^\circ$, so the center of $\wt G$ is equal to $\wt Z \coloneqq \ol Z \cdot \wt Z^\circ$ and $\wt Z \cap \wt G^\circ = \wt Z^\circ$ is a torus. Moreover, we have $\wt G = \wt G^\circ \cdot \wt Z$. Note that $\eta_{\wt G}|_{\ol G(\F_q)} = [\wt G(\F_q)\co \ol G(\F_q)] \eta_{\ol G}$. If $\ol T^\circ \subset \ol G^\circ$ is a maximal $\F_q$-torus and $\ol T = \ol T^\circ \cdot \ol Z$ and $\wt T = \ol T^\circ \cdot \wt Z$, then by \cite[Remark 2.6.5]{Kal21b}, if $\wt\theta\co \wt T(\F_q) \to \ol\Q_\ell^\times$ is a character such that $\wt\theta|_{\ol T(\F_q)} = \theta$, then the restriction of $R_{\wt T}^{\wt G}(\wt\theta)$ to $\ol G(\F_q)$ is equal to $R_{\ol T}^{\ol G}(\theta)$. Thus we may pass from $\ol G$ to $\wt G$ to assume that $\ol Z \cap \ol G^\circ$ is connected.
    
    Since $\ol Z \cap \ol G^\circ$ is connected, Lang's theorem implies that $(\ol G^\circ \cdot \ol Z)(\F_q) = \ol G^\circ(\F_q) \cdot \ol Z(\F_q)$. By \cite[Proposition~7.5]{DL76} (taking $s = 1$), we have
    \[
    \eta_{\ol G^\circ} = \sum_{i=1}^n a_i R_{\ol T_i^\circ}^{\ol G^\circ}(\theta_i^\circ)
    \]
    for some $a_i \in \Q$ and pairs $(\ol T_i^\circ, \theta_i^\circ)$ as usual. But now
    \[
    \ind_{\ol G^\circ(\F_q)}^{\ol G(\F_q)}R_{\ol T_i^\circ}^{\ol G^\circ}(\theta_i^\circ) = \sum_{\theta_i|_{\ol T^\circ(\F_q)} = \theta_i^\circ} R_{\ol T_i}^{\ol G}(\theta_i)
    \]
    by \cite[Remark 2.6.5]{Kal21b}, and $\eta_{\ol G} = \ind_{\ol G^\circ(\F_q)}^{\ol G(\F_q)} \eta_{\ol G^\circ}$, so we conclude.
\end{proof}

\begin{prop}\label{prop:dl-exhaustion}
    Let $k$ be a field which is either $\ol\Q_\ell$ or $\ol\F_\ell$, and let $V$ be a nonzero finite-dimensional $k$-representation of $\ol G(\F_q)$ on which $Z(\ol G)(\F_q)$ acts through a character (e.g., an irreducible representation). If $k = \ol\Q_\ell$ (resp.\ $k = \ol\F_\ell$), then there exists a generalized maximal $\F_q$-torus $\ol T \subset \ol G$ and a character $\theta\co \ol T(\F_q) \to \ol\Q_\ell^\times$ (resp.\ $\theta\co \ol T(\F_q) \to \ol\Z_\ell^\times$) such that some irreducible subquotient of $V$ is also an irreducible constituent of $R_{\ol T}^{\ol G}(\theta)$ (resp.\ the $\ell$-modular reduction of $R_{\ol T}^{\ol G}(\theta)$).
\end{prop}

\begin{proof}
    First, note that since $R_{\ol T}^{\ol G}(\theta) = \ind_{\ol G^\circ(\F_q) \ol T(\F_q)}^{\ol G(\F_q)} R_{\ol T}^{\ol G^\circ \cdot \ol T}(\theta)$ by \cite[Corollary~2.6.2]{Kal21b}, Frobenius reciprocity reduces one to the case that $\ol G = \ol G^\circ \cdot \ol Z$, an assumption we now make.
    
    By twisting $V$ by a character of $\ol T(\F_q)/\ol T^\circ(\F_q)$, we may assume that $V$ has finite order central character. As in the proof of Lemma~\ref{lemma:dl-restriction}, we may pass to a central quotient of $\ol G$ to assume that $\ol G$ is of finite type. We are now in the setting of Lemma~\ref{lemma:dl-7.5}, which immediately implies the result.
\end{proof}

\subsection{Pairings}

This section serves a technical purpose for \cite{CF26b}; it will not be used in this paper. The following lemma is a minor extension of \cite[Theorem~6.8]{DL76} to the setting of paraductive groups. 

\begin{lemma}\label{lemma:dl-6.8}
    Let $\ol G$ be a paraductive $\F_q$-group scheme such that $Z(\ol G) \cap \ol G^\circ$ is connected, and let $\ol T, \ol T'\subset \ol G$ be generalized maximal $\F_q$-tori. If $\theta\co \ol T(\F_q) \to \ol\Q_\ell^\times$ and $\theta'\co \ol T'(\F_q) \to \ol\Q_\ell^\times$ are characters, then
    \[
    \langle R_{\ol T}^{\ol G}(\theta), R_{\ol T'}^{\ol G}(\theta')\rangle_{\ol G(\F_q)} = \#\{w \in W_{\ol G}(\ol T, \ol T')\co \theta = {}^w\theta'\},
    \]
    where $W_{\ol G}(\ol T, \ol T') = \{g \in \ol G(\F_q)\co \ol T = {}^g\ol T'\}/\ol T'(\F_q)$.
\end{lemma}

\begin{proof}
    Let $g_1, \dots, g_n \in \ol G(\F_q)$ be representatives for the quotient $\ol G(\F_q)/\ol G^\circ(\F_q) \cdot \ol T(\F_q)$, and let $\ol Z = Z(\ol G)$. By the Mackey formula, we have
    \begin{align*}
        \langle R_{\ol T}^{\ol G}(\theta), R_{\ol T'}^{\ol G}(\theta')\rangle_{\ol G(\F_q)} &= \langle \ind_{(\ol G^\circ \cdot \ol Z)(\F_q)}^{\ol G(\F_q)} R_{\ol T}^{\ol G^\circ \cdot \ol Z}(\theta), \ind_{(\ol G^\circ \cdot \ol Z)(\F_q)}^{\ol G(\F_q)} R_{\ol T'}^{\ol G^\circ \cdot \ol Z}(\theta')\rangle_{\ol G(\F_q)} \\
            &= \sum_{i=1}^n \langle R_{{}^{g_i}\ol T}^{\ol G^\circ \cdot \ol Z}({}^{g_i}\theta), R_{\ol T'}^{\ol G^\circ \cdot \ol Z}(\theta')\rangle_{(\ol G^\circ \cdot \ol Z)(\F_q)}.
    \end{align*}
    By \cite[Remark 2.6.5]{Kal21b}, if $\theta^\circ = \theta|_{\ol T^\circ(\F_q)}$ and $\theta'^\circ = \theta'|_{\ol T'^\circ(\F_q)}$ then $R_{{}^{g_i}\ol T}^{\ol G^\circ \cdot \ol Z}({}^{g_i}\theta)$ restricts to $R_{{}^{g_i}\ol T^\circ}^{\ol G^\circ}({}^{g_i}\theta^\circ)$ and $R_{\ol T'}^{\ol G^\circ \cdot \ol Z}(\theta')$ restricts to $R_{\ol T'^\circ}^{\ol G^\circ}(\theta'^\circ)$ as virtual $\ol G^\circ(\F_q)$-representations. Note that $R_{{}^{g_i}\ol T}^{\ol G^\circ \cdot \ol Z}({}^{g_i}\theta)$ has central character ${}^{g_i}\theta|_{\ol Z(\F_q)} = \theta|_{\ol Z(\F_q)}$ and $R_{\ol T'}^{\ol G^\circ \cdot \ol Z}(\theta')$ has central character $\theta'|_{\ol Z(\F_q)}$. Since $(\ol G^\circ \cdot \ol Z)(\F_q) = \ol G^\circ(\F_q) \cdot \ol Z(\F_q)$ by Lang's theorem (as $\ol Z \cap \ol G^\circ$ is connected by hypothesis), it follows that
    \[
    \langle R_{{}^{g_i}\ol T}^{\ol G^\circ \cdot \ol Z}({}^{g_i}\theta), R_{\ol T'}^{\ol G^\circ \cdot \ol Z}(\theta')\rangle_{(\ol G^\circ \cdot \ol Z)(\F_q)} = \begin{cases}
        \langle R_{{}^{g_i}\ol T^\circ}^{\ol G^\circ}({}^{g_i}\theta^\circ), R_{\ol T'^\circ}^{\ol G^\circ}(\theta'^\circ)\rangle_{\ol G^\circ(\F_q)} &\text{if } \theta|_{\ol Z(\F_q)} = \theta'|_{\ol Z(\F_q)}, \\
        0 &\text{otherwise.}
    \end{cases}
    \]
    Similarly, we have $\ol T(\F_q) = \ol T^\circ(\F_q) \cdot \ol Z(\F_q)$, so ${}^{g_i}\theta = {}^w\theta'$ if and only if $\theta|_{\ol Z(\F_q)} = \theta'|_{\ol Z(\F_q)}$ and ${}^{g_i}\theta^\circ = {}^w\theta'^\circ$. Now the result follows from \cite[Theorem~6.8]{DL76}.
\end{proof}

The main point of the following result is that its bound is essentially independent of $q$; it ihas not been seriously optimized. 

\begin{lemma}\label{lemma:dl-6.8-variant}
    Let $\ol G$ be a paraductive $\F_q$-group scheme, let $\ol T \subset \ol G$ be a generalized maximal torus, and let $\theta\co \ol T(\F_q) \to \ol\Q_\ell^\times$ be a character. If $n = [\ol G(\F_q)\co (\ol G^\circ \cdot \ol T)(\F_q)]$ and $m = |\pi_0(\ol Z_0 \cap \ol G^\circ)(\F_q)|$, then the number of irreducible characters of $\ol G(\F_q)$ with nonzero pairing with $R_{\ol T}^{\ol G}(\theta)$ is at most 
    \[
    m n \cdot \#\{w \in \Omega_{\ol G^\circ}(\ol T^\circ)(\F_q)\co {}^w\theta = \theta\},
    \]
    where $\Omega_{\ol G^\circ}(\ol T^\circ) = N_{\ol G^\circ}(\ol T^\circ)/\ol T^\circ$ is the Weyl group of $(\ol G^\circ, \ol T^\circ)$.
\end{lemma}

\begin{proof}
    Let $N = \#\{w \in \Omega_{\ol G^\circ}(\ol T^\circ)(\F_q)\co {}^w\theta = \theta\}$ and let $\ol Z = Z(\ol G)$. Choose an embedding of $\ol Z \cap \ol G^\circ$ into an $\F_q$-torus $\wt Z^\circ$, and let $\wt G = \ol G \times^{\ol Z \cap \ol G^\circ} \wt Z^\circ$. Note that there is a natural embedding $\ol G \subset \wt G$ with torus cokernel, and $\wt G$ has center $\wt Z \coloneqq \wt Z^\circ \cdot \ol Z$ satisfying the condition that $\wt Z \cap \wt G^\circ = \wt Z^\circ$ is a torus. Let $\wt T = \ol T \cdot \wt Z$ and choose an extension $\wt\theta\co \wt T(\F_q) \to \ol\Q_\ell^\times$ of $\theta$. By \cite[Remark 2.6.5]{Kal21b}, the virtual representation $R_{\wt T}^{\wt G}(\wt\theta)$ of $\wt G(\F_q)$ restricts to $R_{\ol T}^{\ol G}(\theta)$ on $\ol G(\F_q)$. By Lemma~\ref{lemma:dl-6.8}, we have
    \[
    \langle R_{\wt T}^{\wt G}(\wt\theta), R_{\wt T}^{\wt G}(\wt\theta)\rangle_{\wt G(\F_q)} \leq n \cdot \#\{w \in \Omega_{\ol G^\circ}(\ol T^\circ)(\F_q)\co {}^w\theta = \theta\} = nN,
    \]
    so we can write $R_{\wt T}^{\wt G}(\wt\theta) = \sum_{i=1}^{nN} a_i\wt\chi_i$ for $a_i \in \Z$ and irreducible characters $\wt\chi_i$. By Clifford's theorem, the restriction $\wt\chi_i|_{\ol G(\F_q) \cdot \wt Z^\circ(\F_q)}$ has at most $[\wt G(\F_q)\co \ol G(\F_q) \cdot \wt Z^\circ(\F_q)]$ irreducible constituents, and the same is therefore true of $\wt\chi_i|_{\ol G(\F_q)}$.
    Note that $[\wt G(\F_q)\co \ol G(\F_q) \cdot \wt Z^\circ(\F_q)] = m$ since $|\rH^1(\F_q, \ol Z \cap \ol G^\circ)| = |\pi_0(\ol Z \cap \ol G^\circ)(\F_q)|$. Since $R_{\ol T}^{\ol G}(\theta) = \sum_{i=1}^{nN} a_i\wt\chi_i|_{\ol G(\F_q)}$, the result follows.
\end{proof}

\subsection{Geometric conjugacy and Lusztig series}\label{sss:geom-conj}

Throughout this section, let $\ol G$ be a paraductive $\F_q$-group scheme and let $\ol T, \ol T' \subset \ol G$ be generalized maximal tori. The following definition is a naive extension of \cite[Definition 5.5]{DL76}.

\begin{defn}\label{def:geom-conj}
    Let $k$ be a field of characteristic $\neq p$, and let $\theta\co \ol T(\F_q) \to k^\times$ and $\theta'\co \ol T'(\F_q) \to k^\times$ be characters. The pairs $(\ol T, \theta)$ and $(\ol T',\theta')$ are said to be \textit{geometrically conjugate} if there exists a positive integer $n \geq 1$ such that the pairs $(\ol T_{\F_{q^n}}, \theta \circ \Nm_{\F_{q^n}/\F_q})$ and $(\ol T'_{\F_{q^n}}, \theta' \circ \Nm_{\F_{q^n}/\F_q})$ are $\ol G(\F_{q^n})$-conjugate.
\end{defn}

Note that, unlike in the case that $\ol G$ is conected reductive, the norm map $\Nm_{\F_{q^n}/\F_q}\co \ol T(\F_{q^n}) \to \ol T(\F_q)$ is typically not surjective. Thus geometric conjugacy is somewhat ``lossy''. We will shortly refine it.

Our present goal is to generalize a result of Lusztig \cite[Corollaire~11.11]{Bon06} to $\ol G$, which roughly speaking shows that (rational) Lusztig series behave well with respect to Lusztig induction and restriction. In order to properly contextualize the results (and because it will be convenient later to have this language at hand), we briefly recall the notion of Lusztig series.

Recall \cite[Definition 5.21]{DL76} that a Deligne--Lusztig dual group to $\ol G^\circ$ is a connected reductive $\F_q$-group $\ol G^*$ whose abstract Cartan $\ol \bT^*$ is equipped with an isomorphism with the dual of the abstract Cartan $\ol \bT$ of $\ol G^\circ$ which sends simple roots to simple coroots. Any pair $(\ol T^\circ, \theta^\circ)$ consisting of a maximal $\F_q$-torus $\ol T^\circ \subset \ol G^\circ$ and a character $\theta^\circ\co \ol T^\circ(\F_q) \to \ol\Q_\ell^\times$ gives rise to a $\ol G^*(\F_q)$-conjugacy class of semisimple elements $\ol s \in \ol G^*(\F_q)$, as follows. First, there is a well-defined $\ol G^*(\F_q)$-conjugacy class of $\F_q$-tori $\ol T^* \subset \ol G^*$ which is dual to the $\ol G^\circ(\F_q)$-conjugacy class of $\ol T^\circ$. Next, there is a natural isomorphism $\Hom(\ol T^\circ(\F_q), \ol\Q_\ell^\times) \cong \ol T^*(\F_q),$\footnote{Strictly speaking, this depends on a choice of injection $(\Q/\Z)_{p'} \subset \ol\Q_\ell^\times$; however, we will not make any statements which depend on this choice.} so $\theta^\circ$ gives rise to a semisimple element $\ol s \in \ol T^*(\F_q)$.

If $V$ is an irreducible $\ol\Q_\ell$-representation of $\ol G^\circ(\F_q)$, then $V$ is an irreducible constituent of the Deligne--Lusztig induction $R_{\ol T^\circ}^{\ol G^\circ}(\theta^\circ)$ for some such $(\ol T^\circ, \theta^\circ)$, and we say that $V$ lies in the \textit{rational} (resp.\ \textit{geometric}) \textit{Lusztig series} $\cE(\ol G^\circ, [\ol t])$ (resp.\ $\cE(\ol G^\circ, (\ol t))$ for a semisimple element $\ol t \in \ol G^*(\F_q)$ if the element $\ol s$ is $\ol G^*(\F_q)$-conjugate (resp.\ $\ol G^*(\ol\F_q)$-conjugate) to $\ol t$. By \cite[Proposition~5.22]{DL76}, the pairs $(\ol T^\circ, \theta^\circ)$ and $(\ol T'^\circ, \theta'^\circ)$ are geometrically conjugate if and only if the corresponding semisimple elements $\ol s, \ol s' \in \ol G^*(\F_q)$ are $\ol G^*(\ol\F_q)$-conjugate.

It is a theorem of Lusztig, proven in \cite[Th\'eor\`eme 11.8(b)]{Bon06}, that the rational Lusztig series $\cE(\ol G^\circ, [\ol s])$ form a partition of the set of irreducible $\ol\Q_\ell$-representations of $\ol G^\circ(\F_q)$. This generalizes \cite[Th\'eor\`eme 6.2]{DL76}, which proved the analogous assertion for geometric Lusztig series. Lusztig induction and restriction interact with Lusztig series in the obvious way, i.e., if $\ol L^\circ \subset \ol G^\circ$ is a twisted Levi $\F_q$-subgroup then $R_{\ol L^\circ}^{\ol G^\circ}$ sends $\cE(\ol L^\circ, [\ol s])$ to $\cE(\ol G^\circ, [\ol s])$ \cite[Th\'eor\`eme 11.10]{Bon06} and (consequently) $^*R_{\ol L^\circ}^{\ol G^\circ}$ sends $\cE(\ol G^\circ, [\ol s])$ to the union of some Lusztig series $\cE(\ol L^\circ, [\ol t])$ such that $\ol s$ and $\ol t$ are $\ol G^*(\F_q)$-conjugate.

The difference between geometric and rational Lusztig series is somewhat subtle; for instance, they agree when the center of $\ol G^\circ$ is connected. In fact, if $\ol G^\circ \subset \wt G^\circ$ is an embedding such that $\wt G^\circ$ has connected center and $(\ol G^\circ)_{\der} = (\wt G^\circ)_{\der}$, then every rational Lusztig series for $\ol G^\circ$ is the set of restrictions to $\ol G^\circ(\F_q)$ of the irreducible representations occurring in a geometric Lusztig series for $\wt G^\circ(\F_q)$. To a first approximation, one can imagine that a rational Lusztig series is the subset of a geometric Lusztig series with a fixed central character.\footnote{In fact, this is precisely what a rational Lusztig series is for every absolutely simple group which is not either of (absolute) type A or type D; we are not aware of a reference for this fact, and we will not need it, but the key point of the proof is that the fundamental group in all other types is of prime order.}

We now define notions of ``geometric'' and ``semi-rational'' Lusztig series for paraductive $\F_q$-group schemes $\ol G$. Note that we do \emph{not} define rational Lusztig series in this setting; we expect that the ``proper'' definition of rational Lusztig series coincides with the definition of semi-rational Lusztig series when $Z(\ol G) \cap \ol G^\circ$ is connected, but not otherwise.

We say that an embedding $\ol G \subset \wt G$ of paraductive $\F_q$-group schemes is a \emph{regular embedding} provided that the following conditions hold:
\begin{itemize}
    \item $(\ol G^\circ)_{\der} = (\wt G^\circ)_{\der}$,
    \item $\pi_0(\ol G)(\ol\F_q) = \pi_0(\wt G)(\ol\F_q)$,
    \item $\wt G = \ol G \cdot (Z(\wt G) \cap \wt G^\circ)$,
    \item $Z(\wt G) \cap \wt G^\circ$ is a torus.
\end{itemize}
Note that this extends the usual definition of regular embeddings of connected reductive $\F_q$-groups.

\begin{defn}\label{defn:lusztig-series}
    Let $\theta\co \ol T(\F_q) \to \ol\Q_\ell^\times$ be a character.
    \begin{enumerate}
        \item Let $\cE(\ol G, (\ol T, \theta))$ denote the set of irreducible $\ol\Q_\ell$-representations $\tau$ of $\ol G(\F_q)$ for which there exists a pair $(\ol T', \theta')$ which is geometrically conjugate to $(\ol T, \theta)$ such that $\tau$ is an irreducible constituent of $R_{\ol T'}^{\ol G}(\theta')$. Call $\cE(\ol G, (\ol T, \theta))$ a \textit{geometric Lusztig series}.
        \item Let $\cE(\ol G, [\ol T, \theta])$ denote the subset of $\cE(\ol G, (\ol T,\theta))$ consisting of those representations $\tau$ for which
        \begin{itemize}
            \item $Z(\ol G)(\F_q)$ acts on $\tau$ by the restriction of $\theta$,
            \item there exists an irreducible constituent $\tau_0$ of $\tau|_{\ol G^\circ(\F_q)}$ such that $\tau_0$ lies in the rational Lusztig series corresponding to $(\ol T^\circ, \theta|_{\ol T^\circ(\F_q)})$.
        \end{itemize}
        Call $\cE(\ol G, [\ol T,\theta])$ a \textit{semi-rational Lusztig series}.\footnote{In view of the proofs in \cite[\S 11]{Bon06}, it seems likely that a ``correct'' definition of \emph{rational Lusztig series} is along the lines of the following definition of $\cE_0(\ol G, [\ol T, \theta])$, involving passage to a regular embedding. Since we are not aware of an equivalent intrinsic definition, we choose not to develop this definition very far.}
        \item Let $\cE_0(\ol G, [\ol T, \theta]) \subset \cE(\ol G, [\ol T,\theta])$ denote the set of irreducible $\ol\Q_\ell$-representations of $\ol G(\F_q)$ for which there exists a regular embedding $\ol G \subset \wt G$ such that, if $\wt T = \ol T \cdot Z(\wt G)$, then there exists an extension $\wt\theta\co \wt T(\F_q) \to \ol\Q_\ell^\times$ of $\theta$ and an irreducible constituent $\wt\tau$ of $R_{\wt T}^{\wt G}(\wt\theta)$ such that $\tau$ is an irreducible constituent of $\wt\tau|_{\ol G(\F_q)}$. Note that if $\tau \in \cE_0(\ol G, [\ol T, \theta])$, then $\tau$ is an irreducible constituent of $R_{\ol T}^{\ol G}(\theta)$.
    \end{enumerate}
    If instead $\theta\co \ol T(\F_q) \to \ol\F_\ell^\times$ is a character, then let $\cE(\ol G, (\ol T,\theta))$ (resp.\ $\cE(\ol G, [\ol T,\theta])$, resp.\ $\cE_0(\ol G, [\ol T, \theta])$) be the set of irreducible $\ol\F_\ell$-representations of $\ol G(\F_q)$ which occur as irreducible constituents of the $\ell$-modular reduction of some element of $\cE(\ol G, (\ol T,\wt\theta))$ (resp.\ $\cE(\ol G, [\ol T,\wt\theta])$, resp.\ $\cE_0(\ol G, [\ol T, \theta])$), where $\wt\theta\co \ol T(\F_q) \to \ol\Z_\ell^\times$ is some character lifting $\theta$. We will again call $\cE(\ol G, (\ol T,\theta))$ and $\cE(\ol G, [\ol T,\theta])$ a \textit{geometric Lusztig series} and \textit{semi-rational Lusztig series}, respectively.
\end{defn}

We will see that the geometric Lusztig series and the semi-rational Lusztig series both partition the set of irreducible representations of $\ol G(\F_q)$. Geometric Lusztig series are too coarse for our purposes, while the sets $\cE_0(\ol G, [\ol T, \theta])$ are too fine; for instance, they do not partition the set of irreducible representations. We keep them around for a minor bookkeeping purpose in \cite{CF26b}. The following lemma (which fails for geometric Lusztig series) shows that semi-rational Lusztig series are refined enough to have some basic finiteness properties.

\begin{lemma}\label{lemma:finiteness-of-torus-char-pairs}
    If $\ol L$ is a twisted Levi $\F_q$-subgroup of $\ol G$ containing $\ol T$, then there are only finitely many pairs $(\ol T', \theta')$, up to $\ol L(\F_q)$-conjugacy, such that $\cE(\ol G, [\ol T',\theta'])$ and $\cE(\ol G, [\ol T,\theta])$ have nonempty intersection.
\end{lemma}

\begin{proof}
    Observe that $\theta$ and $\theta'$ must have the same restriction to the center of $\ol L(\F_q)$, so this follows from the fact that $\ol L(\F_q)$ is finite mod center.
\end{proof}

The first part of the following lemma generalizes \cite[Th\'eor\`eme 11.8(a)]{Bon06}, while the second part (partially) generalizes \cite[Th\'eor\`eme 2.2]{BM89}; the lemma shows that two semi-rational Lusztig series (for fixed choice of $\ol Z$ as above) are either disjoint or coincide.

\begin{lemma}\label{lemma:geometric-lusztig-series}
    Let $\ol U, \ol U' \subset \ol G^\circ_{\ol\F_q}$ be the unipotent radicals of Borel $\ol\F_q$-subgroups containing $\ol T^\circ_{\ol\F_q}, \ol T'^\circ_{\ol\F_q}$, respectively, and let $\theta\co \ol T(\F_q) \to \ol\Q_\ell^\times$ and $\theta'\co \ol T'(\F_q) \to \ol\Q_\ell^\times$ be characters.
    \begin{enumerate}
        \item If $\rH_c^*(Y_{\ol U}^{\ol G}, \ol\Q_\ell)_\theta$ and $\rH_c^*(Y_{\ol U'}^{\ol G}, \ol\Q_\ell)_{\theta'}$ have an irreducible $\ol G(\F_q)$-constituent in common, then $\cE(\ol G, [\ol T, \theta]) = \cE(\ol G, [\ol T', \theta'])$.
        \item Suppose that $\theta$ and $\theta'$ factor through $\ol\Z_\ell^\times$, and let $\ol\theta$ (resp.\ $\ol\theta'$) be the composition of $\theta$ (resp.\ $\theta'$) with the map $\ol\Z_\ell^\times \to \ol\F_\ell^\times$. If the $\ell$-modular reductions of $\rH_c^*(Y_{\ol U}^{\ol G}, \ol\Q_\ell)_\theta$ and $\rH_c^*(Y_{\ol U'}^{\ol G}, \ol\Q_\ell)_{\theta'}$ have an irreducible $\ol G(\F_q)$-constituent in common, then $\cE(\ol G, [\ol T, \ol\theta]) = \cE(\ol G, [\ol T', \ol\theta'])$.
    \end{enumerate}
\end{lemma}

\begin{proof}
    Recall from Lemma~\ref{lemma:disc-dl-induction} that $\rH_c^i(Y_{\ol U}^{\ol G}, \ol\Q_\ell)_\theta = \ind_{(\ol G^\circ \cdot \ol Z)(\F_q)}^{\ol G(\F_q)} \rH_c^i(Y_{\ol U}^{\ol G^\circ \cdot \ol Z}, \ol\Q_\ell)_\theta$, and similarly for $\rH_c^i(Y_{\ol U'}^{\ol G}, \ol\Q_\ell)_{\theta'}$. By Clifford theory, under the assumptions of (1) it follows that there is some $g \in \ol G(\F_q)$ such that $\rH_c^*(Y_{^g\ol U}^{\ol G^\circ \cdot \ol Z}, \ol\Q_\ell)_{^g\theta}$ and $\rH_c^*(Y_{\ol U'}^{\ol G^\circ \cdot \ol Z}, \ol\Q_\ell)_{\theta'}$ have an irreducible $(\ol G^\circ \cdot \ol Z)(\F_q)$-constituent in common. If $\theta^\circ = \theta|_{\ol T^\circ(\F_q)}$, then we have $\rH_c^*(Y_{^g\ol U}^{\ol G^\circ \cdot \ol Z}, \ol\Q_\ell)_{^g\theta} = \rH_c^*(Y_{^g\ol U}^{\ol G^\circ}, \ol\Q_\ell)_{^g\theta^\circ}$ as $\ol G^\circ(\F_q)$-representations and similarly for $\theta'$. Thus \cite[Th\'eor\`eme 11.8(a)]{Bon06} shows that the pairs $(^g \ol T^\circ, {}^g\theta^\circ)$ and $(\ol T'^\circ, \theta'^\circ)$ correspond to rationally conjugate semisimple elements of $\ol G^*(\F_q)$, where $\ol G^*$ is a Deligne--Lusztig dual group for $\ol G^\circ$. Moreover, it is clear that $\theta|_{\ol Z(\F_q)} = \theta'|_{\ol Z(\F_q)}$, so the conclusion of (1) follows.

    For (2), the same argument as in the previous paragraph reduces one to the case that $\ol G$ is connected; so assume that this is the case. This is now a simple consequence of \cite[Th\'eor\`eme 2.2]{BM89}, as we will explain. If $\ol s, \ol s' \in \ol G^*(\F_q)$ correspond to $(\ol T, \theta)$, $(\ol T', \theta')$, respectively, as in \cite[(5.21.6)]{DL76}, then for any positive integer $n$, the construction shows that the pair $(\ol T,\theta^n)$ corresponds to $\ol s^n$. By \cite[Th\'eor\`eme 2.2]{BM89}, under the assumptions of (2) there is an element $\ol t \in Z_{G^*}(\ol s)(\F_q)$ of $\ell$-power order such that $\ol s'$ is $\ol G^*(\F_q)$-conjugate to $\ol s \ol t$. If $\ol t^{\ell^n} = 1$, then $\ol s'^{\ell^n}$ and $\ol s^{\ell^n}$ are $\ol G^*(\F_q)$-conjugate, so the prime-to-$\ell$ parts of $\ol s$ and $\ol s'$ are $\ol G^*(\F_q)$-conjugate, as desired.
\end{proof}

The following two results (as well as their proofs) are partial analogues of results of Lusztig \cite[Th\'eor\`eme 11.10, Corollaire~11.11]{Bon06}.

\begin{prop}\label{prop:bonnafe-11.10}
    Let $\ol L \subset \ol G$ be a twisted Levi $\F_q$-subgroup of $\ol G$ containing $\ol T$, let $\theta\co \ol T(\F_q) \to \ol\Q_\ell^\times$ be a character, and let $\rho \in \cE(\ol L, [\ol T, \theta])$. If $\chi$ is an irreducible $\ol G(\F_q)$-constituent of $R^{\ol G}_{\ol L}(\rho)$, then $\chi$ lies in $\cE(\ol G, [\ol T,\theta])$.
\end{prop}

\begin{proof}
    We begin by reducing to the case that $\ol G$ is of finite type. By passing to a character twist, we may assume that there exists a constant $\F_q$-subgroup scheme $\ol Z_0 \subset \ol L$ which is central in $\ol G$ such that $\ol Z_0 \cap \ol G^\circ = 1$ and $\ol Z_0(\F_q) \cdot \ol G^\circ(\F_q)$ is of finite index in $\ol G(\F_q)$ and $\ol Z_0(\F_q)$ is torsion-free and $\rho|_{\ol Z_0(\F_q)}$ is trivial. But then $R_{\ol L}^{\ol G}(\rho) = R_{\ol L/\ol Z_0}^{\ol G/\ol Z_0}(\rho)$, so we may pass from $\ol G$ to $\ol G/\ol Z_0$ to assume that $\ol G$ is of finite type.
    
    Let $\ol B$ be a Borel $\ol\F_q$-subgroup of $\ol L^\circ_{\ol\F_q}$ containing $\ol T^\circ$, and let $\ol P$ be a parabolic $\ol\F_q$-subgroup of $\ol G^\circ_{\ol\F_q}$ with Levi $\ol L^\circ$. Let $\ol V$ be the unipotent radical of $\ol B$, and let $\ol U$ be the unipotent radical of $\ol P$. Let $k$ and $k'$ be integers such that $\chi$ occurs as a $\ol G(\F_q)$-constituent of $\rH_c^k(Y_{\ol U}^{\ol G}, \ol\Q_\ell)_\rho$ and $\rho$ occurs as an $\ol L(\F_q)$-constituent of $\rH_c^{k'}(Y_{\ol V}^{\ol L}, \ol\Q_\ell)_\theta$. Thus $\chi$ is a $\ol G(\F_q)$-constituent of the representation
    \[
    \rH_c^k(Y_{\ol U}^{\ol G}, \ol\Q_\ell) \otimes_{\ol\Q_\ell[\ol L(\F_q)]} (\rH_c^{k'}(Y_{\ol V}^{\ol L}, \ol\Q_\ell) \otimes_{\ol \Q_\ell[\ol T(\F_q)]} \theta).
    \]
    Let $\cL_{\ol G}\co \ol G \to \ol G$ and $\cL_{\ol L}\co \ol L \to \ol L$ denote the Lang maps $g \mapsto g^{-1}\Fr_q(g)$. Observe that the natural map $\cL_{\ol G}^{-1}(\ol U)/(\ol U \cap \Fr_q^{-1}(\ol U)) \to Y_{\ol U}^{\ol G}$ is a $\ol G(\F_q) \times \ol T(\F_q)$-equivariant isomorphism, and similarly for the natural map $\cL_{\ol L}^{-1}(\ol V)/(\ol V \cap \Fr_q^{-1}(\ol V)) \to Y_{\ol V}^{\ol L}$. Since $(\ol U \cap \Fr_q^{-1}(\ol U))_{\mathrm{red}}$ and $(\ol V \cap \Fr_q^{-1}(\ol V))_{\mathrm{red}}$ are both scheme-theoretically isomorphic to affine spaces, say of dimensions $d$ and $e$, respectively, we have $\rH_c^k(Y_{\ol U}^{\ol G}, \ol\Q_\ell) \cong \rH_c^{k+2d}(\cL_{\ol G}^{-1}(\ol U), \ol\Q_\ell)$ and $\rH_c^{k'}(Y_{\ol V}^{\ol L}, \ol\Q_\ell) \cong \rH_c^{k'+2e}(\cL_{\ol L}^{-1}(\ol V), \ol\Q_\ell)$ equivariantly with respect to the various group actions involved.

    As in the proof of \cite[11.5]{DM20}, there is a natural map
    \[
    \cL_{\ol G}^{-1}(\ol U) \times^{\ol L(\F_q)} \cL_{\ol L}^{-1}(\ol V) \to \cL_{\ol G}^{-1}(\ol U\ol V)
    \]
    given by $(g, h) \mapsto gh$, which one checks to be an isomorphism. By the K\"unneth formula and the behavior of cohomology under quotients by finite groups, the tensor product representation $\rH_c^k(Y_{\ol U}^{\ol G}, \ol\Q_\ell) \otimes_{\ol\Q_\ell[\ol L(\F_q)]} \rH_c^{k'}(Y_{\ol V}^{\ol L}, \ol\Q_\ell)$ is a $\ol G(\F_q) \times \ol T(\F_q)$-subrepresentation of $\rH_c^{k+k'}(Y_{\ol U\ol V}^{\ol G}, \ol\Q_\ell)$. It follows that $\chi$ is an irreducible $\ol G(\F_q)$-constituent of $\rH_c^{k+k'}(Y_{\ol U\ol V}^{\ol G}, \ol\Q_\ell)_\theta$. By Proposition~\ref{prop:dl-exhaustion}, we may choose a pair $(\ol T', \theta')$ such that $\chi$ is an irreducible $\ol G(\F_q)$-constituent of $R_{\ol T'}^{\ol G}(\theta')$. Let $\ol B'$ be a Borel $\ol\F_q$-subgroup of $\ol G^\circ_{\ol\F_q}$ containing $\ol T'^\circ$, and let $\ol U'$ be its unipotent radical. Then there exists some integer $j$ such that $\chi$ is an irreducible $\ol G(\F_q)$-constituent of $\rH_c^j(Y_{\ol U'}^{\ol G}, \ol\Q_\ell)_{\theta'}$. Lemma~\ref{lemma:geometric-lusztig-series}(1) therefore implies that $\cE(\ol G, [\ol T', \theta']) = \cE(\ol G, [\ol T, \theta])$.
\end{proof}

\begin{cor}\label{cor:bonnafe-11.11}
    Let $\ol L$ be a twisted Levi $\F_q$-subgroup of $\ol G$ containing $\ol T$, let $\theta\co \ol T(\F_q) \to \ol\Q_\ell^\times$ be a character, let $\chi$ be an irreducible $\ol G(\F_q)$-constituent of $R_{\ol T}^{\ol G}(\theta)$, and let $\rho$ be an irreducible $\ol L(\F_q)$-constituent of ${}^*R^{\ol G}_{\ol L}(\chi)$. Then there exists a pair $(\ol T', \theta')$ in $\ol L$ such that $\cE(\ol G, [\ol T', \theta']) = \cE(\ol G, [\ol T, \theta])$ and $\rho \in \cE(\ol L, [\ol T', \theta'])$.
\end{cor}

\begin{proof}
    This is immediate from Proposition~\ref{prop:bonnafe-11.10} and the definitions. 
\end{proof}

Finally, we record a technical result concerning $\cE_0$, which shows that it is independent of the choice of regular embedding.

\begin{lemma}\label{lemma:E_0-independence-of-regular-embedding}
    Let $\ol G \subset \wt G$ be a regular embedding, let $\ol T \subset \ol G$ be a generalized maximal $\F_q$-torus, let $\wt T = \ol T \cdot Z(\wt G)$, and let $\theta\co \ol T(\F_q) \to \ol\Q_\ell^\times$ be a character. Then $\tau \in \cE_0(\ol G, [\ol T, \theta])$ if and only if there exists a character $\wt\theta\co \wt T(\F_q) \to \ol\Q_\ell^\times$ extending $\theta$ and some irreducible constituent $\wt\tau$ of $R_{\wt T}^{\wt G}(\wt\theta)$ such that $\tau$ is an irreducible subrepresentation of $\wt\tau|_{\ol G(\F_q)}$.
\end{lemma}

\begin{proof}
    If $\tau \in \cE_0(\ol G, [\ol T, \theta])$, then by definition there is a regular embedding $\ol G \subset \wt G'$ and a character $\wt\theta'\co \wt T'(\F_q) \to \ol\Q_\ell^\times$, where $\wt T' = \ol T \cdot Z(\wt G')$, and an irreducible constituent $\wt\tau'$ of $R_{\wt T'}^{\wt G'}(\wt\theta')$ such that $\tau$ is an irreducible subrepresentation of $\wt\tau'|_{\ol G(\F_q)}$. Note that $\wt G \cong \ol G \times^{Z(\ol G)} Z(\wt G)$, and similarly for $\wt G'$. Let 
    \[
    \wt G'' = \wt G \times^{Z(\ol G)} Z(\wt G') \cong \wt G',
    \]
    let $\wt T'' = \ol T \cdot Z(\wt G'')$, let $\theta''$ be a character of $\wt T''(\F_q)$ extending $\theta'$, and let $\wt\theta$ be the restriction of $\wt\theta''$ to $\wt T(\F_q)$. By \cite[Remark 2.6.5]{Kal21b}, the Deligne--Lusztig induction $R_{\wt T'}^{\wt G'}(\wt\theta')$ (resp.\ $R_{\wt T}^{\wt G}(\wt\theta)$) is the restriction of $R_{\wt T''}^{\wt G''}(\wt\theta'')$ to $\wt G'(\F_q)$ (resp.\ $\wt G(\F_q)$). Since $\wt G''(\F_q) = \wt G'(\F_q) \cdot Z(\wt G'')(\F_q)$ by Lang's theorem, it follows that every irreducible constituent of $R_{\wt T}^{\wt G}(\wt\theta)$ (resp.\ $R_{\wt T'}^{\wt G'}(\wt\theta')$) extends to an irreducible constituent of $R_{\wt T''}^{\wt G''}(\wt\theta'')$. These observations combine to prove the claim.
\end{proof}

\subsection{Non-singular representations}

Under some genericity assumptions, we can improve ``geometric conjugacy'' to ``rational conjugacy'' in Lemma~\ref{lemma:geometric-lusztig-series}.

\begin{defn}\label{defn:finite-group-non-singular}
    Let $\ol G$ be a paraductive $\F_q$-group scheme, and let $\ol T \subset \ol G$ be a generalized maximal torus. If $k$ is a field of characteristic $\neq p$, then a character $\theta\co \ol T(\F_q) \to k^\times$ is \textit{non-singular} if $\theta|_{\ol T^\circ(\F_q)}$ is non-singular in the sense of \cite[Definition 5.15]{DL76} (see also \cite[Lemma 3.4.14]{Kal19}), i.e., for a positive integer $n$ such that $\ol T^\circ_{\F_{q^n}}$ is split, we have $\theta \circ \Nm_{\F_{q^n}/\F_q} \circ \alpha^\vee \neq 1$ for every coroot $\alpha^\vee$ of $T^\circ_{\F_{q^n}}$.
\end{defn}

\begin{lemma}\label{lemma:non-singular-conjugacy}
    Let $\ol G$ be a paraductive $\F_q$-group scheme, let $\ol T, \ol T' \subset \ol G$ be generalized maximal tori, let $k \in \{\ol\Q_\ell, \ol\F_\ell\}$ and let $\theta\co \ol T(\F_q) \to k^\times$ and $\theta'\co \ol T'(\F_q) \to k^\times$ be characters. If $\theta$ is non-singular and $\cE(\ol G, [\ol T, \theta]) = \cE(\ol G, [\ol T', \theta'])$, then $(\ol T, \theta)$ and $(\ol T', \theta')$ are $\ol G(\F_q)$-conjugate.
\end{lemma}

\begin{proof}
    Suppose first $k = \ol\Q_\ell$. Let $\ol G^*$ be the Deligne--Lusztig dual group for $\ol G^\circ$, and let $\ol s, \ol s' \in \ol G^*(\F_q)$ be semisimple elements corresponding to $(\ol T^\circ, \theta^\circ)$ and $(\ol T'^\circ, \theta'^\circ)$, respectively, where $\theta^\circ = \theta|_{\ol T^\circ(\F_q)}$, and similarly for $\theta'^\circ$. Now \cite[Th\'eor\`eme 11.8]{Bon06} shows that $\ol s$ and $\ol s'$ are $\ol G^*(\F_q)$-conjugate. Since $\theta$ is non-singular, the element $\ol s$ is regular, and thus the same is true of $\ol s'$, i.e., $\theta'$ is non-singular. This case is therefore a restatement of \cite[Proposition~2.6.11]{Kal21b}.
    
    Now suppose $k = \ol\F_\ell$. The proof of \cite[Proposition~2.6.11]{Kal21b} works nearly verbatim provided that one has the result in the case $\ol G = \ol G^\circ$; so we will assume $\ol G$ is connected. Choose lifts $\wt\theta$ and $\wt\theta'$ of $\theta$ and $\theta'$, respectively, to $\ol\Z_\ell^\times$-valued characters, and let $\ol s$ and $\ol s'$ be associated to $\theta$ and $\theta'$ as above. If $\ol T^*$ and $\ol T'^*$ are $\F_q$-tori of $\ol G^*$ which are dual to $\ol T$ and $\ol T'$, then up to $\ol G^*(\F_q)$-conjugacy we have $\ol s \in \ol T^*(\F_q)$ and $\ol s' \in \ol T'^*(\F_q)$. By \cite[Th\'eor\`eme 2.2]{BM89}, there is some $\ol t \in Z_{\ol G^*}(\ol s')(\F_q)$ of $\ell$-power order such that $\ol s$ and $\ol s'\ol t$ are $\ol G^*(\F_q)$-conjugate. Since $\theta$ is assumed non-singular, every $\ell$-power $\ol s^{\ell^n}$ is regular. If $\ol t^{\ell^n} = 1$, then $\ol s^{\ell^n}$ and $\ol s'^{\ell^n}$ are $\ol G^*(\F_q)$-conjugate, so $\ol s'^{\ell^n}$ is regular and thus the tori $\ol T^*$ and $\ol T'^*$ are $\ol G^*(\F_q)$-conjugate. But then $\ol T$ and $\ol T'$ are $\ol G(\F_q)$-conjugate by \cite[(5.21.4)]{DL76}, so we may assume $\ol T = \ol T'$. In this case, $\ol s^{\ell^n}$ and $\ol s'^{\ell^n}$ are conjugate by the relative Weyl group of $(\ol G^*, \ol T^*)$ and hence $\theta$ and $\theta'$ are conjugate by the relative Weyl group of $(\ol G, \ol T)$, as desired.
\end{proof}

\begin{defn}\label{defn:finite-group-ns-reps}
    Let $k$ be a field among $\ol\F_\ell$ and $\ol\Q_\ell$, and let $\tau$ be an irreducible $k$-representation of $\ol G(\F_q)$. If $\tau \in \cE(\ol G, [\ol T,\theta])$, where $\theta\colon \ol T(\F_q) \to k^\times$ is a non-singular character, then we will call $\tau$ \textit{non-singular}.
\end{defn}

\begin{lemma}\label{lemma:ns-dl-res-cuspidal}
    Let $\ol G$ be a paraductive $\F_q$-group scheme, let $k$ be a field among $\ol\Q_\ell$ and $\ol\F_\ell$, let $\tau$ be a non-singular cuspidal irreducible $k$-representation of $\ol G(\F_q)$ lying in a semi-rational Lusztig series $\cE(\ol G, [\ol T, \theta])$, and let $\ol H \subset \ol G$ be a twisted Levi $\F_q$-subgroup containing $\ol T$. If $(\ol T', \theta')$ is a pair in $\ol H$ such that $\cE(\ol G, [\ol T,\theta]) = \cE(\ol G, [\ol T',\theta'])$, then $\ol T^\circ$ and $\ol T'^\circ$ are elliptic and every irreducible representation occurring in $\cE(\ol H, [\ol T',\theta'])$ is non-singular and cuspidal.
\end{lemma}

\begin{proof}
    We may and do assume that $\ol G$ is connected. If $k = \ol\Q_\ell$, then cuspidality of $\tau$ implies that $\ol T$ is elliptic. By Lemma~\ref{lemma:non-singular-conjugacy}, the pairs $(\ol T, \theta)$ and $(\ol T', \theta')$ are $\ol G(\F_q)$-conjugate, so the conclusion follows from \cite[Theorem~8.3]{DL76}.

    Assume now that $k=\ol\F_\ell$. Let $\wt\theta$ and $\wt\theta'$ denote the Teichm\"uller lifts of $\theta$ and $\theta'$, respectively. Note first that $\ol T$ is elliptic: for this, one may reduce to the case that $\ol G$ has connected center, and thus $\theta$ is in general position (in the sense of \cite[Definition 5.15(ii)]{DL76}). In this case, $V = R_{\ol T}^{\ol G}(\wt\theta)$ is an irreducible $\ol\Q_\ell$-representation by \cite[Theorem~6.8]{DL76}, and $V_{\ol\F_\ell}$ is also irreducible by \cite[Corollaire~3.6]{Br90}. Since $V_{\ol\F_\ell}$ is cuspidal by hypothesis, it follows that $V$ is cuspidal and hence $\ol T$ is elliptic. By Lemma~\ref{lemma:non-singular-conjugacy}, the pairs $(\ol T, \theta)$ and $(\ol T', \theta')$ are $\ol G(\F_q)$-conjugate, and thus $(\ol T, \wt\theta)$ and $(\ol T', \wt\theta')$ are $\ol G(\F_q)$-conjugate. Since $\wt\theta$ and $\wt\theta'$ are non-singular and valued in $\ol\Q_\ell^\times$, we conclude as before.
\end{proof}

\begin{lemma}\label{lemma:ns-cuspidal-lift}
Let $\ol G$ be a paraductive $\F_q$-group scheme, and let $\ol\tau$ be an irreducible cuspidal non-singular $\ol\F_\ell$-representation of $\ol G(\F_q)$. Then there exists an irreducible cuspidal $\ol\Q_\ell$-representation $\wt\tau$ of $\ol G(\F_q)$ with finite-order central character, and a $\ol G(\F_q)$-stable $\ol\Z_\ell$-lattice $\Lambda \subset \wt\tau$ such that $\ol\tau$ occurs as an irreducible subquotient of $\Lambda \otimes_{\ol\Z_\ell}\ol\F_\ell$.
\end{lemma}

\begin{proof}
As in the proof of Lemma~\ref{lemma:dl-restriction}, after quotienting by a torsion-free finite-index constant subgroup of $Z(\ol G)$ on which the central character of $\ol\tau$ is trivial, we may assume that $\ol G$ has finite component group. By Definition~\ref{defn:finite-group-ns-reps}, there is a generalized maximal torus $\ol T \subset \ol G$, a character $\wt\theta\colon \ol T(\F_q)\to \ol\Z_\ell^\times$ whose reduction $\ol\theta$ is non-singular, and an irreducible constituent $\wt\tau$ of $R_{\ol T}^{\ol G}(\wt\theta)$ such that $\ol\tau$ occurs as an irreducible subquotient of the $\ell$-modular reduction of $\wt\tau$. Since $\ol\tau$ is cuspidal and non-singular, Lemma~\ref{lemma:ns-dl-res-cuspidal}, applied with $\ol H=\ol G$ and $(\ol T',\theta')=(\ol T,\ol\theta)$, implies that every irreducible representation in $\cE(\ol G,[\ol T,\ol\theta])$ is cuspidal. In particular, every irreducible subquotient of the mod $\ell$ reduction of $\wt\tau$ is cuspidal.
\end{proof}

\subsection{Parabolic induction and cuspidality}\label{sss:cusp}

Throughout this section, let $\ol G$ be a paraductive $\F_q$-group scheme. Let $k$ be an algebraically closed field of characteristic $\neq p$. A particularly useful special case of the Lusztig induction $R_{\ol L}^{\ol G}$ is the case that $\ol L$ is the centralizer of a \emph{split} $\F_q$-torus of $\ol G$ and the implicit unipotent group $\ol U$ is defined over $\F_q$. In this case, if $\ol P = \ol L \cdot \ol U$ then $Y_{\ol U}^{\ol G} = \ol G(\F_q)/\ol U(\F_q)$, so the definitions show
\[
    R_{\ol L}^{\ol G} = \ind_{\ol P(\F_q)}^{\ol G(\F_q)}\co \Rep_k(\ol L(\F_q)) \to \Rep_k(\ol G(\F_q)),
\]
and therefore $^*R^{\ol G}_{\ol L}(-) = (-)_{\ol U(\F_q)}$, the module of coinvariants.

Recall that a finite-dimensional $k$-representation $\chi$ of $\ol G^\circ(\F_q)$ is \textit{cuspidal} if and only if, for every parabolic $\F_q$-subgroup $\ol P^\circ \subset \ol G$ with Levi $\ol L^\circ$ and finite-dimensional $k$-representation $\rho$ of $\ol L^\circ(\F_q)$, no irreducible subquotient of $\chi$ is a subrepresentation of $\ind_{\ol P^\circ(\F_q)}^{\ol G^\circ(\F_q)}(\rho)$. Note that $\chi$ is cuspidal if and only if every irreducible subquotient of $\chi$ is cuspidal.

By extension, if $\chi$ is a finite-dimensional $k$-representation of $\ol G(\F_q)$, we will say that $\chi$ is \textit{cuspidal} if the restriction of $\chi$ to $\ol G^\circ(\F_q)$ is cuspidal; as above, this is the case if and only if every irreducible subquotient of the restriction of $\chi$ to $\ol G^\circ(\F_q)$ is cuspidal.

\begin{lemma}\label{lemma:cuspidal-equivalence}
    The following statements are equivalent for any finite-dimensional $k$-representation $\chi$ of $\ol G(\F_q)$:
    \begin{enumerate}
        \item $\chi$ is cuspidal,
        \item $\chi_{\ol U^\circ(\F_q)} = 0$ for all parabolic $\F_q$-subgroups $\ol P^\circ \subset \ol G^\circ$ with unipotent radical $\ol U^\circ$,
        \item $\chi^{\ol U^\circ(\F_q)} = 0$ for all parabolic $\F_q$-subgroups $\ol P^\circ \subset \ol G^\circ$ with unipotent radical $\ol U^\circ$.
    \end{enumerate}
\end{lemma}

\begin{proof}
    Since $k$ is of characteristic $\neq p$ and $\ol U^\circ(\F_q)$ is of $p$-power order for every unipotent radical $\ol U^\circ$ of a parabolic $\F_q$-subgroup of $\ol G^\circ$, the functor $V \mapsto V^{\ol U^\circ(\F_q)}$ of invariants is exact and naturally isomorphic to the functor $V \mapsto V_{\ol U^\circ(\F_q)}$ of coinvariants. The lemma follows from these observations and the adjunction.
\end{proof}

\begin{lemma}\label{lemma:cuspidal-supports-exist}
    Let $\chi$ be an irreducible $k$-representation of $\ol G(\F_q)$. There exists a split $\F_q$-torus $\ol S^\circ \subset \ol G$ and an irreducible cuspidal $k$-representation $\rho$ of $Z_{\ol G}(\ol S^\circ)(\F_q)$ such that
    \begin{enumerate}
        \item $\ol S^\circ$ is the maximal split central $\F_q$-torus of $L\coloneqq Z_{\ol G}(\ol S^\circ)$,
        \item $\chi$ is an irreducible subrepresentation of $R_{\ol L}^{\ol G}(\rho)$ for some $\ol P$ defined over $\F_q$.
    \end{enumerate}
    The pair $(\ol S^\circ, \rho)$ is unique up to $\ol G(\F_q)$-conjugacy. If $\chi$ has finite order central character, then so does $\rho$.
\end{lemma}

\begin{proof}
    Let $\chi_0$ be an irreducible $k$-subrepresentation of $\chi|_{\ol G^\circ(\F_q)}$. By definition of cuspidality, there is an $\F_q$-torus $\ol S^\circ \subset \ol G$ such that if $\ol L = Z_{\ol G}(\ol S^\circ)$, then there is an irreducible cuspidal $k$-representation $\rho_0$ of $\ol L^\circ(\F_q)$ such that $\chi_0$ is a $k$-subrepresentation of the parabolic induction of $\rho_0$. We may and do assume that $\ol S^\circ$ is the maximal central split $\F_q$-subtorus of $\ol L$. Let $\ol U \subset \ol G^\circ$ be the unipotent radical of a parabolic $\F_q$-subgroup of $\ol G^\circ$ with Levi factor $\ol L^\circ$, and let $\ol P = \ol L \ol U$. By definition, $\rho_0$ is an irreducible quotient of $(\chi_0)_{\ol U(\F_q)}$. Since $\rho_0$ is nonzero, in particular $\chi_{\ol U(\F_q)}$ is nonzero, hence it admits an $\ol L(\F_q)$-irreducible quotient $\rho$ such that $\rho|_{\ol L^\circ(\F_q)}$ contains $\rho_0$. By definition, the representation $\rho$ is cuspidal, and by adjunction we see that $\chi$ is an irreducible $k$-subrepresentation of $\ind_{\ol P(\F_q)}^{\ol G(\F_q)}(\rho)$, as desired. Observe that if $\chi$ has finite order central character, then it factors through a finite quotient of $\ol G(\F_q)$, and it is clear that the same is then true of $\rho$.

    Next, we prove uniqueness. Let $(\ol S^\circ, \rho)$ and $(\ol S'^\circ, \rho')$ be two pairs satisfying the conditions of the lemma, and let $\ol U$ and $\ol U'$ be the unipotent radicals of parabolic $\F_q$-subgroups of $\ol G^\circ$ with Levi factors $\ol L^\circ \coloneqq Z_{\ol G^\circ}(\ol S^\circ)$ and $\ol L'^\circ \coloneqq Z_{\ol G^\circ}(\ol S'^\circ)$, respectively. Let $\chi_0$ and $\chi_0'$ be irreducible $k$-subrepresentations of $\chi|_{\ol G^\circ(\F_q)}$ such that $(\chi_0)_{\ol U(\F_q)}$ (resp.\ $(\chi_0')_{\ol U'(\F_q)}$) admits $\rho$ (resp.\ $\rho'$) as an irreducible subrepresentation. By Clifford's theorem, there is some $g \in \ol G(\F_q)$ which conjugates $\chi_0'$ to $\chi_0$; by passing to this conjugate, we may assume $\chi_0 = \chi_0'$. In this case, the usual uniqueness of cuspidal support \cite[Corollary 5.2]{Hiss93} shows that $\rho \cong \rho'$, as desired.
\end{proof}

The following technical lemma will be useful in some reduction arguments later.

\begin{lemma}\label{lemma:paraductive-series-extension}
Let $k \in \{\ol\Q_\ell,\ol\F_\ell\}$, and let $\ol H\subset\ol G$ be a closed paraductive $\F_q$-subgroup scheme such that $\ol H^\circ \cdot Z(\ol G)$ is of finite index in $\ol G$. Let $\ol S \subset \ol H$ be a generalized maximal torus, let $\theta\co \ol S(\F_q)\to k^\times$ be a character, and suppose that $\tau\in\cE(\ol H,[\ol S,\theta])$ is irreducible.

Then $\ol T \coloneqq Z_{\ol G}(\ol S^\circ)$ is a generalized maximal $\F_q$-torus of $\ol G$ with $\ol T\cap\ol H=\ol S$. If $\pi$ is an irreducible $k$-representation of $\ol G(\F_q)$ whose restriction to $\ol H(\F_q)$ admits $\tau$ as an irreducible subquotient, then there exists a character $\eta\co\ol T(\F_q)\to k^\times$ extending $\theta$ such that $\pi\in\cE(\ol G,[\ol T,\eta])$. The representation $\pi$ is cuspidal (resp.\ non-singular) if and only if the same holds for $\tau$.
\end{lemma}

\begin{proof}
The hypotheses imply that $\ol T^\circ \coloneqq \ol S^\circ \cdot (Z(\ol G) \cap \ol G^\circ)$ is a maximal $\F_q$-torus of $\ol G^\circ$. Thus $\ol T = Z_{\ol G}(\ol T^\circ)$ is a generalized maximal $\F_q$-torus of $\ol G$ by definition. The final claim is clear, so it remains to show that $\pi \in \cE_0(\ol G, [\ol T, \eta])$. For this, the definition and \cite[Part III, no.\ 16.1, Theorem~33]{Serre77} reduce us to the case $k = \ol\Q_\ell$.

Let $\ol U \subset \ol G^\circ_{\ol\F_q}$ be the unipotent radical of a Borel $\ol\F_q$-subgroup of $\ol G^\circ$ containing $\ol T^\circ$, so $\ol S^\circ_{\ol\F_q} \ol U$ is also a Borel $\ol\F_q$-subgroup of $\ol H^\circ_{\ol\F_q}$. By assumption, there is some integer $m$ such that $\tau$ is an irreducible subrepresentation of $\rH^m_c(Y_{\ol U}^{\ol H}, \ol\Q_\ell)_\theta$. Observe that 
\[
\rH^m_c(Y_{\ol U}^{\ol G}, \ol\Q_\ell) \cong \rH^m_c(Y_{\ol U}^{\ol H}, \ol\Q_\ell) \otimes_{\ol\Q_\ell[\ol H(\F_q)]} \ol\Q_\ell[\ol G(\F_q)],
\]
so $\pi$ is an irreducible subquotient of $\rH_c^m(Y_{\ol U}^{\ol G}, \ol\Q_\ell) \otimes_{\ol\Q_\ell[\ol S(\F_q)]} \theta$. Every irreducible subquotient of $\ol\Q_\ell[\ol T(\F_q)] \otimes_{\ol\Q_\ell[\ol S(\F_q)]} \theta$ as a representation of $\ol T(\F_q)$ is a character extending $\theta$, so the result follows from Lemma~\ref{lemma:geometric-lusztig-series} and the fact that $\rH_c^m(Y_{\ol U}^{\ol G}, \ol\Q_\ell)$ is finitely generated as a $\ol\Q_\ell[\ol T(\F_q)]$-module.
\end{proof}

We conclude with a question.

\begin{question}\label{question:cuspidal-lifting}
    Does there exist a constant $C$ depending only on the root datum of $\ol G^\circ$ such that for all $\ell > C$ and all irreducible cuspidal $\ol\F_\ell$-representations $\chi_0$ of $\ol G(\F_q)$, there exists an irreducible cuspidal $\ol\Q_\ell$-representation $\chi$ such that $\chi_0$ occurs as an irreducible constituent of the $\ell$-modular reduction of $\chi$?
\end{question}

A positive answer to Question~\ref{question:cuspidal-lifting} would allow us to extend the proof of Theorem~\ref{thm:intro-debacker-reeder} somewhat; see Remark~\ref{rmk:singular-generalization}. According to \cite[Theorem~7.8]{GHM94}, the answer is positive for $\ol G = \GL_n$ by work of Dipper--James \cite{DJ86}. If one assumes that $\chi_0$ is moreover supercuspidal, then a positive answer follows from Geck's conjecture \cite[(6.6)]{Geck92} by \cite[Proposition~3.3]{Hiss96}; this conjecture was proven for unipotent modular representations in \cite{DuM18} when $p$ is good. If $\ell$ is ``small'', then according to \cite[Introduction]{GHM94} the answer was shown to be negative for $\ol G = G_2$ in Hiss' Habilitationsschrift. Beyond these cases, we are not aware of a proof or a counterexample.

\section{Tate cohomology for finite groups}\label{section:rosetta-stone}

There are several different types of correspondences between the representation theory of pairs of finite reductive groups that have the ``feel'' of Langlands functoriality: 
\begin{itemize}
    \item Shintani descent, which can be viewed as a form of ``base change functoriality''. 
    \item Lusztig induction and restriction.
    \item The Glauberman correspondence.
\end{itemize}
Each of these items is a correspondence of characteristic zero (virtual) representations, defined in very different ways from each other. However, we will show that upon reducing modulo certain primes $\ell$, they admit (in a wide class of group-theoretic situations) a common description in terms of Tate cohomology. This is a shadow, at the level of finite reductive groups, of the principle (exemplified in \cite{TV}, \cite{Fe23}, \cite{F24}) that Tate cohomology realizes functoriality in the Local Langlands Correspondence.

\subsection{Preliminaries}

In this section, we prove a few key results of an essentially combinatorial nature which will allow us to compute Tate cohomology in practice. Below, we use $\sigma$ to denote a generator of a cyclic group of order $\ell$. Throughout this section, we let $k$ be a field of characteristic $\ell$, and we let $k[\sigma]$ be the group ring of $\langle \sigma \rangle$.

For $a \in \Z/2$ we have \emph{Tate cohomology} groups $\rT^a(\sigma, -)$, defined as in the introduction. Note that if $\Pi$ is a representation of a group $\Gamma \rtimes \tw{\sigma}$, then $\rT^i(\sigma, \Pi)$ is naturally a representation of the fixed-point subgroup $\Gamma^\sigma$. 

\begin{lemma}\label{lem:N_sigma-identity}
Let $N_\sigma \coloneqq 1 + \sigma + \ldots + \sigma^{\ell-1} \in k[\sigma]$. Then we have 
\[
N_\sigma = (\sigma-1)^{\ell-1} \in k[\sigma].
\]
\end{lemma}
\begin{proof}
This follows trivially from the polynomial identity $\sum_{i=0}^{\ell-1} X^i = (X-1)^{\ell-1}$ in $k[X]$.
\end{proof}

\begin{lemma}\label{lem:semisimplifications-Tate-cohomology-coincide} If $\Pi$ has finite length as a $\Gamma$-representation, then the semisimplifications of $\rT^0(\sigma, \Pi)$ and $\rT^1(\sigma, \Pi)$ are isomorphic as representations of $\Gamma^\sigma$.
\end{lemma}

\begin{proof}
From the defining short exact sequences for $\rT^i(\sigma, \Pi)$, we have
\[
[\rT^0(\sigma,\Pi)]=[\ker(\sigma-1)]-[N_\sigma\Pi]=[\ker(N_\sigma)]-[(\sigma-1)\Pi]=[\rT^1(\sigma,\Pi)],
\]
as desired.
\end{proof}

Despite Lemma~\ref{lem:semisimplifications-Tate-cohomology-coincide}, it will be useful in \cite{CF26b} to consider both $\rT^0$ and $\rT^1$, as the various cup product maps have considerably different behaviors. To simplify the notation, when $\sigma$ is clear from context we will write $\rT^i(\Pi) \coloneqq \rT^i(\sigma, \Pi)$.

\begin{lemma}\label{lemma:nonzero-tate-cohomology}
Let $V$ be a finite-dimensional $k[\sigma]$-module. Let $1 \leq m_1, \dots, m_n \leq \ell$ be the sizes of the (unipotent) Jordan blocks of $\sigma$. Then
    \begin{equation}\label{eq:nonzero-tate-cohomology}
    \dim_k \rT^j(V) = \#\{1 \leq i \leq n\colon m_i < \ell\}
    \end{equation}
    for either $j \in \Z/2\Z$. In particular, if $\dim_k V$ is not divisible by $\ell$, then $\rT^j(V) \neq 0$.
\end{lemma}

\begin{proof}
    By breaking up $V$ into a direct sum of $\sigma$-stable subspaces, we may assume that $V \cong k[\sigma]/((\sigma-1)^m)$ for some $m$. Note that
    \[
    (\sigma-1)^\ell = \sigma^\ell -1 = 0 \in k[\sigma],
    \]
    so $m \leq \ell$. Let $e_i = (\sigma-1)^{m-i}$ for $0\leq i\leq m$. Clearly $V^\sigma = ke_1$ and $(\sigma-1)V = \bigoplus_{i=1}^{m-1} ke_i$. Using Lemma~\ref{lem:N_sigma-identity}, we see that:
    \begin{itemize}
    \item If $m < \ell$, then $\rT^0(V) = ke_1$ and $\rT^1(V) = V/\bigoplus_{i=1}^{m-1} ke_i$.
    \item If $m = \ell$, then $N_\sigma(V) = ke_1$ and $\ker N_\sigma = \bigoplus_{i=1}^{\ell-1} ke_i$. Thus in this case $\rT^0(V) = 0 = \rT^1(V)$.
    \end{itemize}
    This proves \eqref{eq:nonzero-tate-cohomology}. The final claim follows from \eqref{eq:nonzero-tate-cohomology} and the observation that if $\dim_k V$ is not divisible by $\ell$, then $m_i < \ell$ for some $i$.
\end{proof}

\begin{lemma}\label{lemma:tate-cohom-devissage}
    Let $\Gamma$ be a locally profinite group which admits a compact open subgroup of pro-order prime to $\ell$, let $\sigma$ be an automorphism of $\Gamma$ of order $\ell$, and let $V$ be a finite length smooth $k$-representation of $\Gamma \rtimes \langle\sigma\rangle$. Suppose 
    \[
    V^{\ss} \cong \bigoplus_{i=1}^m V_i^{e_i} \oplus \bigoplus_{j=1}^n W_j^{f_j}
    \]
    as a $k[\Gamma]$-module, where the $V_i$ and $W_j$ are pairwise non-isomorphic simple $k[\Gamma]$-modules such that $V_i \cong {}^\sigma V_i$ and $W_j \not\cong {}^\sigma W_j$. Equipping $V_i$ with its canonical $k[\Gamma \rtimes \langle\sigma\rangle]$-module structure, there exists an embedding of $k[\Gamma^\sigma]$-modules
    \[
    \rT^a(\sigma, V)^{\ss} \subset \bigoplus_{i=1}^m (\rT^a(\sigma, V_i)^{\ss})^{e_i}
    \]
    for both $a \in \Z/2$.
\end{lemma}

\begin{proof}
    Note first that if $U$ is a simple $k[\Gamma]$-module whose isomorphism class is $\sigma$-stable, then $U$ admits a unique $k[\Gamma \rtimes \langle\sigma\rangle]$-module structure by the argument of \cite[Proposition~6.1]{TV}. For a short exact sequence $0 \to A \to B \to C \to 0$ of $k[\Gamma\rtimes\langle\sigma\rangle]$-modules, there is an exact sequence $\rT^a(\sigma,A) \to \rT^a(\sigma,B) \to \rT^a(\sigma,C)$, and the conclusion thereby propagates from $A$ and $C$ and $B$. Using the socle filtration and splitting into direct summands, we may therefore assume that $V$ is simple as a representation of $\Gamma \rtimes \langle\sigma\rangle$. Note that $\sigma$ permutes the isotypic components of $V|_\Gamma$ transitively. If $V|_\Gamma$ is not irreducible, then $V$ is an induced $k[\sigma]$-module, and its Tate cohomology vanishes. Otherwise, $V|_\Gamma$ is irreducible and the lemma is clear.
\end{proof}

\subsection{Lower bounds on Tate cohomology}\label{ss:tate-lower-bounds} 
The ``modular functoriality'' results of \cite{F24} require control of Tate cohomology, or at least ``lower bounds'' on it. We will establish some results in this direction, which will ultimately be used to relate Tate cohomology to Shintani descent mod $\ell$, and separately to Lusztig restriction mod $\ell$.

\subsubsection{Brauer characters} Recall that if $\Gamma$ is a finite group, $k$ is an algebraically closed field of characteristic $\ell > 0$, and $V$ is a finite-dimensional $k$-representation of $\Gamma$, then the \emph{Brauer character} $\chi_V$ is the function $\chi_V\co \Gamma_{\ell'} \to W(k)$ defined by
\[
\chi_V(\gamma) = \sum_{i=1}^{\dim V} [\alpha_i]
\]
where $\Gamma_{\ell'}$ is the set of elements of $\Gamma$ of order prime to $\ell$, $W(k)$ is the ring of Witt vectors of $k$, $\{\alpha_1, \dots, \alpha_{\dim V}\}$ is the multi-set of eigenvalues for the action of $\gamma$ on $V$, and $[\alpha]$ refers to the Teichm\"uller lift of $\alpha \in k^\times$. The Brauer character determines the isomorphism class of the semisimplification of $V$ by \cite[\S18.2, Corollary~1]{Serre77}.

More generally, we will say that a class function $\Gamma_{\ell'} \to W(k)$ is a \emph{Brauer character} of $\Gamma$ if it is a $\Z$-linear combination of Brauer characters of finite-dimensional $k$-representations of $\Gamma$. 

\begin{defn}
Let $\chi_1, \dots, \chi_n$ be the Brauer characters associated to the irreducible $k$-representations of $\Gamma$, so every Brauer character of $\Gamma$ can be written uniquely in the form $\sum_{i=1}^n m_i\chi_i$ for $m_i \in \Z$. If $\chi = \sum_{i=1}^n m_i\chi_i$ and $\eta = \sum_{i=1}^n r_i\chi_i$, then we write $\chi \leq \eta$ if $m_i \leq r_i$ for all $i$, and we write 
\[
|\chi| \coloneqq \sum_{i=1}^n |m_i|\chi_i.
\]
If $\{\eta_i\}_{i \in I}$ is a nonempty finite set of Brauer characters and $\eta_i = \sum_{j=1}^n c_{ij} \chi_j$ for $c_{ij} \in \Z$, then we write
\[
\inf_{i \in I}\{\eta_i\} \coloneqq \sum_{j=1}^n \min_{i \in I}\{c_{ij}\} \chi_j \text{ and } \sup_{i \in I}\{\eta_i\} \coloneqq \sum_{j=1}^n \max_{i \in I}\{c_{ij}\}\chi_j.
\]
In other words, $\inf_{i \in I}\{\eta_i\}$ (resp.\ $\sup_{i \in I}\{\eta_i\}$) is the greatest lower bound (resp.\ least upper bound) of the set $\{\eta_i\}_{i \in I}$ under the partial order $\leq$ introduced above.
\end{defn}

\subsubsection{The lower bound}

In this section, we will identify a fairly explicit representation which occurs as a submodule of $\rT^i(\sigma, \ol V)^{\ss}$, whenever $\ol V$ is the $\ell$-modular reduction of a stable lattice in a $\ol\Q_\ell[\Gamma \rtimes \langle\sigma\rangle]$-module $V$. To make this subrepresentation most useful, we need to have some information about the field of definition of $V$. We thank Santosh Nadimpalli for pointing out that the following statement is not obvious.

\begin{lemma}\label{lemma:stable-module-extension}
    Let $\Gamma$ be a finite group, let $\sigma$ be an automorphism of $\Gamma$ of order $\ell\neq 2$, let $K/\Q_\ell^{\unr}$ be a finite extension of degree prime to $\ell - 1$, and let $V_0$ be an absolutely irreducible $K[\Gamma]$-module whose isomorphism class is $\sigma$-stable. Then $V_0$ admits a $K[\Gamma \rtimes \langle\sigma\rangle]$-module structure extending the given $K[\Gamma]$-module structure.
\end{lemma}

\begin{proof}
    Let $V = (V_0)_{\ol\Q_\ell}$. Recall first that the Brauer group of any finite extension $L/K$ is trivial, so $(V_0)_L$ extends to an $L[\Gamma \rtimes \langle\sigma\rangle]$-module if and only if $V$ extends to a $\ol\Q_\ell[\Gamma \rtimes \langle\sigma\rangle]$-module whose character $\chi$ takes values in $L$. We will first show that these conditions hold for $L = K(\mu_\ell)$.
    
    If $n$ is the order of $\Gamma \rtimes \langle\sigma\rangle$, then any $\chi$ as above takes values in $K(\mu_n)$. Since the isomorphism class of $V_0$ is $\sigma$-stable, if $V_0^\sigma$ denotes the twist of $V_0$ by $\sigma$ then there exists a $\Gamma$-equivariant isomorphism $f\colon V_0 \to V_0^\sigma$. Note that $f^\ell$ is a $\Gamma$-equivariant automorphism of $V_0$, so by Schur's lemma there is some $c \in K^\times$ such that $f^\ell = c$. Thus $(V_0)_{K(\sqrt[\ell]{c})}$ admits an extension, and any $\chi$ as above takes values in $K(\mu_\ell, \sqrt[\ell]{c})$. But now
    \begin{equation}\label{eqn:intersection-of-field-extns}
    K(\mu_n) \cap K(\mu_\ell, \sqrt[\ell]{c}) = K(\mu_\ell).
    \end{equation}
    Indeed, if $L$ is the left hand side of \eqref{eqn:intersection-of-field-extns}, then $L$ is an abelian extension of $K$ containing $K(\mu_\ell)$. Since $\ell \neq 2$ and $K/\Q_\ell^{\unr}$ is of degree prime to $\ell-1$, the Galois group $\Gal(K(\mu_\ell, \sqrt[\ell]{c})/K)$ has derived group equal to $\Gal(K(\mu_\ell, \sqrt[\ell]{c})/K(\mu_\ell))$, and it follows that $L = K(\mu_\ell)$, as desired.

    We have now seen that $(V_0)_{K(\mu_\ell)}$ extends to a $K(\mu_\ell)[\Gamma \rtimes \langle\sigma\rangle]$-module. Let $\cE$ be the set of such extensions, so $\cE$ is of cardinality $\ell$. There is a natural action of $(\Z/\ell)^\times \ltimes \Z/\ell$ on $\cE$, where $(\Z/\ell)^\times \cong \Gal(K(\mu_\ell)/K)$ acts by the usual Galois action and $\Z/\ell$ acts through twisting by powers of a nontrivial character of $\langle\sigma\rangle$. The action of $1 \ltimes \Z/\ell$ on $\cE$ is simply transitive, so if $\wt V_0 \in \cE$ is a chosen extension then the stabilizer of $\wt V_0$ in $(\Z/\ell)^\times \ltimes \Z/\ell$ is a complement to $1 \ltimes \Z/\ell$, hence conjugate to $(\Z/\ell)^\times \ltimes 1$. But this means that $\wt V_0$ admits a character twist which is defined over $K$, as desired.
\end{proof}

\begin{remark}
    Lemma~\ref{lemma:stable-module-extension} can fail for $\ell = 2$. For example, if $\Gamma = \SL_3(\F_2)$ then $\Gamma$ admits a unique $6$-dimensional irreducible $\ol\Q_2$-representation $V$, and hence $V$ is defined over $K = \Q_2^{\unr}$. If $\sigma$ is the nontrivial automorphism of $\SL_3$ over $\F_2$ which preserves the standard pinning, then the isomorphism class of $V$ is necessarily $\sigma$-stable. However, $V$ does not admit an extension to a $K[\Gamma \rtimes \langle\sigma\rangle]$-module: to see this, let
    \[
    u = \begin{pmatrix}
        1&1& \\ &1& \\ &&1
    \end{pmatrix}
    \]
    The trace of $u$ on $V$ is $2$ and $u^2 = 1$, so $1$ occurs as an eigenvalue with multiplicity $4$, and $-1$ occurs with multiplicity $2$. Note that $(u \rtimes \sigma)^4$ is $\Gamma$-conjugate to $u$, so if $V_1$ is an extension of $V_{\ol\Q_2}$ to an $\ol\Q_2[\Gamma \rtimes \langle\sigma\rangle]$-module, then $u \rtimes \sigma$ has precisely two eigenvalues $a$ and $b$ on $V_1$ which are primitive $8$th roots of unity. If $V_2$ is the other extension of $V$ to an $\ol\Q_2[\Gamma \rtimes \langle\sigma\rangle]$-module, then the eigenvalues of $u \rtimes \sigma$ on $V_2$ which are primitive $8$th roots of unity are $-a$ and $-b$. Since $\Gal(\ol K/K)$ acts transitively on the primitive $8$th roots of unity, it follows that $\{a, b\} \neq \{-a, -b\}$, i.e., $a \neq -b$. But now $a + a^3 = \pm\sqrt{-2}$ and $a + a^7 = \pm\sqrt{2}$, and neither of these lies in $K(\sqrt{-1})$, so neither $V_1$ nor $V_2$ is defined over $K(\sqrt{-1})$, let alone $K$.
\end{remark}

\begin{thm}\label{thm:eigenvalues-and-jordan-blocks-variant}
    Let $\Gamma$ be a finite group, let $\ell$ be a prime number, let $\sigma$ be an automorphism of $\Gamma$ of order $\ell$, and let $V$ be a $\ol\Z_\ell[\Gamma \rtimes \langle\sigma\rangle]$-module which is finite free as a $\ol\Z_\ell$-module. Let $V_0, \ldots, V_{\ell-1}$ be the $\sigma$-eigenspaces of $V_{\ol \Q_\ell}$ corresponding to the $\ell$th roots of unity, in some order. For each $i$, write $\ol V_i = (V_i \cap V) \otimes_{\ol\Z_\ell}\ol\F_\ell$, and let $\chi_{\ol V_i}$ be its Brauer character as a $\Gamma^\sigma$-representation. Then we have
    \begin{equation}\label{eqn:Brauer-char-lower-bound}
    \chi_{\rT^a(\sigma, V_{\ol\F_\ell})} \geq \sup_{0 \leq i \leq \ell-1}\{\chi_{\ol V_i}\} - \inf_{0 \leq i \leq \ell-1}\{\chi_{\ol V_i}\} \quad \text{for both $a \in \Z/2\Z$.}
    \end{equation}
    In particular, if $V_{\ol\Q_\ell}$ is defined (as a $\ol\Q_\ell[\Gamma]$-module) over a finite extension $K/\Q_\ell^{\mathrm{unr}}$ of ramification degree prime to $\ell-1$ then
    \begin{equation}\label{eqn:Brauer-unramified-lower-bound}
    \chi_{\rT^a(\sigma, V_{\ol\F_\ell})} \geq |\gamma \mapsto \chi_{V_{\ol\Q_\ell}}(\gamma \rtimes \sigma)|.
    \end{equation}
\end{thm}

\begin{proof}
    Since the semisimplifications of $\rT^0$ and $\rT^1$ are isomorphic as $\Gamma^\sigma$-representations, it is sufficient to consider the case $a = 0$. Note that $\sigma$ stabilizes the flag 
    \[
    V_0 \subset V_0 \oplus V_1 \subset \cdots \subset \bigoplus_{i=0}^{\ell-1} V_i = V_{\ol\Q_\ell}.
    \]
    Let 
    \[
    0 = W_{-1} \subset W_0 \subset W_1 \subset \cdots \subset W_{\ell-1} = V_{\ol\F_\ell}
    \]
    be the induced flag obtained by setting $W_i = ((\bigoplus_{j=0}^i V_j) \cap V) \otimes_{\ol\Z_\ell}\ol\F_\ell$. Since $\sigma-1$ acts by $\zeta_i - 1 \in \ol\Z_\ell$ on $V_i \cong \bigoplus_{j=0}^i V_j/\bigoplus_{j=0}^{i-1} V_j$, it also acts by $\zeta_i-1$ on the lattice 
    \[
    \frac{(\bigoplus_{j=0}^i V_j) \cap V }{(\bigoplus_{j=0}^{i-1} V_j)  \cap V } \inj \frac{\bigoplus_{j=0}^i V_j}{\bigoplus_{j=0}^{i-1} V_j} \cong V_i,
    \]
    hence $\sigma-1$ annihilates $W_i/W_{i-1}$. Using Lemma~\ref{lem:N_sigma-identity}, we deduce that $N_\sigma(V_{\ol\F_\ell}) \subset W_0$, so $\chi_{N_\sigma(V_{\ol\F_\ell})} \leq \chi_{\ol V_0}$. By symmetry, we have $\chi_{N_\sigma(V_{\ol\F_\ell})} \leq \chi_{\ol V_i}$ for all $i$, i.e., 
    \begin{equation}\label{eqn:Brauer-norm-upper-bound}
    \chi_{N_\sigma(V_{\ol\F_\ell})} \leq \inf_{0 \leq i \leq \ell-1}\{\chi_{\ol V_i}\}.
    \end{equation}
    Note also that $\chi_{V_{\ol\F_\ell}^\sigma} \geq \chi_{\ol V_0}$, so by the same argument
    \begin{equation}\label{eqn:fixed-point-lower-bound}
    \chi_{V_{\ol\F_\ell}^\sigma} \geq \sup_{0 \leq i \leq \ell-1}\{\chi_{\ol V_i}\}.
    \end{equation}
    Combining \eqref{eqn:Brauer-norm-upper-bound} and \eqref{eqn:fixed-point-lower-bound} with the definition of $\rT^0(\sigma, V)$ yields \eqref{eqn:Brauer-char-lower-bound} in the case $a = 0$.
        
    For \eqref{eqn:Brauer-unramified-lower-bound}, suppose that $V_{\ol\Q_\ell} = U \otimes_K \ol\Q_\ell$ for a $K[\Gamma]$-module $U$. By Lemma~\ref{lemma:stable-module-extension}, if $\ell \neq 2$ then we may take $U$ to be a $K[\Gamma \rtimes \langle\sigma\rangle]$-module; if $\ell = 2$, then the condition on $K$ is vacuous and we may simply increase $K$ if needed to assume the same. Without loss of generality, assume that $\sigma$ acts on $V_0$ with eigenvalue $1$. 
    It follows that the $V_i$, $i \neq 0$, are permuted transitively by $\sigma$, so $\ol V_i \cong \ol V_j$ for $i, j \neq 0$. Thus the right side of \eqref{eqn:Brauer-char-lower-bound} is equal to $|\chi_{\ol V_0} - \chi_{\ol V_1}|$. On the other hand, if $\zeta_i$ is the eigenvalue by which $\sigma$ acts on $V_i$, then for $\gamma \in \Gamma^\sigma_{\ell'}$ we have
    \[
    \chi_{V_{\ol\Q_\ell}}(\gamma \rtimes \sigma) = \sum_{i=0}^{\ell-1} \zeta_i \chi_{\ol V_i}(\gamma) = \chi_{\ol V_0}(\gamma) + \left(\sum_{i=1}^{\ell-1} \zeta_i\right) \chi_{\ol V_1}(\gamma) = \chi_{\ol V_0}(\gamma) - \chi_{\ol V_1}(\gamma).
    \]
    Combining these two observations yields \eqref{eqn:Brauer-unramified-lower-bound}.
\end{proof}

\begin{remark}\label{remark:strict-inequality}
    The inequality in \eqref{eqn:Brauer-unramified-lower-bound} can be strict. For example, let $\Gamma = S_3$, let $\ell = 3$, and let $\sigma$ be the automorphism of $S_3$ induced by conjugation by a $3$-cycle in $\Gamma$, so $\Gamma^\sigma \cong \Z/3$. Let $V$ be a $\ol\Z_3[\Gamma]$-module such that $V_{\ol\Q_3}$ is an irreducible $2$-dimensional representation of $\Gamma$ and $V_{\ol\F_3}$ is a semisimple representation of $\Gamma$. Then $\sigma$ acts trivially on $V_{\ol\F_3}$, so $\rT^a(\sigma, V_{\ol\F_3}) \cong V_{\ol\F_3}$ as $\Gamma^\sigma$-representations, which is $2$-dimensional with trivial action. However, if $\zeta_3$ is a primitive cube root of unity and $V_i$ is the $\zeta_3^i$-eigenspace for $\sigma$ on $V_{\ol\Q_3}$, then $\chi_{\ol V_0} = 0$ and $\chi_{\ol V_1} = \chi_{\ol V_2}$ is the character of the $1$-dimensional trivial representation. On the other hand, there does exist a $\Gamma$-stable lattice $U$ in $V_{\ol\Q_3}$ such that $\rT^a(\sigma, U_{\ol\F_3})$ is $1$-dimensional and thus realizes the lower bound of Theorem~\ref{thm:eigenvalues-and-jordan-blocks-variant}.
    
    We are not aware of examples in which $\rT^a(\sigma, V_{\ol\F_\ell})$ admits an irreducible subquotient whose existence is not already implied by \eqref{eqn:Brauer-unramified-lower-bound}.
\end{remark}

\subsection{Shintani descent}\label{sss:shintani}
In this section, we show that Tate cohomology ``(partially) realizes the Frobenius twist of Shintani descent mod $\ell$''. This result will not be used in the remainder of this paper; we include it mainly because it generalizes (with a weaker conclusion) the later Corollary~\ref{cor:glauberman-correspondence} and is of independent interest.\footnote{If suitably extended to paraductive $\F_q$-group schemes, this result should also give rise to results on (small degree) base change functoriality. Since such results are not necessary for our purposes and may require some work to optimize, we do not pursue them here.}

We first recall some notation on Shintani descent, for which \cite{Kaw87} is a good reference. Let $\ol G$ be a connected linear algebraic $\F_q$-group, and let $m$ be a positive integer. Let $\sim_m$ denote the equivalence relation on $\ol G(\F_{q^m})$ induced by the twisted conjugation action $g \cdot h = gh\Fr_q(g)^{-1}$, and let $\sim$ denote the equivalence relation on $\ol G(\F_q)$ induced by conjugation. We define a map $n_m\co \ol G(\F_{q^m})/{\sim_m} \to \ol G(\F_q)/{\sim}$ by
\[
n_m(\alpha^{-1}\Fr_q(\alpha)) = \Fr_q^m(\alpha)\alpha^{-1}
\]
whenever $\alpha \in \ol G(\ol\F_q)$ satisfies $\alpha^{-1}\Fr_q(\alpha) \in \ol G(\F_{q^m})$; by Lang's theorem, this is enough to define $n_m$. The map $n_m$ is easily seen to be a (well-defined) bijection. The special case $m = 1$ is still of interest, and we write $t = n_1$.

As in \cite[1.2]{Kaw87}, for an element $x \in \ol G(\F_q)$ let $\ord(\ol x)$ denote the order of the image of $x$ in $\pi_0Z_{\ol G}(x)$, and let $M = M_{\ol G}$ be the least common multiple of $\ord(\ol x)$, as $x$ ranges over elements of $\ol G(\F_q)$. We will assume for simplicity that $m$ and $M$ are relatively prime. Let $r \in \Z$ be such that $rm \equiv 1 \pmod{M}$, and define $N_m\co \ol G(\F_{q^m})/\sim_m \to \ol G(\F_q)/\sim$ by $N_m = t^{-r} \circ n_m$; notably, \cite[(1.2.6)]{Kaw87} and \cite[Proposition 3.11]{Dig86b} show that
\begin{equation}\label{eqn:restriction-of-shintani-power}
N_m|_{\ol G(\F_q)} = [m]|_{\ol G(\F_q)}
\end{equation}
where $[m]\co \ol G(\F_q) \to \ol G(\F_q)$ is the conjugation-equivariant map $[m](x) = x^m$.

Let $\chi$ be the character of an irreducible $\ol\Q_\ell$-representation $V$ of $\ol G(\F_{q^m})$ whose isomorphism class is $\Fr_q$-stable. We will define a class function $\chi_0$ on $\ol G(\F_q)$ up to multiplication by an $m$th root of unity, called a \textit{Shintani descent} of $\chi$, as follows. First choose an extension $\widetilde\chi$ of $\chi$ to $\ol G(\F_{q^m}) \rtimes \langle\Fr_q\rangle$, where we regard $\Fr_q$ as an order $m$ automorphism of $\ol G(\F_{q^m})$. Note that $\wt\chi$ is unique up to multiplication by an $m$th root of unity. Define the class function $\chi_0$ on $\ol G(\F_q)$ by
\begin{equation}\label{eqn:shintani-descent-relation}
\chi_0(N_m(g)) = \widetilde{\chi}(g \rtimes \Fr_q).
\end{equation}
Observe that $\chi_0$ is only well-defined up to multiplication by an $m$th root of unity.

Now suppose that $m = \ell \neq p$ is a prime number. We are interested in the mod $\ell$ reduction of $\chi_0$, i.e., the restriction of $\chi_0$ to the elements of $\ol G(\F_q)$ of order prime to $\ell$. Let $\ol G(\F_q)_{\ell'} \subset \ol G(\F_q)$ denote the subset of elements of order prime to $\ell$.

\begin{prop}\label{prop:tate-cohomology-shintani-descent}
    Suppose $\ell \nmid M_{\ol G}$ as above. Let $U$ be an irreducible $\ol\F_\ell$-representation of $\ol G(\F_q)$ with Brauer character $\eta$, and suppose that $\eta$ occurs with nonzero coefficient in the expansion of $\chi_0 \circ [\ell]|_{\ol G(\F_q)_{\ell'}}$ in the basis of irreducible Brauer characters. Then $U$ is isomorphic to an irreducible subquotient of $\rT^a(\Fr_q, \ol V)$ for both $a \in \Z/2\Z$.
\end{prop}

\begin{proof}
    By \eqref{eqn:restriction-of-shintani-power}, we have $N_\ell|_{\ol G(\F_q)_{\ell'}} = [\ell]|_{\ol G(\F_q)_{\ell'}}$, so $N_\ell$ induces a bijection $\ol G(\F_q)_{\ell'}/{\sim} \to \ol G(\F_q)_{\ell'}/{\sim}$. If $V_0, \dots, V_{\ell-1}$ are the eigenspaces for the action of $\Fr_q$ on $V$ corresponding to the $\ell$th roots of unity $\zeta_0, \dots, \zeta_{\ell-1}$, then \eqref{eqn:shintani-descent-relation} shows that we have
    \begin{equation}\label{eqn:shintani-descent-relation-2}
    \chi_0 \circ [\ell] = \sum_{i=0}^{\ell-1} \zeta_i\chi_{\ol V_i}.
    \end{equation}
    For each $i$, write $\chi_{\ol V_i} = \sum_{j=1}^n c_{ij}\eta_j$ for $c_{ij} \in \Z_{\geq 0}$, where $\eta_1, \dots, \eta_n$ are the irreducible Brauer characters. Let $d_j = \max_{0 \leq i \leq \ell-1} c_{ij} - \min_{0 \leq i \leq \ell-1} c_{ij}$, so that
    \begin{equation}\label{eqn:shintani-descent-sup-inf}
    \sup_{0 \leq i \leq \ell-1}\{\chi_{\ol V_i}\} - \inf_{0 \leq i \leq \ell-1}\{\chi_{\ol V_i}\} = \sum_{j=1}^n d_j \eta_j.
    \end{equation}
    By \eqref{eqn:shintani-descent-relation-2}, we have
    \[
    \chi_0 \circ [\ell]|_{\ol G(\F_q)_{\ell'}} = \sum_{j=1}^{n} \left(\sum_{i=0}^{\ell-1} c_{ij}\zeta_i\right)\eta_j.
    \]
    For a fixed $j$, the sum $\sum_{i=0}^{\ell-1} c_{ij}\zeta_i$ vanishes if and only if $c_{ij}$ is independent of $i$, i.e., $d_j = 0$. Thus if $\eta_j$ occurs with nonzero coefficient in $\chi_0 \circ [\ell]$ then $d_j \neq 0$, so $\eta_j \leq \chi_{\rT^a(\Fr_q, \ol V)}$ by \eqref{eqn:shintani-descent-sup-inf} and Theorem~\ref{thm:eigenvalues-and-jordan-blocks-variant}.
\end{proof}


\subsection{Lusztig restriction}

In this subsection we use Theorem~\ref{thm:eigenvalues-and-jordan-blocks-variant} to show that Lusztig restriction provides a lower bound for Tate cohomology in a precise sense. Let $\ol G$ be a paraductive $\F_q$-group scheme (in the sense of Definition~\ref{def:paraductive}). We will use the notation and terminology of that section.

\begin{prop}\label{prop:tate-cohom-dl-restriction}
    Suppose that $[\ol G(\F_q): \ol G^\circ(\F_q) \cdot Z(\ol G)(\F_q)]$ is prime to $\ell$. Let $\ol T \subset \ol L \subset \ol G$ be twisted Levi subgroups such that $\ol T$ is a generalized maximal torus, and let $\theta\co \ol T(\F_q) \to \ol\Q_\ell^\times$ be a character of order prime to $\ell$. Let $\chi$ be an irreducible character of $\ol G(\F_q)$ which has nonzero pairing with $R_{\ol T}^{\ol G}(\theta)$, and let $s \in \ol T(\F_q)$ be an order $\ell$ element such that $Z_{\ol G^\circ}(s) \subset \ol L \subset Z_{\ol G}(s)$. 
    \begin{enumerate}
        \item For all $g \in \ol L(\F_q)$ of order prime to $\ell$ we have
        \[
        \chi(sg) = {}^*R^{\ol G}_{\ol L}(\chi)(g).
        \]
        \item If $\sigma$ is the $\F_q$-automorphism of $\ol G$ induced by $s$-conjugation and $\chi$ is defined over a finite extension of $\Q_\ell^{\unr}$ of degree prime to $\ell-1$, then the Brauer character of $\rT^i(\sigma, \chi)$ admits $\left|\ol{{}^*R^{\ol G}_{\ol L}(\chi)}\right|$ as a lower bound.
        \item If $\ell$ is a good prime for $\ol G^\circ$, then $\left|\ol{{}^*R^{\ol G}_{\ol L}(\chi)}\right| \neq 0$.
    \end{enumerate}
\end{prop}

\begin{proof}
    We may twist by a character and pass to a central quotient of $\ol G$ as usual to assume that $\ol G$ is of finite type and thus $\ol G(\F_q)$ is finite. Observe that since $s$ is of order $\ell$ and $g$ is a commuting element of order prime to $\ell$, it follows that $s$ is a power of the semisimple part $t$ of $sg$, and in particular $Z_{\ol G^\circ}(t) \subset \ol L^\circ$. By Lemma~\ref{lemma:dl-restriction-character}, it follows that
    \begin{equation}\label{eqn:char-equals-dl-restriction}
    \chi(sg) = {}^*R^{\ol G}_{\ol L}(\chi)(sg).
    \end{equation}
    Let $\eta$ be an irreducible character of $\ol L(\F_q)$ with nonzero pairing with ${}^*R^{\ol G}_{\ol L}(\chi)$. Let $\ol T'$ be a generalized maximal torus of $\ol L$, and let $\theta'\co \ol T'(\F_q) \to \ol\Q_\ell^\times$ be a character such that $\eta$ has nonzero pairing with $R_{\ol T'}^{\ol L}(\theta')$; such a pair $(\ol T', \theta')$ exists by Lemma~\ref{lemma:dl-7.5}. Proposition~\ref{prop:bonnafe-11.10} shows that $(\ol T', \theta')$ is geometrically conjugate to $(\ol T, \theta)$ when considered as pairs arising from $\ol G$, and the restrictions of $\theta$ and $\theta'$ to $Z(\ol G)(\F_q)$ are equal. Since $[\ol G(\F_q): \ol G^\circ(\F_q) \cdot Z(\ol G)(\F_q)]$ is prime to $\ell$ and $\theta$ is of order prime to $\ell$, it follows that $\theta'$ is of order prime to $\ell$. But $s$ is central in $\ol L(\F_q)$ and $R_{\ol T'}^{\ol L}(\theta')|_{Z(\ol L)(\F_q)} = \theta'|_{Z(\ol L)(\F_q)}$, so we have $\eta(sg) = \theta'(s)\eta(g) = \eta(g)$ for all $g \in \ol L(\F_q)$. Since this equality holds for every such $\eta$, we conclude that
    \[
    {}^*R^{\ol G}_{\ol L}(\chi)(sg) = {}^*R^{\ol G}_{\ol L}(\chi)(g),
    \]
    which combines with \eqref{eqn:char-equals-dl-restriction} to yield (1). Statement (2) follows directly from Theorem~\ref{thm:eigenvalues-and-jordan-blocks-variant}.

    For (3), note that our hypotheses imply that every irreducible constituent of $\chi|_{\ol G^\circ(\F_q)}$ is an irreducible constituent of the Deligne--Lusztig induction of some pair $(\ol T_0, \theta_0)$ corresponding to an element of the Deligne--Lusztig dual group of $\ol G^\circ$ which is of order prime to $\ell$. Thus the claim follows from \cite[Theorem~1.7]{CE99}.
\end{proof}

We next note that Tate cohomology carries cuspidal representations to cuspidal representations. The converse is not true in general, but it is true in an important special case; see Proposition~\ref{prop:glauberman-cuspidal}.

\begin{prop}\label{prop:tate-cohomology-cuspidal}
    Let $\ol G$ be a paraductive $\F_q$-group scheme equipped with an automorphism $\sigma$ of finite prime order $\ell \neq p$, let $\ol H = \ol G^\sigma$, and let $V$ be a finite-dimensional cuspidal $\ol\F_\ell$-representation of $\ol G(\F_q)$ whose isomorphism class is $\sigma$-stable. Then $\rT^i(\sigma, V)$ is a (possibly zero) cuspidal representation of $\ol H(\F_q)$.
\end{prop}

\begin{proof}
    By definition of cuspidality, we may assume that $\ol G$ is connected. The result then follows from \cite[Corollary~3.3.3]{DN25a} (which is stated for $p$-adic groups but holds with an identical proof for finite groups, as mentioned in the beginning of \cite[\S 3]{DN25a}).
\end{proof}

We note the following curious corollary, which may be of independent interest.

\begin{cor}\label{cor:dl-res-cuspidal}
    Let $\ol G$ be a connected reductive group over $\F_q$, let $\ell \neq p$ be a prime number, let $V$ be a cuspidal $\ol\Q_\ell$-representation of $\ol G(\F_q)$ defined over a finite extension of $\Q_\ell^{\unr}$ of degree prime to $\ell-1$ and lying in a prime-to-$\ell$ Lusztig series, and let $\ol H \subset \ol G$ be a twisted Levi $\F_q$-subgroup which is the centralizer of an element of $\ol G(\F_q)$ of order $\ell$. Every irreducible $\ol\F_\ell$-representation whose Brauer character occurs with nonzero coefficient in the $\ell$-modular reduction of ${}^*R^{\ol G}_{\ol H}(V)$ is cuspidal.
\end{cor}

\begin{proof}
Let $\sigma$ be the $\F_q$-automorphism of $\ol G$ induced by conjugation by an element of $\ol G(\F_q)$ of order $\ell$ whose centralizer is $\ol H$. Proposition~\ref{prop:tate-cohom-dl-restriction}(2) shows that the Brauer character of $\rT^i(\sigma,V)$ dominates $\left|\ol{{}^*R^{\ol G}_{\ol H}(V)}\right|$. Hence every irreducible representation which appears in the support of $\left|\ol{{}^*R^{\ol G}_{\ol H}(V)}\right|$ occurs in $\rT^i(\sigma,V)$, and Proposition~\ref{prop:tate-cohomology-cuspidal} shows that the latter is cuspidal.
\end{proof}

\begin{remark}\label{remark:non-cuspidal-dl-restriction}
    One reason that Corollary~\ref{cor:dl-res-cuspidal} is surprising is that, in the same setting, it can happen that the characteristic zero virtual representation $^*R_{\ol H}^{\ol G}(V)$ itself is nonzero and has no cuspidal constituents. To show this, we begin by summarizing some of the theory of unipotent $\ol\Q_\ell$-representations of finite symplectic groups, which is collected in a very readable form in \cite[Chapter~4]{GM20}. By \cite[Theorem~4.4.13]{GM20}, if $n$ is any positive integer then the set of irreducible unipotent $\ol\Q_\ell$-representations of $\Sp_{2n}(\F_q)$ is in natural bijection with the set of equivalence classes of ``symbols'' $S = \begin{pmatrix} X \\ Y \end{pmatrix} = \begin{pmatrix} x_1 < \cdots < x_r \\ y_1 < \cdots < y_s\end{pmatrix}$, where $x_i, y_j \in \Z_{\geq 0}$, satisfying the conditions that $r - s$ is odd and
    \[
    \sum_{i=1}^r x_i + \sum_{j=1}^s y_j = n + \left\lfloor\frac{(r+s-1)^2}{4} \right\rfloor.
    \]
    The equivalence relation on symbols is generated by two operations: namely, we say 
    \[
    \begin{pmatrix} X \\ Y \end{pmatrix} \sim \begin{pmatrix} Y \\ X \end{pmatrix}
    \]
    and
    \[
    \begin{pmatrix} x_1 < \cdots < x_r \\ y_1 < \cdots < y_s \end{pmatrix} \sim \begin{pmatrix} 0 < x_1 + 1 < \cdots < x_r + 1 \\ 0 < y_1 + 1 < \cdots < y_s + 1 \end{pmatrix}.
    \]
    By another theorem of Lusztig \cite[Theorem~4.4.28]{GM20}, for each $n$ there is at most one cuspidal unipotent representation of $\Sp_{2n}(\F_q)$. Moreover, a cuspidal unipotent representation exists if and only if $n = s(s+1)$ for some $s \in \Z_{> 0}$, in which case it corresponds to the equivalence class of the symbol $S = \begin{pmatrix} \\ 0&1&\cdots&2s \end{pmatrix}$.

    Now fix $n \geq 1$, and let $\ol S$ be an elliptic maximal $\F_q$-subtorus of $\Sp_2 \times 1 \subset \Sp_2 \times \Sp_{2n-2} \subset \Sp_{2n}$. If $\ol L = Z_{\Sp_{2n}}(\ol S)$, then $\ol L \cong \ol S \times \Sp_{2n-2}$. By \cite[Corollaire~11.11]{Bon06}, if $V$ is a unipotent representation of $\Sp_{2n}(\F_q)$ then every irreducible constituent of the Lusztig restriction $^*R^{\Sp_{2n}}_{\ol L}(V)$ is unipotent. By a theorem of Asai \cite[Theorem~4.6.9]{GM20}, if $n = s(s+1)$ for some $s \in \Z_{>0}$ and $V$ is moreover cuspidal, then the irreducible constituents of $^*R^{\Sp_{2n}}_{\ol L}(V)$ are precisely those corresponding to the equivalence classes of the symbols
    \[
    S_i = \begin{pmatrix}
        i-1 \\ 0&1&\cdots&i-1&i+1&\cdots&2s
    \end{pmatrix}
    \]
    for some $i \in \Z_{>0}$. More precisely, if $V_i$ is the unipotent representation of $\ol L(\F_q)$ corresponding to $S_i$, then we have
    \begin{equation}\label{eq:DLrest-remark}
    ^*R^{\Sp_{2n}}_{\ol L}(V) = \sum_{i=1}^{2s} (-1)^i V_i.
    \end{equation}
    In particular, $^*R^{\Sp_{2n}}_{\ol L}(V) \neq 0$. However, the group $\Sp_{2n-2}(\F_q)$ does not admit any cuspidal unipotent representations, so $^*R^{\Sp_{2n}}_{\ol L}(V)$ is nonzero and has no cuspidal constituents.
\end{remark}    

\begin{example}
    If $n = 2$ then $\Sp_2 \cong \SL_2$ and \eqref{eq:DLrest-remark} gives 
    \begin{equation}\label{eqn:dl-restriction-sp4}
    ^*R^{\Sp_4}_{\ol L}(V) = V_2 - V_1,
    \end{equation}
    where $V_2$ is the Steinberg representation and $V_1$ is the trivial representation. (The signs can be seen using the degree formula \cite[Proposition~4.4.15]{GM20}, noting that $\ol L(\F_q)$ only has two irreducible unipotent $\ol\Q_\ell$-representations.) Let's see why this does not contradict Corollary~\ref{cor:dl-res-cuspidal}. If $\ell$ divides the order of $\ol S(\F_q)$ and $\ol L(\F_q)$ is the centralizer of an element of order $\ell$, then $\ell$ is odd and it is well-known that the $\ell$-modular reduction of $V_2$ has two irreducible constituents, namely $\ol V_1$ and another cuspidal $\ol\F_\ell$-representation $U$. By \eqref{eqn:dl-restriction-sp4}, we have
    \[
    \ol{{}^*R^{\Sp_4}_{\ol L}(V)} = U,
    \]
    which is indeed a cuspidal $\ol\F_\ell$-representation (which is consistent with Corollary~\ref{cor:dl-res-cuspidal}).
\end{example}

\subsection{The Glauberman correspondence}\label{ss:glauberman}

For applications to Weil--Heisenberg representations, we need a sharper version of Theorem~\ref{thm:eigenvalues-and-jordan-blocks-variant} (under stronger hypotheses).

\subsubsection{Technical preliminaries} We begin with two simple lemmas.

\begin{lemma}\label{lemma:maximal-jordan-blocks}
    Let $V$ be a finite-dimensional vector space over a field $k$, let 
    \[
    \cF = (0 = V_0 \subset V_1 \subset \cdots \subset V_n = V)
    \]
    be a flag of $V$, let $d_i = \dim_k V_i/V_{i-1}$ for $1 \leq i \leq n$, and let $\sigma$ be an automorphism of $V$ stabilizing $\cF$ and acting trivially on each quotient $V_i/V_{i-1}$. If $d_{j_1} \leq \cdots \leq d_{j_n}$ and $d_{j_0} \coloneqq 0$, then $\dim \End_k(V)^\sigma \geq d_1^2 + \cdots + d_n^2$, with equality if and only if for each $0 \leq i \leq n - 1$, the unipotent automorphism $\sigma$ has $d_{j_{i+1}} - d_{j_i}$ Jordan blocks of size $n - i$.
\end{lemma}

\begin{proof}
    Let $P$ denote the parabolic $k$-subgroup of $\GL(V)$ corresponding to $\cF$, and let $U$ be the unipotent radical of $P$. Note that $\sigma \in U$, so we have
    \[
    \dim_k \End_k(V)^\sigma = \dim \GL(V)^\sigma \geq \dim P^\sigma \geq \dim P - \dim U = d_1^2 + \cdots + d_n^2,
    \]
    where the second inequality is an equality if and only if the $P$-orbit of $\sigma$ is open in $U$; thus equality can hold for elements in at most one $P$-orbit of $U$. It is elementary to check that if $\sigma$ has Jordan blocks as described, then $\GL(V)^\sigma = P^\sigma$ is of dimension $d_1^2 + \cdots + d_n^2$, and the lemma follows.
\end{proof}

\begin{lemma}\label{lemma:order-ell-special-fiber-filtration}
    Let $V$ be a finite free $\ol\Z_\ell$-module and let $\sigma$ be an automorphism of $V$ of order $\ell$.
    \begin{enumerate}
        \item Let $V_1$ and $V_2$ be two $\ol\Q_\ell$-subspaces of $V_{\ol\Q_\ell}$ on which $\sigma$ acts by a scalar, and let $W \coloneqq (V_1 + V_2) \cap V$. Then we have
        \[
        (\sigma_{\ol\F_\ell} - 1)\ol W  \subset \ol V_1 \cap \ol V_2.
        \]
        \item Let $V_1, \dots, V_\ell \subset V_{\ol\Q_\ell}$ be the eigenspaces for $\sigma_{\ol\Q_\ell}$ corresponding to $\ell$th roots of unity, of dimensions $d_1 \leq \cdots \leq d_\ell$. If the Jordan block structure for $\sigma_{\ol\F_\ell}$ is as in the equality case of Lemma~\ref{lemma:maximal-jordan-blocks}, and $U_n \coloneqq (\bigoplus_{i=1}^n V_i) \cap V$ for all $0 \leq n \leq \ell$, then
        \[
        (\sigma_{\ol\F_\ell} - 1)\ol U_n = \ol U_{n-1}.
        \]
    \end{enumerate}
\end{lemma}

\begin{proof}
    We begin with (1). Let $x \in W$, so we may write $x = x_1 + x_2$ with $x_i \in V_i$. Let $\zeta_1, \zeta_2$ be the scalars by which $\sigma$ acts on $V_i$, so $\zeta_1^\ell = \zeta_2^\ell = 1$. Since $(\sigma - \zeta_1)x_1 = 0$ and $(\sigma-\zeta_2)x_2 = 0$, we have
    \[
    (\sigma-\zeta_2) x = (\sigma-\zeta_2)x_1  \in V_1 \quad \text{and} \quad (\sigma-\zeta_1)x = (\sigma-\zeta_1) x_2 \in V_2.
    \]
    Since $\zeta_i \in \ol \Z_\ell^\times$ and $\sigma$ preserves $V$, the left sides of the above equations lie in $V$, hence the right sides do as well. Hence we may reduce both equations over $\ol \F_\ell$, and upon so doing we obtain $(\sigma_{\ol\F_\ell} - 1) \ol x \in \ol V_1 \cap \ol V_2$ because $\zeta_i \equiv 1 \pmod{\ell}$ for both $i$. 
    
    For (2), observe that $(\sigma_{\ol\F_\ell} - 1)(\ol U_2) \subset \ol U_1$ by (1), and the inclusion is an equality by the structure of Jordan blocks. Passing from $V$ to $V/V_1$, we conclude (2) by induction.
\end{proof}

\subsubsection{Tate cohomology realizes the Glauberman correspondence}  In \cite[Corollary~8]{Gla68}, Glauberman established the celebrated \emph{Glauberman correspondence}, which shows that if $S$ and $\Gamma$ are finite groups of relatively prime orders such that $S$ is solvable and $S$ acts on $\Gamma$, then there is a canonical one-to-one correspondence between irreducible $\CC[\Gamma]$-representations with $S$-stable isomorphism class and irreducible $\Gamma^S$-representations. If $S = \langle \sigma \rangle$ is cyclic, then by \cite[Theorem~3]{Gla68} this correspondence is uniquely characterized by the condition that it sends an $S$-stable character $\chi$ of $\Gamma$ to a character $\lambda$ of $\Gamma^S$ such that there exists $\epsilon \in \{\pm 1\}$ and an extension $\wt\chi$ of $\chi$ to an irreducible character $\wt\chi$ of $\Gamma \rtimes S$ such that $\wt\chi(1 \rtimes \sigma) \in \Z$ and
\begin{equation}\label{eqn:glauberman-condition}
    \wt\chi(t \rtimes \sigma) = \epsilon\lambda(t) \text{ for all } t \in \Gamma^S.
\end{equation}
Observe the similarity between \eqref{eqn:glauberman-condition} and \eqref{eqn:shintani-descent-relation} in the case that $\Gamma = \ol G(\F_{q^m})$ for a linear algebraic $\F_q$-group $\ol G$ and $\sigma$ acts by $\Fr_q$; we will discuss this further in Remark~\ref{remark:glauberman-shintani}.

The following property will be recorded but not used in this paper; we define it because it is the key condition which guarantees good behavior of cup products, which will be important in \cite{CF26b}.

\begin{defn}\label{def:extremal}
    We will say that a finitely generated $k[\sigma]$-module $V$ is \emph{minimal} if 
    \[
    \text{$V \cong k^{\oplus m} \oplus k[\sigma]^{\oplus n}$ as $k[\sigma]$-modules for some $m, n \geq 0$;}
    \]
    similarly, $V$ is \emph{maximal} if 
    \[
    \text{$V \cong (k[\sigma]/((\sigma-1)^{\ell-1}))^{\oplus m} \oplus k[\sigma]^{\oplus n}$ for some $m, n \geq 0$.}
    \]
    If $V$ is either maximal or minimal, then we will say that $V$ is \emph{extremal}.
\end{defn}

\begin{thm}\label{theorem:eigenvalues-and-jordan-blocks}
    Let $\Delta \subset \Gamma$ be finite groups, let $\ell$ be a prime number, let $\sigma$ be an automorphism of $\Gamma$ of order $\ell$ which preserves $\Delta$, and let $(\pi, V)$ be a finitely generated $\ol\Z_\ell[\Gamma \rtimes \langle\sigma\rangle]$-module such that $V$ is a projective $\ol\Z_\ell[\Delta]$-module and $\ol V \coloneqq V \otimes_{\ol\Z_\ell} \ol\F_\ell$ is a simple $\ol\F_\ell[\Delta]$-module. Choose an ordering $\xi_1, \dots, \xi_\ell$ of the $\ell$th roots of unity in $\ol\Z_\ell$ such that the dimensions $d_i$ of the $\xi_i$-eigenspaces $V_i$ of $\pi(\sigma)_{\ol\Q_\ell}$ on $V_{\ol\Q_\ell}$ satisfy $d_1 \leq d_2 \leq \cdots \leq d_\ell$, and let $d_0 = 0$.
    \begin{enumerate}
        \item For $0 \leq i \leq \ell - 1$, the number of Jordan blocks of size $\ell - i$ for $\pi(\sigma)_{\ol\F_\ell}$ is $d_{i+1} - d_i$.
        \item Let $\ol V_i \coloneqq (V_i \cap V)\otimes_{\ol \Z_\ell} \ol \F_\ell$. Then $\ol V_j \subset \ol V_{j+1}$ for all $0 \leq j < \ell - 1$ and
        \begin{equation}\label{eq:glauberman-theorem-2}
        \rT^i(\pi(\sigma)_{\ol\F_\ell}, \ol V) \cong \ol V_\ell/\ol V_1
        \end{equation}
        as $\ol\F_\ell[\Gamma^\sigma]$-modules for $i \in \Z/2\Z$.
        \item If $\Tr_{V_{\ol\Q_\ell}}(\pi(\sigma)) \in \Z$, then $\ol V$ is extremal as an $\ol\F_\ell[\pi(\sigma)]$-module (i.e., either $d_1 = d_{\ell-1}$ or $d_2 = d_\ell$) and there exists $\epsilon \in \{\pm 1\}$ such that for all $\gamma \in \Gamma^\sigma_{\ell'}$ and $i \in \Z/2\Z$ we have
        \begin{equation}\label{eqn:trace-and-tate-cohomology}
        \chi_{\rT^i(\pi(\sigma)_{\ol\F_\ell}, \ol V)}(\gamma) = \epsilon \Tr_{V_{\ol\Q_\ell}}(\gamma \rtimes \sigma)
        \end{equation}
        If $\Tr_{V_{\ol\Q_\ell}}(\pi(\sigma)) > 0$, then $\ol V$ is minimal and $\epsilon = 1$ in \eqref{eqn:trace-and-tate-cohomology}; if instead $\Tr_{V_{\ol\Q_\ell}}(\pi(\sigma)) < 0$, then $\ol V$ is maximal and $\epsilon = -1$.\footnote{Note that if $\ell = 2$, then $\ol V$ is both minimal and maximal.}
    \end{enumerate}
\end{thm}

\begin{proof}
    Since $\ol V$ is a simple $\ol\F_\ell[\Delta]$-module, the natural map $\ol\F_\ell[\Delta] \to \End_{\ol\F_\ell}(\ol V)$ is surjective. Since $V$ is a finite $\ol\Z_\ell$-module, the map $\ol\Z_\ell[\Delta] \to \End_{\ol\Z_\ell}(V)$ is also surjective by Nakayama's lemma, and since $V$ is a projective $\ol\Z_\ell[\Delta]$-module it follows that there is a splitting
    \begin{equation}\label{eqn:algebra-decomp}
    \ol\Z_\ell[\Delta] \cong \End_{\ol\Z_\ell}(V) \oplus A
    \end{equation}
    for some $\ol\Z_\ell$-algebra $A$. Since $\pi$ extends to a representation of $\Delta \rtimes \langle\sigma\rangle$, we see that $\sigma$ stabilizes the factors in the decomposition \eqref{eqn:algebra-decomp}.

    If $k \in \{\ol\F_\ell, \ol\Q_\ell\}$ and $V_k \coloneqq V \otimes_{\ol \Z_\ell}k$, then we have
    \begin{equation}\label{eqn:group-algebra-decomposition}
    k[\Delta]^\sigma \cong \End_k(V_k)^\sigma \oplus (A \otimes_{\ol\Z_\ell} k)^\sigma.
    \end{equation}
    Observe that $\dim_k k[\Delta]^\sigma$ is independent of the choice of $k$, because it can be computed as the set of sums $\sum_{\delta \in \Delta} c_\delta [\delta]$, where $c_\delta \in k$ satisfies $c_{\sigma(\delta)} = c_\delta$ for all $\delta \in \Delta$. In general, one has
    \[
    \dim_{\ol\Q_\ell} \End_{\ol\Q_\ell}(V_{\ol\Q_\ell})^\sigma \leq \dim_{\ol\F_\ell}\End_{\ol\F_\ell}(V_{\ol\F_\ell})^\sigma
    \]
    and similarly $\dim_{\ol\Q_\ell}(A \otimes_{\ol\Z_\ell} \ol\Q_\ell)^\sigma \leq \dim_{\ol\F_\ell}(A \otimes_{\ol\Z_\ell} \ol\F_\ell)^\sigma$, so by \eqref{eqn:group-algebra-decomposition} it follows that $\dim \End_k(V_k)^\sigma$ is independent of $k$.
    
    Now we conclude (1). Observe that $\pi(\sigma)$ preserves the flag $V_1 \subset V_1 \oplus V_2 \subset \cdots \subset V_{\ol\Q_\ell}$ of $\ol\Q_\ell$-vector spaces. If $\Lambda_i = V \cap (\bigoplus_{j=1}^i V_i)$, then $\pi(\sigma)$ preserves the flag $0 = \Lambda_0 \subset \Lambda_1 \subset \cdots \subset \Lambda_\ell = V$ of $\ol\Z_\ell$-modules. Hence $\pi(\sigma)_{\ol\F_\ell}$ preserves the flag $0 = \ol\Lambda_0 \subset \ol\Lambda_1 \subset \cdots \subset \ol\Lambda_\ell = V_{\ol\F_\ell}$, and by Lemma~\ref{lemma:maximal-jordan-blocks} it follows that
    \[
    \dim (\End_{\ol\F_\ell}(V_{\ol\F_\ell})^\sigma) \geq d_1^2 + \cdots + d_\ell^2,
    \]
    with equality if and only if the number of Jordan blocks of size $\ell - i$ for $\pi(\sigma)_{\ol\F_\ell}$ is equal to $d_{i+1} - d_i$. Since on the other hand $\End_{\ol\Q_\ell}(V_{\ol\Q_\ell})^\sigma = \prod_{i=1}^\ell \End_{\ol\Q_\ell}(V_i)$, we find
    \[
    \dim (\End_{\ol\Q_\ell}(V_{\ol\Q_\ell})^\sigma) = d_1^2 + \cdots + d_\ell^2,
    \]
    and (1) follows by Lemma~\ref{lemma:maximal-jordan-blocks} again.

    For (2), observe that $\dim_{\ol\F_\ell} \ol V^\sigma = d_\ell$ by (1). Since $\ol V_\ell \subset \ol V^\sigma$, we conclude that $\ol V_\ell = \ol V^\sigma$ and in particular $\ol V_1 \subset \ol V_\ell$. Lemma~\ref{lemma:order-ell-special-fiber-filtration}(2) shows that $(\pi(\sigma)_{\ol\F_\ell} - 1)^{\ell-1}(\ol V) = \ol V_1$, so \eqref{eq:glauberman-theorem-2} holds for $i = 0$ by definition and Lemma~\ref{lem:N_sigma-identity}. The case $i = 1$ is similar, and we leave it to the reader.\footnote{In fact, the case $i = 1$ in (2) follows from the case $i = 0$ in (2) and (3), since the latter statements and \cite[Theorem 3]{Gla68} imply that $\ol V_\ell/\ol V_1$ is an irreducible $\Delta$- (and hence $\Gamma$-)representation, and Lemma~\ref{lem:semisimplifications-Tate-cohomology-coincide} implies that $\rT^1$ and $\rT^0$ have isomorphic semisimplifications.}
    
    Finally, for (3), let $n = \Tr_{V \otimes_{\ol\Z_\ell} \ol\Q_\ell}(\pi(\sigma))$, fix a primitive $\ell$th root of unity $\zeta \in \ol\Z_\ell$, and for $0 \leq i \leq \ell - 1$ let $\lambda_i$ denote the character of $\Gamma^\sigma \times \langle\sigma\rangle$ on the $\zeta^i$-eigenspace of $\pi(\sigma)$ on $V \otimes_{\ol\Z_\ell} \ol\Q_\ell$. For $\gamma \in \Gamma^\sigma_{\ell'}$ we have then
    \begin{equation}\label{eqn:brauer-char-1}
    \Tr_{V \otimes_{\ol\Z_\ell} \ol\Q_\ell}(\gamma \rtimes \sigma) = \sum_{i=0}^{\ell - 1} \zeta^i \lambda_i(\gamma),
    \end{equation}
    Taking $\gamma = 1$, the assumption that \eqref{eqn:brauer-char-1} lies in $\Z$ implies that 
    \begin{equation}\label{eq:dimension-equalities}
    m_1 = m_2 = \cdots = m_{\ell-1} = m_0 - n.
    \end{equation}
    If $n \geq 0$, then it follows that $m_0 \geq m_i$ for all $i \geq 0$, and by Lemma~\ref{lemma:nonzero-tate-cohomology} and (1) it follows that $\ol V$ is minimal and that $\dim_{\ol\F_\ell} \rT^i(\pi(\sigma), \ol V) = m_0 - m_1 = n$. If instead $n \leq 0$, then we have $m_0 \leq m_i$ for all $i \geq 0$, and by the same reasoning it follows that $\ol V$ is maximal and $\dim_{\ol\F_\ell} \rT^i(\pi(\sigma), \ol V) = m_1 - m_0 = -n$. Thus in either case we see that $\ol V$ is extremal and $\dim_{\ol\F_\ell} \rT^i(\pi(\sigma), \ol V) = |n|$.

    By part (2), \eqref{eq:dimension-equalities} implies that the $\lambda_1 = \lambda_2 = \ldots = \lambda_{\ell-1}$ agree on $\Gamma^\sigma_{\ell'}$. Hence, using the above observations, we find that
    \begin{equation}\label{eqn:brauer-char-2}
    \sum_{i=1}^{\ell-1} \zeta^i \lambda_i(\gamma) = \lambda_0(\gamma) - \lambda_1(\gamma) = \epsilon \chi_{\rT^i(\pi(\sigma), V \otimes_{\ol\Z_\ell} \ol\F_\ell)}(\gamma)
    \end{equation}
    where $\epsilon = 1$ if $m_0 > m_1$ and $\epsilon = -1$ if $m_0 < m_1$ (and $\epsilon$ may be chosen arbitrarily if $m_0 = m_1$); note that the second equality follows from (2). Combining \eqref{eqn:brauer-char-1} and \eqref{eqn:brauer-char-2} gives \eqref{eqn:trace-and-tate-cohomology}, as desired.
\end{proof}

The Glauberman correspondence was proven in \cite{Gla68} in an essentially character-theoretic manner. The following immediate corollary of Theorem~\ref{theorem:eigenvalues-and-jordan-blocks} shows that if $S$ is cyclic of order $\ell$, then Tate cohomology realizes the Glauberman correspondence in characteristic $\ell$. We remark that, aside from the extremality claim, the same result was observed using different language for $i = 0$ in \cite{Alp76} and \cite{Dade78}; the $0$th Tate cohomology group is one example of the \emph{Brauer construction} in modular representation theory.

\begin{cor}\label{cor:glauberman-correspondence}
    Let $\Gamma$ be a finite group, and let $S = \langle\sigma\rangle$ be a cyclic group of prime order $\ell$ not dividing the order of $\Gamma$. Let $V$ be a $\ol\Z_\ell[\Gamma \rtimes S]$-module such that $V_{\ol \Q_\ell}$ is irreducible over $\ol\Q_\ell[\Gamma]$.  Let $U$ be a $\ol\Z_\ell[\Gamma^\sigma]$-module with irreducible $\ol\Q_\ell$-fiber such that $U_{\ol \Q_\ell}$ corresponds to $V$ under the Glauberman correspondence.    Then
    \[
    \rT^i(\sigma, \ol V) \cong \ol U \quad \text{for both $i \in \Z/2\Z$.}
    \]
\end{cor}

\begin{proof}
    Note that $U$ is a projective $\ol\Z_\ell[\Gamma]$-module and $U \otimes_{\ol\Z_\ell} \ol\F_\ell$ is a simple $\ol\F_\ell[\Gamma]$-module since $\ell$ does not divide the order of $\Gamma$; see \cite[\S15.5, Proposition~43]{Serre77}. If $\chi$ and $\lambda$ are the characters of $\Gamma \rtimes \langle\sigma\rangle$ and $\Gamma^\sigma$ on $V_{\ol \Q_\ell}$ and $U_{\ol \Q_\ell}$, respectively, then from \eqref{eqn:glauberman-condition} we see in particular that $\chi(\sigma) \in \Z$, so Theorem~\ref{theorem:eigenvalues-and-jordan-blocks}(3) yields the result.
\end{proof}

\begin{remark}
    Theorem~\ref{theorem:eigenvalues-and-jordan-blocks} is only interesting for our purposes when $\sigma|_\Delta$ is an outer automorphism; if $\sigma|_\Delta$ is given by conjugation through an element $s \in \Delta$, then it simply says that $\rT^i(\sigma_{\ol\F_\ell}, \ol V) = 0$. Indeed, if $V$ is a projective $\ol\Z_\ell[\Delta]$-module, then the character of $V \otimes_{\ol\Z_\ell} \ol\Q_\ell$ takes the value $0$ on $s$ by \cite[\S 16.2, Theorem~36]{Serre77}. However, the vanishing of $\rT^i(\sigma, \ol V)$ is obvious a priori because $\ol V$ is a projective $\ol\F_\ell[\sigma]$-module.
    
    It would be interesting to know to what extent the projectivity and simplicity assumptions can be weakened. These assumptions cannot be removed completely, however, as Remark~\ref{remark:strict-inequality} shows.
\end{remark}

\subsubsection{An extension to some infinite groups}\label{sss:infinite-glauberman}

Below, we will want to apply a version of the Glauberman correspondence to representations of groups arising from the $\F_q$-group schemes considered in \S\ref{ss:ns-dl-reps}, which are often infinite. To this end, we explain how to extend the Glauberman correspondence slightly past the case of finite groups.

\begin{hypothesis}\label{hypothesis:infinite-glauberman}
    Throughout this section, let $\Gamma$ be a group, and suppose that there exists a central subgroup $Z \subset \Gamma$ and a finite normal subgroup $\Gamma_0 \subset \Gamma$ such that $Z$ is of finite index in $\Gamma$ and $\Gamma/\Gamma_0$ is abelian. Let $\ell$ be a prime number not dividing the order of $\Gamma_0 \times (\Gamma/Z)$, nor the order of any element in $\Gamma/\Gamma_0$. Let $\sigma$ be an automorphism of $\Gamma$ of prime order $\ell$ which stabilizes $\Gamma_0$ and acts trivially on $\Gamma/\Gamma_0$.
\end{hypothesis}

Note that since $\sigma$ is of order prime to the finite group $\Gamma_0$ and $\sigma$ acts trivially on $\Gamma/\Gamma_0$, the maps $\Gamma^\sigma \to \Gamma/\Gamma_0$ and $Z^\sigma \to Z/(Z \cap \Gamma_0)$ are surjective. In particular, we may in practice pass from $Z$ to $\bigcap_{i=0}^{\ell-1} \sigma^i(Z)$ to assume that $Z$ is $\sigma$-stable.

\begin{example}
The key example to keep in mind is the following: let $\ol{G}$ be a paraductive $\F_q$-group scheme. For all sufficiently large $\ell$, we may take $\Gamma = \ol{G}(\F_{q^\ell})$, $\Gamma_0 = \ol{G}^\circ(\F_{q^\ell})$, $Z = Z(\ol G)(\F_{q^\ell})$, and $\sigma = \Fr_q$. (See Proposition~\ref{prop:banal-primes}(4) for an important special case of this statement, with precise bounds on $\ell$.)
\end{example}

\begin{lemma}\label{lemma:infinite-glauberman}
    There is a unique bijection $\chi \mapsto \wt\chi$ from the set of irreducible $\ol\Q_\ell$-valued characters of $\Gamma^\sigma$ to the set of $\sigma$-stable $\ol\Q_\ell$-valued irreducible characters of $\Gamma$, with the property that there is some $\epsilon \in \{\pm 1\}$ such that
    \[
    \chi(\gamma) = \epsilon\wt\chi(\gamma \rtimes \sigma),
    \]
    for all $\gamma \in \Gamma^\sigma$, where we also use $\wt\chi$ to denote the unique extension of $\wt\chi$ to a character of $\Gamma \rtimes \langle\sigma\rangle$ satisfying $\wt\chi(\sigma) \in \Z$.
\end{lemma}

\begin{proof}
    As above, we may and do assume that $\sigma$ stabilizes $Z$. Let $\chi$ be an irreducible $\ol\Q_\ell$-valued character of $\Gamma^\sigma$. Let $n$ be the order of $Z \cap \Gamma_0$, so $\ell \nmid n$ by hypothesis. Since $\Gamma/\Gamma_0$ is abelian, there exists a character $\theta\colon \Gamma/\Gamma_0 \to \ol\Q_\ell^\times$ such that the central character of $\chi\cdot\theta|_{\Gamma^\sigma}$ is killed by an integer $m$ which is divisible by $n$ and not divisible by $\ell$. In this case, $\chi \cdot \theta|_{\Gamma^\sigma}$ factors through $\Gamma^\sigma/Z^m$. Since every finite order element of $Z$ is of order prime to $\ell$, we may pass to a prime-to-$\ell$ multiple of $m$ to assume that the map $Z^m \to \Gamma/\Gamma_0$ is a $\sigma$-equivariant injection. Thus after passing to a prime-to-$\ell$ multiple of $m$ we may by hypothesis assume that $\sigma$ acts trivially on $Z^m$, and $(\Gamma/Z^m)^\sigma \subset \Gamma/Z^m$ is of order prime to $\ell$ by hypothesis (since $\ell \nmid m$), so we have $(\Gamma/Z^m)^\sigma \cong \Gamma^\sigma/Z^m$. By the Glauberman correspondence for finite groups, we obtain an irreducible character $\eta$ of $\Gamma/Z^m \rtimes \langle\sigma\rangle$ and $\epsilon \in \{\pm 1\}$ satisfying
    \[
    \chi(\gamma) \cdot \theta(\gamma) = \epsilon\eta(\gamma\rtimes\sigma)
    \]
    for all $\gamma \in \Gamma^\sigma$. We define then $\wt\chi(\gamma\rtimes\sigma^n) = \eta(\gamma\rtimes\sigma^n)\theta(\gamma)^{-1}$ for all $\gamma \in \Gamma$. If $\gamma \in \Gamma^\sigma$, then we have
    \[
    \chi(\gamma) = (\chi(\gamma) \cdot \theta(\gamma)) \cdot \theta(\gamma)^{-1} = \epsilon\eta(\gamma\rtimes\sigma) \cdot \theta(\gamma)^{-1} = \epsilon\wt\chi(\gamma \rtimes\sigma).
    \]
    Note that $\wt\chi$ is an irreducible character of $\Gamma \rtimes\langle\sigma\rangle$ because $\eta$ is an irreducible character and $\theta$ is a $\sigma$-stable character (since $\sigma$ acts trivially on $\Gamma/\Gamma_0$ by hypothesis). The reverse construction is completely similar and will be left to the reader, as will the verification that these constructions define inverse bijections.
\end{proof}

By extension from the usual terminology, we will call the bijection from Lemma~\ref{lemma:infinite-glauberman} the \textit{Glauberman correspondence}. We extend this bijection by linearity to a homomorphism
\[
\Gla\co K_0(\Rep_{\ol\Q_\ell}(\Gamma^\sigma)) \to K_0(\Rep_{\ol\Q_\ell}(\Gamma)),
\]
which we will also call the Glauberman correspondence. We will similarly refer to the induced map $K_0(\Rep_{\ol\F_\ell}(\Gamma^\sigma)) \to K_0(\Rep_{\ol\F_\ell}(\Gamma))$ as the Glauberman correspondence. It seems likely that this is the same as the map $\operatorname{BC}_\ell$ from \cite[\S 2]{Cot26a} (see \cite[Remark (1.3.2) (ii)]{Kaw87}), but we do not know or check this.

\begin{remark}\label{rmk:glauberman-twisting-equivariance}
We record one observation from the proof of Lemma~\ref{lemma:infinite-glauberman}. If $\chi$ is an irreducible character of $\Gamma^\sigma$ and $\theta$ is a $\sigma$-stable character of $\Gamma/\Gamma_0$, then
\[
\Gla(\chi\theta|_{\Gamma^\sigma})=\Gla(\chi)\theta
\]
Moreover, if $m$ is prime to $\ell$ and divisible by $|Z\cap\Gamma_0|$, and if $\chi$ factors through $\Gamma^\sigma/Z^m$, then $\widetilde\chi$ factors through $\Gamma/Z^{m'}$ for some prime-to-$\ell$ multiple $m'$ of $m$ for which $\Gamma/Z^{m'}$.
\end{remark}

\begin{lemma}\label{lemma:tate-cohom-infinite-glauberman}
    Let $V$ be a $\ol\F_\ell[\Gamma \rtimes \langle\sigma\rangle]$-module which is irreducible as an $\ol\F_\ell[\Gamma]$-module. Let $U$ be an $\ol\F_\ell[\Gamma^\sigma]$-module corresponding to $V$ under the Glauberman correspondence. Then
    \[
    \rT^i(\sigma, V) \cong U \quad \text{for both $i \in \Z/2\Z$}.
    \]
\end{lemma}

\begin{proof}
    By Remark~\ref{rmk:glauberman-twisting-equivariance}, we may twist $U$ and $V$ by an $\ol\F_\ell^\times$-valued character of $\Gamma/\Gamma_0$ to assume that there is some $m$ not divisible by $\ell$ such that $Z^m$ acts trivially on $U$ and $V$. In this case, the result follows from Corollary~\ref{cor:glauberman-correspondence}.
\end{proof}

The following lemma will be applied in the setting of parabolic induction for paraductive group schemes.

\begin{cor}\label{cor:induction-commutes-sometimes}
    Let $\Delta \subset \Gamma$ be a $\sigma$-stable subgroup, let $V$ be an $\ol\F_\ell[\Delta \rtimes \langle\sigma\rangle]$-module of finite dimension over $\ol\F_\ell$, let $U$ be the $\ol\F_\ell[\Delta]^\sigma$-module corresponding to $V$ under the Glauberman correspondence, and suppose
    \[
    \langle \ind_\Delta^\Gamma(V), \ind_\Delta^\Gamma(V) \rangle = \langle \ind_{\Delta^\sigma}^{\Gamma^\sigma}(U), \ind_{\Delta^\sigma}^{\Gamma^\sigma}(U)\rangle
    \]
    Then
    \[
    \ind_\Delta^\Gamma(V)^{\ss} = \bigoplus_{i=1}^m V_i^{e_i}
    \]
    as an $\ol\F_\ell[\Gamma]$-module, where the $V_i$ are pairwise non-isomorphic simple $\ol\F_\ell[\Gamma]$-modules such that $V_i \cong {}^\sigma V_i$. If $U_i$ is the $\ol\F_\ell[\Gamma^\sigma]$-module corresponding to $V_i$ under the Glauberman correspondence, then
    \[
    \ind_{\Delta^\sigma}^{\Gamma^\sigma}(U)^{\ss} \cong \bigoplus_{i=1}^m U_i^{e_i}.
    \]
\end{cor}

\begin{proof}
    This is immediate from Lemma~\ref{lemma:tate-cohom-devissage}, Lemma~\ref{lemma:tate-cohom-infinite-glauberman}, and \cite[Proposition~3.3]{TV}. 
\end{proof}

\begin{remark}\label{remark:isaacs-navarro}
    Observe that if $\ind_\Delta^\Gamma(V)$ is irreducible, then Corollary~\ref{cor:induction-commutes-sometimes} says that $\ind_\Delta^\Gamma(V)$ corresponds to $\ind_{\Delta^\sigma}^{\Gamma^\sigma}(U)$ under the Glauberman correspondence. This is precisely \cite[Theorem~A (a)]{IN91} in the case that the group $S$ of \textit{loc.\ cit.}\ is solvable.
\end{remark}

\subsubsection{Fields of definition} The representation $V$ in Theorem~\ref{thm:eigenvalues-and-jordan-blocks-variant} will in practice be obtained as a lattice inside a representation of a group $\Gamma \cong \ol G(\F_{q^\ell})$ as above, and the latter will arise from the Glauberman correspondence with respect to a generator of $\Gal(\F_{q^\ell}/\F_q)$. In order to check the hypotheses of Theorem~\ref{thm:eigenvalues-and-jordan-blocks-variant} in practice, we will use the following two lemmas.

\begin{lemma}\label{lemma:integer-valued-characters}
    Let $\widetilde{\chi}$ be an irreducible $\sigma$-stable character of $\Gamma$, and let $\chi$ be the irreducible character of $\Gamma^\sigma$ corresponding to $\widetilde{\chi}$ under the Glauberman correspondence.
    \begin{enumerate}
        \item If $\chi$ takes values in a number field $K$, then so does $\widetilde{\chi}$.
        \item If $\ell'$ is a prime number not dividing the order of $\Gamma^\sigma$ and $\gamma \in \Gamma$ is of order $\ell'$, then $\widetilde{\chi}(\gamma) \in \Z$.
    \end{enumerate}
\end{lemma}

\begin{proof}
    As observed following Hypothesis~\ref{hypothesis:infinite-glauberman}, the map $\Gamma^\sigma \to \Gamma/\Gamma_0$ is surjective. Thus by Remark~\ref{rmk:glauberman-twisting-equivariance} we may pass from $\chi$ to $\chi \otimes \theta$ for some character $\theta$ of $\Gamma/\Gamma_0$ to assume that $\chi$ factors through a finite quotient of $\Gamma$ of order prime to $\ell$; we may then pass to such a quotient to assume that $\Gamma$ is finite. Observe that if $s \in \Gal(\ol\Q/\Q)$ then by uniqueness and the relation \eqref{eqn:glauberman-condition}, if $\widetilde{\chi^s}$ is the irreducible character of $\Gamma$ corresponding to the twist $\chi^s$ under the Glauberman correspondence, then $\widetilde{\chi^s} = (\widetilde{\chi})^s$. Taking $s \in \Gal(\ol K/K)$ shows (1). For (2), observe that $\chi$ is valued in $\Q(\mu_{|\Gamma^\sigma|})$, while $\widetilde\chi(\gamma) \in \Q(\mu_{\ell'})$. Since $\Q(\mu_{|\Gamma^\sigma|}) \cap \Q(\mu_{\ell'}) = \Q$, (1) shows that $\widetilde\chi(\gamma) \in \Q$ and hence $\widetilde\chi(\gamma) \in \Z$ by \cite[\S6.5, Proposition~15]{Serre77}.
\end{proof}

\begin{lemma}\label{lemma:character-values-irreducible-factors}
    Let $K/\Q_\ell^{\mathrm{unr}}$ be a finite extension, and let $V$ be a finite-dimensional representation of $\Gamma$ over $K$. If the character of each irreducible factor of $V_{\ol K}$ takes values in $K$, then each irreducible factor of $V$ is absolutely irreducible.
\end{lemma}

\begin{proof}
    Let $V= \bigoplus_{i=1}^m V_i^{\oplus n_i}$ denote the isotypic decomposition. By Schur's lemma, we have $\End_{K[\Gamma]}(V) \cong \prod_{i=1}^m \Mat_{n_i}(D_i)$, where $D_i = \End_{K[\Gamma]}(V_i)$ is a division algebra which is finite-dimensional over its center $K_i$, itself a finite extension of $K$. Note that the Brauer group of each $K_i$ is trivial by class field theory, so $D_i = K_i$ for all $i$. By \cite[\S12.2, Proposition~35]{Serre77} (which applies after twisting $V$ by a character of $\Gamma$, as usual), it follows that each $V_i$ is absolutely irreducible.
\end{proof}

\subsubsection{Modular reduction}

Later (in \cite{CF26a}), we will need the following lemma, which is a mild extension of \cite[Part~III, no.\ 15.5, Proposition~43]{Serre77}.

\begin{lemma}\label{lemma:infinite-glauberman-modular-reduction}
    Let $\ell$ be a prime number, and let $\Gamma$ be a group with a finite normal subgroup $\Gamma_0 \subset \Gamma$ satisfying Hypothesis~\ref{hypothesis:infinite-glauberman}. Let $\rho_1$ and $\rho_2$ be two $\ol\Z_\ell$-representations of $\Gamma$ such that
    \begin{enumerate}
        \item $(\rho_1)_{\ol\Q_\ell}$ and $(\rho_2)_{\ol\Q_\ell}$ are irreducible,
        \item $(\rho_1)_{\ol\F_\ell}$ and $(\rho_2)_{\ol\F_\ell}$ are isomorphic.
    \end{enumerate}
    Then $(\rho_1)_{\ol\F_\ell}$ and $(\rho_2)_{\ol\F_\ell}$ are irreducible, and there exists a character $\chi\co \Gamma/\Gamma_0 \to 1 + \fm_{\ol\Z_\ell}$ such that $\rho_2 \cong \rho_1 \otimes \chi$.
    
    Conversely, if $\ol\rho$ is an irreducible $\ol\F_\ell$-representation of $\Gamma$, then there exists a $\ol\Z_\ell$-representation $\rho$ of $\Gamma$ such that $\rho_{\ol\F_\ell} \cong \ol\rho$ and $\rho_{\ol\Q_\ell}$ is irreducible.
\end{lemma}

\begin{proof}
    Observe that $\ol\Z_\ell^\times \cong \ol\F_\ell^\times \times (1 + \fm_{\ol\Z_\ell})$. Because $\Gamma_0$ is a finite normal subgroup of $\Gamma$ of order prime to $\ell$ and $\Gamma/\Gamma_0$ is abelian and $\Gamma$ admits a central subgroup of finite index prime to $\ell$, for the first claim we may twist $\rho_1$ and $\rho_2$ to assume that there is a central subgroup $Z_0 \subset \Gamma$ such that $\Gamma/Z_0$ is of finite order prime to $\ell$ and $\rho_1$ and $\rho_2$ factor through $\Gamma/Z_0$. Then irreducibility of representations of $\Gamma/Z_0$ under reduction modulo $\ell$ follows, since $\ell \nmid |\Gamma/Z_0|$. For the claim of the second paragraph, we may perform a similar twisting to assume $\ol\rho$ factors through $\Gamma/Z_0$. Thus we may pass from $\Gamma$ to $\Gamma/Z_0$ to assume that $\Gamma$ is finite of order prime to $\ell$. In this case, the claim follows from \cite[Part~III, no.\ 15.5, Proposition~43]{Serre77}.
\end{proof}

\section{Large prime degree base change}\label{sec:large-prime-degree-base-change}

In this section, we specialize the results of the previous section on the Glauberman correspondence to the case of ``large'' prime degree base change for paraductive $\F_q$-group schemes $\ol G$. Specifically, we will show that the Glauberman correspondence preserves cuspidality and Lusztig series.

\subsection{Banal primes}

We begin by analyzing conditions under which we may apply the Glauberman correspondence when $\ol G$ is the special fiber of a point stabilizer in the Bruhat--Tits building.

\begin{defn}\label{def:banal-prime}
If $\Gamma$ is a locally profinite group, then we say that a prime number $\ell$ is \emph{banal} for $\Gamma$ if $\ell$ does not divide the pro-order of any compact open subgroup of $\Gamma$.
\end{defn}

\begin{prop}\label{prop:banal-primes}
    Let $\ell$ be a banal prime for $G(F)$ such that $\ell > \rk G+1,$\footnote{This assumption is not quite optimal for every claim that follows, as the proof shows, but it does not follow from banality and the claims can fail without it. For example, let $F = \Q_2$, let $\ell = 3$, let $D$ be a central division algebra of dimension $9$ over $F$, and let $G = D^1$. Then $\ell$ is banal for $G$ but $\ell = \rk G + 1$ and (1), (2), and (3) all fail (and therefore (4) also fails): for (1) and (2), this follows from the fact that $G_{F_\ell} \cong \SL_{3, F_\ell}$ contains an isotropic torus isomorphic to $\Res_{F_{2\ell}/F_\ell} \G_m$, which has pro-order divisible by $\ell$. For (3), the proof of \cite[Theorem~10.3.1]{KP} shows that $\cB(G) = \{x\}$ and $x$ is the barycenter of a chamber in $\cB(G_{F_\ell})$.} and let $x \in \cB(G)$. Then the following properties hold. 
    \begin{enumerate}
        \item If $S$ is a maximal split $F$-torus of $G$, then $S_{F_\ell}$ is a maximal split $F_\ell$-torus of $G_{F_\ell}$.
        \item The prime $\ell$ is banal for $G(F_\ell)$,
        \item If $[x]$ is a vertex in $\cB(G_{\der})$, then it is also a vertex in $\cB((G_{\der})_{F_\ell})$,
        \item Let $\ol G_{[x]}$ be the quotient of the special fiber of the smooth affine $\cO_F$-group scheme associated to $G(F)_{[x]}$ by its unipotent radical, and let $\rho$ be a cuspidal $\ol\Q_\ell$-representation of $\ol G_{[x]}(\F_q)$ (in the sense of \S\ref{sss:cusp}). The tuple 
        \[
        (\Gamma = \ol G_{[x]}(\F_{q^\ell}), \Gamma_0 = \ol G_{[x]}^\circ(\F_{q^\ell}), Z = Z(\ol G)(\F_{q^\ell}),\sigma = \Fr_q)
        \]
        satisfies Hypothesis~\ref{hypothesis:infinite-glauberman}.
    \end{enumerate}
\end{prop}

\begin{proof}
    We begin with (1). By passing from $G$ to $Z_G(S)/S$, we may assume that $G$ is anisotropic; by further passing separately to the maximal central torus of $G$ and to the universal cover of $G_{\mathrm{der}}$, we may assume that $G$ is either a torus or semisimple and simply connected. If $G = T$ is a torus, then $\Aut(X^*(T_{\ol F}))$ has no elements of order $\ell$: indeed, an element of order $\ell$ would have minimal polynomial of degree $\ell - 1$, and this cannot be the case since $\ell - 1 > \dim G$ by hypothesis. Thus $T_{F_\ell}$ is anisotropic, and we may pass to the case that $G$ is semisimple and simply connected. By \cite[Remark 10.3.2]{KP}, it follows that $G$ is of inner type $\mathrm{A}_n$, so there exist finite extensions $E_1, \dots, E_n$ of $F$ and division algebras $D_1, \dots, D_n$ with $Z(D_i) = E_i$ such that $G \cong \prod_{i=1}^n \Res_{E_i/F} \SL(D_i)$. Let $d_i = \sqrt{\dim_{E_i} D_i}$ and $e_i = [E_i\co F]$, so $\rk G = \sum_{i=1}^n (d_i - 1)e_i$. Since $\ell > \rk G + 1$ by hypothesis, we have $\ell \nmid d_ie_i$. By local class field theory, $D_i$ corresponds to an order $d_i$ element of $\rH^2(E_i, \G_m) \cong \Q/\Z$, and restriction-corestriction shows that $(D_i)_{E_i F_\ell}$ is a central division algebra over $E_i F_\ell$. Thus $G_{F_\ell} \cong \prod_{i=1}^n \Res_{E_i F_\ell/F_\ell} \SL((D_i)_{E_i F_\ell})$ is anisotropic, as desired.

    Now recall that if $S$ is a maximal split $F$-torus of $G$ then every point of $\cB(G)$ is $G(F)$-conjugate to a point in the apartment $\cA(S)$. By (1), the base change $S_{F_\ell}$ is also a maximal split $F_\ell$-torus of $G_{F_\ell}$, and it is clear that the natural map $\cA(S) \to \cA(S_{F_\ell})$ is an isomorphism of simplicial complexes. Thus (3) holds and every point of $\cB(G_{F_\ell})$ is $G(F_\ell)$-conjugate to a point of $\cB(G)$.
    
    By the Bruhat--Tits fixed point lemma, every compact open subgroup of $G(F_\ell)$ stabilizes some point of $\cB(G_{F_\ell})$. Hence in order to show (2), it is enough to show that the pro-order of $G(F_\ell)_x$ is not divisible by $\ell$. But the pro-order of $G(F_\ell)_x$ is the product of $p^\infty$ and the order of $\ol G_x(\F_{q^\ell})$. By assumption, $\ell$ does not divide the order of $\ol G_x(\F_q)$, so (2) follows from \cite[Lemma~2.1]{Cot26a}.

    For (4), we must show
    \begin{enumerate}[label=(\Alph*)]
        \item $\ell$ does not divide $|\ol G_{[x]}^\circ(\F_{q^\ell})|$,
        \item $\ell$ does not divide the order of $\ol G_{[x]}(\F_{q^\ell})/Z$,
        \item $\sigma$ acts trivially on $\ol G_{[x]}(\F_{q^\ell})/\ol G_{[x]}^\circ(\F_{q^\ell})$.
    \end{enumerate}
    Item (A) is clear from (2); for (B), let $H = G_{\mathrm{sc}} \times Z(G)^\circ_{\red}$, where $G_{\mathrm{sc}}$ is the universal cover of $G_{\der}$, and let $\pi\co H \to G$ denote the map induced by the universal cover and multiplication. Note that $\ker \pi \subset Z(G_{\mathrm{sc}})$, so there is an exact sequence
    \[
    (\ker \pi)(F_\ell) \to H(F_\ell)_{[x]} \to G(F_\ell)_{[x]} \to \rH^1(F_\ell, \ker \pi),
    \]
    and $(\ker \pi)(F_\ell)$ and $\rH^1(F_\ell, \ker\pi)$ have no $\ell$-torsion since $|\ker \pi|$ is only divisible by primes which are at most $\rk G + 1$ (as one sees by the classification of connected reductive groups over algebraically closed fields). Thus we reduce from $G$ to $H$, and by passing to direct factors we may assume $G$ is semisimple and simply connected. In this case, $\ol G_{[x]}$ is connected, so (B) follows from (A).

    Finally, for item (C), observe that if $\pi_1(G)$ refers to Borovoi's fundamental group then by \cite[Corollary~11.6.3]{KP} we have a $W_F/I_F = \langle\Fr\rangle$-equivariant inclusion
    \[
    \ol G_{[x]}(\ol\F_q)/\ol G_{[x]}^\circ(\ol\F_q) \subset \pi_1(G)_{I_F}.
    \]
    Thus it suffices to show that
    \begin{equation}\label{eqn:fundamental-group-equality}
    (\pi_1(G)_{I_F})^{\Fr} = (\pi_1(G)_{I_F})^{\Fr^\ell}.
    \end{equation}
    By \cite[Lemma~1.8]{Bor98}, we may pass to an inner form of $G$ to assume that $G$ is quasi-split; fix a Borel $F$-subgroup $B \subset G$ and a maximal $F$-subtorus $T \subset B$. By definition, there is then a $\Gal(\ol F/F)$-equivariant isomorphism
    \[
    \pi_1(G) \cong X_*(T_{\ol F})/\Span_{\Z}(\Phi^\vee(G_{\ol F}, T_{\ol F})),
    \]
    where $\Phi^\vee(G_{\ol F}, T_{\ol F})$ denotes the set of coroots for the pair $(G_{\ol F}, T_{\ol F})$. Observe that the action of $\Fr$ on $\pi_1(G)_{I_F}$ is induced by an automorphism of $X_*(T_{\ol F})$, which is of order not divisible by $\ell$ since $\ell - 1 > \rk G = \dim T$ by hypothesis. This establishes \eqref{eqn:fundamental-group-equality} and hence (C).
\end{proof}

\subsection{Cuspidality}

The main aim of this section is to show that when applied to $\ol\Q_\ell$-representations of finite groups of Lie type, the Glauberman correspondence sends cuspidal representations to cuspidal representations (a partial converse of Proposition~\ref{prop:tate-cohomology-cuspidal}).

\begin{lemma}\label{lemma:split-rank-stays-same}
    Suppose $\ol G$ is connected, and let $\ell > \rk \ol G + 1$ be a prime number.
    \begin{enumerate}
        \item If $\ol T \subset \ol G$ is a maximal $\F_q$-torus, then
        \[
        \rk_{\F_q}(\ol T) = \rk_{\F_{q^\ell}}(\ol T_{\F_{q^\ell}})
        \]
        and
        \[
        (N_{\ol G}(\ol T)/\ol T)(\F_q) = (N_{\ol G}(\ol T)/\ol T)(\F_{q^\ell}).
        \]
        \item If $\ol T_\ell \subset \ol G_{\F_{q^\ell}}$ is a maximal $\F_{q^\ell}$-torus (resp.\ $\ol P_\ell \subset \ol G_{\F_{q^\ell}}$ is a parabolic $\F_{q^\ell}$-subgroup), then there exists a maximal $\F_q$-torus $\ol T$ (resp.\ a parabolic $\F_q$-subgroup $\ol P \subset \ol G$) such that $\ol T_\ell = \ol T_{\F_{q^\ell}}$ (resp.\ $\ol P_\ell = \ol P_{\F_{q^\ell}}$).
    \end{enumerate}
\end{lemma}

\begin{proof}
    The first claim of (1) is clear from the fact that the action of the Frobenius element of $\Gal(\ol\F_q/\F_q)$ on $X^*(\ol T_{\ol\F_q})$ is of order prime to $\ell$ (since $\ell > \rk \ol G + 1$). The second claim of (1) and the first claim of (2) are established in \cite[Lemma~2.5]{Cot26a}. For the second claim of (2), note that if $\ol T_0 \subset \ol G$ is an $\F_q$-torus which lies in a Borel $\F_q$-subgroup, then every parabolic $\F_{q^\ell}$-subgroup of $\ol G_{\F_{q^\ell}}$ is $\ol G(\F_{q^\ell})$-conjugate to one which contains $(\ol T_0)_{\F_{q^\ell}}$. Every such parabolic $\F_{q^\ell}$-subgroup is of the form $P_{\ol G_{\F_{q^\ell}}}(\lambda)$ for a cocharacter $\lambda\co \G_m \to (\ol T_0)_{\F_{q^\ell}}$, and (1) shows that such a cocharacter is defined over $\F_q$.
\end{proof}

\begin{prop}\label{prop:glauberman-cuspidal}
    Let $\rho$ be a cuspidal $\ol\Q_\ell$-representation of $\ol G(\F_q)$ (in the sense of \S\ref{sss:cusp}) and let $\ell$ be a prime number. Suppose that Hypothesis~\ref{hypothesis:infinite-glauberman} holds with $\Gamma = \ol G(\F_{q^\ell})$, $\Gamma_0 = \ol G^\circ(\F_{q^\ell})$, $\sigma$ induced by a generator of $\Gal(\F_{q^\ell}/\F_q)$, and $Z = Z(\ol G)(\F_{q^\ell})$. Suppose moreover that $\ell > \rk \ol G + 1$. Then the $\ol\Q_\ell$-representation $\rho_\ell$ of $\ol G(\F_{q^\ell})$ corresponding to $\rho$ under the Glauberman correspondence is also cuspidal.
\end{prop}

\begin{proof}
    Recall that by definition a finite-dimensional representation $\rho$ of $\ol G(\F_q)$ is cuspidal if and only if its restriction to $\ol G^\circ(\F_q)$ is cuspidal, so it suffices to prove the proposition in the case that $\ol G$ is connected. In this case, Hypothesis~\ref{hypothesis:infinite-glauberman} just says that $\ell$ does not divide $|\ol G(\F_q)|$. We may further assume that $\rho$ is an irreducible $\ol G(\F_q)$-representation.
    
    Suppose for the sake of contradiction that $\rho_\ell$ is not cuspidal, so by Lemma~\ref{lemma:split-rank-stays-same}(2) there exists a proper parabolic $\F_q$-subgroup $\ol P$ of $\ol G$ with Levi $\ol L$ and a cuspidal character $\chi$ of $\ol L(\F_{q^\ell})$ such that $\langle\rho_\ell, \ind_{\ol P(\F_{q^\ell})}^{\ol G(\F_{q^\ell})}(\chi)\rangle \neq 0$. If $\sigma$ is the automorphism of $\ol G(\F_{q^\ell})$ induced by the $q$-Frobenius automorphism of $\F_{q^\ell}$, then we have
    \[
    0 \neq \langle \rho_\ell, \ind_{\ol P(\F_{q^\ell})}^{\ol G(\F_{q^\ell})}(\chi)\rangle = \langle{}^\sigma\rho_\ell, {}^\sigma\ind_{\ol P(\F_{q^\ell})}^{\ol G(\F_{q^\ell})}(\chi)\rangle = \langle\rho_\ell, \ind_{\ol P(\F_{q^\ell})}^{\ol G(\F_{q^\ell})}({}^\sigma\chi)\rangle
    \]
    since $\rho_\ell$ and $\ol P(\F_{q^\ell})$ are $\sigma$-stable. Since $\ol L$ is also $\sigma$-stable, \cite[Proposition~9.1.5]{Car85} shows that there is some $w \in N_{\ol G}(\ol L)(\F_{q^\ell})$ such that $^\sigma\chi = {}^w\chi$. Since $\ell$ does not divide the order of $N_{\ol G}(\ol L)(\F_{q^\ell})/\ol L(\F_{q^\ell})$ by hypothesis, it follows that $w \in \ol L(\F_{q^\ell})$, i.e., $^\sigma\chi = \chi$. Let $\chi_0$ denote the irreducible representation of $\ol L(\F_q)$ corresponding to $\chi$ under the Glauberman correspondence; by Proposition~\ref{prop:tate-cohomology-cuspidal}, $\chi_0$ is cuspidal.
    
    Let $W$ denote the relative Weyl group of $(\ol G, \ol T)$, where $\ol T \subset \ol L$ is a maximal split $\F_q$-torus. Note that $W$ is also the relative Weyl group of $\ol G_{\F_{q^\ell}}$ by Lemma~\ref{lemma:split-rank-stays-same}(1). By \cite[Proposition~9.2.4]{Car85}, since $\chi$ is cuspidal we have 
    \begin{equation}\label{eqn:carter-pairing-1}
    \langle \ind_{\ol P(\F_{q^\ell})}^{\ol G(\F_{q^\ell})}(\chi), \ind_{\ol P(\F_{q^\ell})}^{\ol G(\F_{q^\ell})}(\chi)\rangle = |\{w \in W\co {}^w\ol L = \ol L, {}^w\chi = \chi\}|
    \end{equation}
    and similarly
    \begin{equation}\label{eqn:carter-pairing-2}
    \langle \ind_{\ol P(\F_q)}^{\ol G(\F_q)}(\chi_0), \ind_{\ol P(\F_q)}^{\ol G(\F_q)}(\chi_0)\rangle = |\{w \in W\co {}^w\ol L = \ol L, {}^w\chi_0 = \chi_0\}|.
    \end{equation}
    Since the $W$-action commutes with $\sigma$, canonicity of the Glauberman correspondence implies that ${}^w \chi = \chi$ if and only if ${}^w \chi_0 = \chi_0$. Since the right sides of \eqref{eqn:carter-pairing-1} and \eqref{eqn:carter-pairing-2} agree, it follows from Corollary~\ref{cor:induction-commutes-sometimes} that every irreducible constituent of $\ind_{\ol P(\F_{q^\ell})}^{\ol G(\F_{q^\ell})}(\ol\chi)$ is the Glauberman correspondent of an irreducible constituent of $\ind_{\ol P(\F_q)}^{\ol G(\F_q)}(\ol\chi_0)$. By \cite[Part~III, no.\ 15.5, Proposition~43]{Serre77}, the analogous statement holds with $\ol\Q_\ell$-coefficients in place of $\ol\F_\ell$-coefficients, contradicting cuspidality of $\rho$.
\end{proof}

\begin{cor}\label{cor:glauberman-cuspidal-support}
    Suppose that Hypothesis~\ref{hypothesis:infinite-glauberman} holds with $\Gamma = \ol G(\F_{q^\ell})$, $\Gamma_0 = \ol G^\circ(\F_{q^\ell})$, $\sigma$ induced by a generator of $\Gal(\F_{q^\ell}/\F_q)$, and $Z = Z(\ol G)(\F_{q^\ell})$, and suppose $\ell > \rk \ol G + 1$. Let $\tau$ be an irreducible $\ol\Q_\ell$-representation of $\ol G(\F_q)$, and let $\tau_\ell$ be the $\ol\Q_\ell$-representation of $\ol G(\F_{q^\ell})$ corresponding to $\tau$ under the Glauberman correspondence. If $(\ol M, \rho)$ is the pair corresponding to $\tau$ via Lemma~\ref{lemma:cuspidal-supports-exist}, then $(\ol M_{\F_{q^\ell}}, \rho_\ell)$ is the pair corresponding to $\tau_\ell$, where $\rho_\ell$ corresponds to $\rho$ under the Glauberman correspondence.
\end{cor}

\begin{proof}
    By Proposition~\ref{prop:glauberman-cuspidal}, the representation $\rho_\ell$ is cuspidal. By twisting by a character of $\Gamma/\Gamma_0 = \Gamma^\sigma/\Gamma_0^\sigma$ and passing to a central quotient of $\Gamma$, we may and do assume that $\ol G$ is of finite type and $\Gamma$ is finite of order prime to $\ell$. This allows us to pass freely between $\ol\Q_\ell$ and $\ol\F_\ell$ by \cite[Part~III, no.\ 15.5, Proposition~43]{Serre77}. In this case, the parabolic induction $\ind_{\ol P(\F_{q^\ell})}^{\ol G(\F_{q^\ell})}(\rho_\ell)$ is semisimple as an $\ol\F_\ell[\Gamma]$-module, so \cite[Proposition~3.3]{TV}, Lemma~\ref{lemma:tate-cohom-devissage}, and Lemma~\ref{lemma:infinite-glauberman} combine to show the claim.
\end{proof}

\subsection{Compatibilities}

We next examine various compatibilities between representation-theoretic constructions in the setting of base change. 

\subsubsection{Compatibility of Glauberman correspondence and Shintani descent} We make the following simple remark, concerning the compatibility of the Glauberman correspondence and Shintani descent, which was also observed already in \cite{Dig86b}.

\begin{remark}\label{remark:glauberman-shintani}
    Let $\ol G$ be a connected reductive group over $\F_q$, let $\ell$ be a prime not dividing $|\ol G(\F_q)|$, let $\Gamma = \ol G(\F_{q^\ell})$, and let $\sigma$ be the automorphism of $\Gamma$ induced by a generator of $\Gal(\F_{q^\ell}/\F_q)$. In this case, the Glauberman correspondence is essentially a special case of Shintani descent: namely, recalling the map $N_\ell\co \ol G(\F_{q^\ell})/{\sim_\ell} \to \ol G(\F_q)/{\sim}$ from \S\ref{sss:shintani}, \cite[(1.2.6)]{Kaw87} and \cite[Proposition 3.11]{Dig86b} show
    \[
    N_\ell|_{\ol G(\F_q)_{\ell'}} = [\ell]|_{\ol G(\F_q)_{\ell'}},
    \]
    where $[\ell]$ denotes the $\ell$-power map.\footnote{Observe that $\ell$ clearly does not divide the integer $M_{\ol G}$ of \S\ref{sss:shintani}.} Since $\ell$ does not divide $|\ol G(\F_q)|$, \cite[Lemma~2.1]{Cot26a} implies that $\ell$ does not divide $|\ol G(\F_{q^\ell})|$ either, so $N_\ell$ induces a bijection $\ol G(\F_q)/{\sim} \cong \ol G(\F_q)/{\sim}$.

    Now let $\chi$ be an irreducible character of $\Gamma$, let $\wt\chi$ be the extension of $\chi$ to a character of $\Gamma \rtimes\langle\sigma\rangle$ such that $\wt\chi(1\rtimes \sigma) \in \Z$, let $\chi_0$ be the Shintani descent of $\chi$ (a class function on $\Gamma_0 = \ol G(\F_q) = \Gamma^\sigma$) corresponding to $\wt\chi$, and let $\chi_1$ be the character of $\Gamma_0$ corresponding to $\chi$ under the Glauberman correspondence. By \eqref{eqn:shintani-descent-relation}, we have
    \[
    \chi_0(N_\ell(g)) = \wt \chi(g \rtimes \sigma)
    \]
    for all $g \in \Gamma_0$; this uniquely determines $\chi_0$ by the previous paragraph. On the other hand, \eqref{eqn:glauberman-condition} shows that there is some $\epsilon \in \{\pm 1\}$
    \[
    \chi_1(g) = \epsilon\wt\chi(g \rtimes \sigma)
    \]
    for all $g \in \Gamma_0$. Thus we find
    \[
    \chi_1 = \epsilon \chi_0 \circ N_\ell = \epsilon\chi_0 \circ [\ell].
    \]
    In other words, the Glauberman correspondence is (up to sign) a twist of Shintani descent by $[\ell]$. Remarkably, as we will see, the Glauberman correspondence realizes large-prime-degree Frobenius-twisted base change for the Langlands correspondence mod $\ell$. This provides one justification for the convention for the norm map from \cite{Kaw87}.
\end{remark}

\subsubsection{Compatibility of Glauberman correspondence and Deligne--Lusztig induction}
The following technical lemma will allow the statement of the proposition below to be slightly cleaner, but it is not strictly logically necessary for our main goal.

\begin{lemma}\label{lemma:embedding-into-large-tori}
    Let $\mu$ be an $\F_q$-group scheme of multiplicative type, and let $\ell$ be a prime number not dividing the order of $\mu(\F_q)$. There exists an $\F_q$-group scheme $\nu$ of multiplicative type and a monic $\F_q$-homomorphism $\mu \to \nu$ such that $\ell$ does not divide the order of $\nu(\F_q)$ and $\rH^1(\F_q, \nu) = \rH^1(\F_{q^\ell}, \nu) = 1$.
\end{lemma}

\begin{proof}
    If $\ell = p$, then one can take $\nu$ to be any torus into which $\mu$ embeds, so assume $\ell \neq p$. We first reduce to the case that $\pi_0(\mu)(\ol\F_q)$ is of order prime to $\ell$. Let $N$ be the order of $\pi_0(\mu)(\F_q)$, so $\ell \nmid N$. Let $\pi_0(\mu)[N^\infty] = \bigcup_{k \in \Z_{>0}} \pi_0(\mu)[N^k]$ be the $\F_q$-subgroup scheme of $N$-power torsion in $\pi_0(\mu)$, and let $\mu_0$ be the open $\F_q$-subgroup scheme of $\mu$ with component group $\pi_0(\mu)[N^\infty]$. Since $(\pi_0(\mu)/\pi_0(\mu)[N^\infty])(\ol\F_q)$ is of order prime to $N$ and $\rH^1(\F_q, \pi_0(\mu)[N^\infty])$ is of $N$-power order, the long exact sequence on Galois cohomology shows $(\pi_0(\mu)/\pi_0(\mu)[N^\infty])(\F_q) = 1$. By standard results on Herbrand quotients, it follows that $\rH^1(\F_q, \pi_0(\mu)/\pi_0(\mu)[N^\infty]) = 1$ as well. Similarly, $\rH^1(\F_{q^\ell}, \pi_0(\mu)/\pi_0(\mu)[N^\infty]) = 1$ by \cite[Lemma 2.1]{Cot26a}. If $\mu_0$ embeds into some $\nu_0$ as in the lemma statement, then we may take $\nu = (\mu \times \nu_0)/\mu_0$, where $\mu_0$ embeds as $z \mapsto (z, z^{-1})$: indeed, from the exact sequence
    \[
    1 \to \mu_0(\F_q) \to \mu(\F_q) \times \nu_0(\F_q) \to \nu(\F_q) \to \rH^1(\F_q, \mu_0) \to \rH^1(\F_q, \mu \times \nu_0) \to \rH^1(\F_q, \nu) \to 1
    \]
    and the fact that $\ell$ does not divide the orders of $\mu_0(\F_q)$ or $\mu(\F_q) \times \nu_0(\F_q)$ or $\rH^1(\F_q, \mu_0)$, we see that $\ell \nmid |\nu(\F_q)|$. Moreover, since $\rH^1(\F_q,\nu_0) = 1$ and the map $\rH^1(\F_q, \mu_0) \to \rH^1(\F_q), \mu)$ is an isomorphism, we find that $\rH^1(\F_q, \nu) = 1$ and similarly $\rH^1(\F_{q^\ell}, \nu) = 1$. Thus we may pass from $\mu$ to $\mu_0$ to assume that $\pi_0(\mu)(\ol\F_q)$ is of order prime to $\ell$.
    
    Next, we reduce to the case that $\mu$ is finite \'etale. With $N$ as above (necessarily prime to $q$), let $\mu_1 = \mu[N^k]$ for some $k \in \Z_{>0}$ such that $\mu_1(\ol\F_q) \to \pi_0(\mu)(\ol\F_q)$ is surjective. If the result holds for $\mu_1$, then there is a monic $\F_q$-homomorphism $\mu_1 \to \nu_1$ as in the lemma statement. Then the pushout $\nu\coloneqq (\mu \times \nu_1)/\mu_1$ satisfies the requirements of the lemma for $\mu$, by the same argument as before. So we may pass to the case that $\mu$ is finite \'etale; in this case, we will show that one can take $\nu$ to be a torus, so $\rH^1(\F_q, \nu) = \rH^1(\F_{q^\ell}, \nu) = 1$ by Lang's theorem.

    Let $M = X^*(\mu_{\ol\F_q})$, and let $F_M$ be the automorphism of $M$ induced by the (arithmetic) Frobenius $F$ in $\Gamma \coloneqq \Gal(\ol\F_q/\F_q)$. Note that $M$ is a finite abelian group of order prime to $\ell$. For each $n \geq 1$, let $\Gamma_n \subset \Gamma$ denote the unique index $n$ closed subgroup. If $\nu$ is any $\F_q$-group scheme of multiplicative type and $X = X^*(\nu_{\ol\F_q})$, then there is a natural $\Gamma$-equivariant isomorphism $\nu(\ol\F_q) \cong \Hom(X, \ol\F_q^\times)$. Any embedding $\ol\F_q^\times \to \Q/\Z$ identifies the action of Frobenius of $\ol\F_q^\times$ with multiplication by $q$ on $\Q/\Z$ and shows that
    \[
    \Hom(\nu(\F_q), \Q/\Z) \cong X_{qF_X},
    \]
    where $F_X$ is the automorphism of $X$ induced by $F$, and the subscript $qF_X$ refers to the coinvariants for $qF_X$. Thus $|\nu(\F_q)|$ is prime to $\ell$ if and only if $X_{qF_X}$ is of order prime to $\ell$. If $\alpha_1, \dots, \alpha_n$ are the eigenvalues for $F_X$ (counted with multiplicity) on $X \otimes \ol\Q$, then
    \[
    |X_{qF_X}| = \prod_{i=1}^n (q - \alpha_i).
    \]
    If $\alpha_i$ is a primitive $m_i$th root of unity and $m_i'$ is the largest divisor of $m_i$ which is prime to $\ell$, then $|X_{qF_X}|$ is of order prime to $\ell$ if and only if $q$ has order distinct from $m_i'$ modulo $\ell$. Thus by duality it is enough to show that for each $m \in M$ there exists a finite free $\Z$-module $X$ equipped with a finite order automorphism $F_X$ and a homomorphism $f\co X \to M$ which intertwines $F_X$ and $F_M$, satisfies $m \in f(X)$, and such that every eigenvalue for $F$ on $X$ is of order whose prime-to-$\ell$ part is larger than the order of $q$ modulo $\ell$.

    Fix $m \in M$, and let $a\geq 1$ be such that $F_M^a(m) = m$. Let $b, c \in \Z_{>0}$ be such that $c$ is prime to $\ell$, and let $\pi\co \Z[\Gamma/\Gamma_{abc}] \to \Z[\Gamma/\Gamma_{ab}]$ be the $\Gamma$-equivariant projection. Note that the map $f_0\co \Z[\Gamma/\Gamma_{abc}] \to M$ given by $f_0([1]) = m$ factors through $\pi$. Observe that there is an injective homomorphism of $\Z[\Gamma]$-modules $j\co \Z[\Gamma/\Gamma_{ab}] \to \Z[\Gamma/\Gamma_{abc}]$ defined by
    \[
    j([F^n]) = \sum_{\substack{0\leq i<abc \\ i \equiv n\pmod{ab}}} [F^i]
    \]
    with the property that the composition of $j$ and $\pi$ is given by multiplication by $c$.
    Assume that $c$ is divisible enough that $cm = 0$. Then $f_0$ factors through a $\Gamma$-equivariant homomorphism
    \[
    f\co X\coloneqq \Z[\Gamma/\Gamma_{abc}]/\Z[\Gamma/\Gamma_{ab}] \to M,
    \]
    where we use $j$ to identify $\Z[\Gamma/\Gamma_{ab}]$ as a submodule of $\Z[\Gamma/\Gamma_{abc}]$. Observe that $X$ is a finite free $\Z$-module equipped with an automorphism $F_X$ (induced by $F$) of finite order. The eigenvalues for $F_X$ on $X \otimes \ol\Q$ are all $abc$th roots of unity which are not $ab$th roots of unity. Now take $b = (ac)^n$, where $n$ is large enough that for each prime $\ell_0 \neq \ell$ dividing $ac$, the power $\ell_0^{n+1}$ is larger than the order of $q$ in $(\Z/\ell)^\times$. If $\alpha$ is an eigenvalue for $F_X$, then it follows that $\alpha$ is of order larger than the order of $q$ in $(\Z/\ell)^\times$.
    Thus $X$ satisfies all the conditions of the previous paragraph, which proves the lemma.
\end{proof}

\begin{hypothesis}\label{hypothesis:glauberman-finite-reductive}
    For the rest of this section, let $\ol G$ be a paraductive $\F_q$-group scheme, let $\ol Z = Z(\ol G)$, and suppose that $\ell \neq p$ is a prime number such that Hypothesis~\ref{hypothesis:infinite-glauberman} holds with $\Gamma = \ol G(\F_{q^\ell})$, $\Gamma_0 = \ol G^\circ(\F_{q^\ell})$, $Z = \ol Z(\F_{q^\ell})$, and $\sigma = \Fr_q$. Suppose moreover that the action of $\Gal(\ol\F_q/\F_q)$ on $\pi_0(\ol G)(\ol\F_q)$ is of order prime to $\ell$ and $\ol G^\circ \cdot \ol Z$ is of prime-to-$\ell$ in $\ol G$.
\end{hypothesis}

\begin{prop}\label{prop:digne-3.5}
    Let $\ol T \subset \ol G$ be a generalized maximal $\F_q$-torus, and let $\theta\co \ol T(\F_q) \to \ol\Q_\ell^\times$ be a character. Suppose $\ell > \rk \ol G^\circ + 1$ and $\ell$ satisfies Hypothesis~\ref{hypothesis:glauberman-finite-reductive}.
    \begin{enumerate}
        \item There is a unique $\sigma$-stable extension $\theta_\ell\co \ol T(\F_{q^\ell}) \to \ol\Q_\ell^\times$ of $\theta$. If $\theta$ is non-singular, then $\theta_\ell$ is also non-singular.
        \item The Glauberman correspondence sends $\cE(\ol G, [\ol T, \theta])$ to $\cE(\ol G_{\F_{q^\ell}}, [\ol T_{\F_{q^\ell}}, \theta_\ell])$.
        \item If $\theta$ is non-singular, then $\Gla(R_{\ol T}^{\ol G}(\theta)) = R_{\ol T_{\F_{q^\ell}}}^{\ol G_{\F_{q^\ell}}}(\theta_\ell)$.
    \end{enumerate}
    Moreover, there is a constant $A$, depending only on the root datum of $\ol G^\circ$, such that for all $\ell > A$ the Glauberman correspondence sends $\cE_0(\ol G, [\ol T, \theta])$ to $\cE_0(\ol G_{\F_{q^\ell}}, [\ol T_{\F_{q^\ell}}, \ol\theta_\ell])$.
\end{prop}

\begin{proof}
    The first statement of (1) is clear from the Glauberman correspondence. Now suppose that $\theta$ is non-singular. Let $n$ be a positive integer such that $\ol T_{\F_{q^n}}$ is split, and let $\alpha^\vee$ be a coroot for $\ol T_{\F_{q^n}}$. By the definition of non-singularity, the composition $\theta \circ \Nm_{\F_{q^n}/\F_q} \circ \alpha^\vee$ is nontrivial. Note that $\theta \circ \Nm_{\F_{q^\ell}/\F_q} = \theta_\ell^\ell$, so we have
    \begin{align*}
    (\theta_\ell \circ \Nm_{\F_{q^{\ell n}}/\F_{q^\ell}} \circ \alpha^\vee)^\ell &= (\theta \circ \Nm_{\F_{q^\ell}/\F_q}) \circ \Nm_{\F_{q^{\ell n}}/\F_{q^\ell}} \circ \alpha^\vee \\
        &= (\theta \circ \Nm_{\F_{q^n}/\F_q} \circ \alpha^\vee) \circ \Nm_{\F_{q^{\ell n}}/\F_{q^n}},
    \end{align*}
    where the final equality follows from the fact that $\alpha^\vee$ is defined over $\F_{q^n}$. Thus $\theta_\ell \circ \Nm_{\F_{q^{\ell n}}/\F_{q^\ell}} \circ \alpha^\vee \neq 1$, so $\theta_\ell$ is non-singular.

    Our next aim is to reduce the remaining claims to the case that $\ol G$ is connected; we begin with a few preliminary reductions. First, by twisting by a character of $\ol G(\F_q)/\ol G^\circ(\F_q)$ (using Remark~\ref{rmk:glauberman-twisting-equivariance}) we may assume that $\theta|_{\ol Z(\F_q)}$ is of finite order prime to $\ell$; by passing to a central quotient of $\ol G$, we may therefore assume that $\ol G$ is of finite type and $\ell$ does not divide the order of $\ol G(\F_q)$.

    Let $\ol Z_0 = \ol Z \cap \ol G^\circ$, and choose a monic $\F_q$-homomorphism $\ol Z_0 \to \wt Z_0$ where $\wt Z_0$ is an $\F_q$-group scheme of multiplicative type such that $\ell$ does not divide $|\wt Z_0(\F_q)|$ and $\rH^1(\F_q, \wt Z_0) = \rH^1(\F_{q^\ell}, \wt Z_0) = 1$; this exists by Lemma~\ref{lemma:embedding-into-large-tori}. Let $\wt G = (\ol G \times \wt Z_0)/\ol Z_0$, where $\ol Z_0$ is embedded into $\ol G \times \wt Z_0$ via $z \mapsto (z, z^{-1})$. Observe that Hypothesis~\ref{hypothesis:infinite-glauberman} still holds with $\wt G(\F_{q^\ell})$, $\wt G^\circ(\F_{q^\ell})$, and $Z = (\ol Z \cdot \wt Z_0)(\F_q)$ in place of $\Gamma$, $\Gamma_0$, and $Z$. By Lemma~\ref{lemma:paraductive-series-extension}, we may reduce (2) to the case $\ol G = \wt G$, i.e., the case that $\rH^1(\F_q, \ol Z \cap \ol G^\circ) = \rH^1(\F_{q^\ell}, \ol Z \cap \ol G^\circ) = 1$. Using (1) and the fact that Deligne--Lusztig induction is concentrated in one degree when $\theta$ is non-singular \cite[Proposition~7.4]{DL76}, we may similarly reduce (3) to the case $\ol G = \wt G$. The final claim is reduced to the case $\ol G = \wt G$ by Lemma~\ref{lemma:E_0-independence-of-regular-embedding}.

    We next reduce to the case $\ol G = \ol G^\circ \cdot \ol Z$. By Lemma~\ref{lemma:disc-dl-induction}, if $\ol U$ is the unipotent radical of a Borel $\ol\F_q$-subgroup of $\ol G^\circ$ containing $\ol T^\circ$, then we have 
    \[
    \rH_c^n(Y_{\ol U}^{\ol G}, \ol\Q_\ell)_\theta \cong \ind_{(\ol G^\circ \cdot \ol Z)(\F_q)}^{\ol G(\F_q)} \rH_c^n(Y_{\ol U}^{\ol G^\circ \cdot \ol Z}, \ol\Q_\ell)_\theta
    \]
    as $\ol G(\F_q)$-representations. Reducing modulo $\ell$, applying \cite[Proposition~3.3]{TV}, Lemma~\ref{lemma:tate-cohom-devissage}, and Lemma~\ref{lemma:infinite-glauberman}, and then lifting to characteristic zero, we see that for every irreducible constituent $\tau$ of $\rH_c^n(Y_{\ol U}^{\ol G}, \ol\Q_\ell)_\theta$, there is some irreducible constituent $\rho_\ell$ of $\rH_c^n(Y_{\ol U_\ell}^{\ol G^\circ_\ell \cdot \ol Z_\ell}, \ol\Q_\ell)_{\theta_\ell}$ such that the Glauberman correspondent $\tau_\ell$ of $\tau$ is an irreducible constituent of $\ind_{(\ol G^\circ \cdot \ol Z)(\F_{q^\ell})}^{\ol G(\F_{q^\ell})}(\rho_\ell)$. Consequently, we may pass from $\ol G$ to $\ol G^\circ \cdot \ol Z$ to assume $\ol G = \ol G^\circ \cdot \ol Z$.

    Finally, we reduce to the case that $\ol G$ is connected. Since $\ol G = \ol G^\circ \cdot \ol Z$ and $\rH^1(\F_q, \ol Z \cap \ol G^\circ) = 1$ and $\rH^1(\F_{q^\ell}, \ol Z \cap \ol G^\circ) = 1$, we have $\ol G(\F_q) = \ol G^\circ(\F_q) \cdot \ol Z(\F_q)$ and $\ol G(\F_{q^\ell}) = \ol G^\circ(\F_{q^\ell}) \cdot \ol Z(\F_{q^\ell})$. By the same reasoning as in \cite[Remark 2.6.5]{Kal21b}, if $\theta^\circ = \theta|_{\ol T^\circ(\F_q)}$ then an irreducible representation of $\ol G(\F_q)$ lies in $\cE(\ol G, [\ol T, \theta])$ precisely when it lies in $\cE(\ol G^\circ, [\ol T^\circ, \theta^\circ])$ and has central character $\theta|_{\ol Z(\F_q)}$. The latter condition is clearly preserved on passage to $\F_{q^\ell}$, so we may pass from $\ol G$ to $\ol G^\circ$ to assume that $\ol G$ is connected.


    At this point, (2) follows from \cite[Lemma 2.2]{Cot26a} and \cite[Corollaire 3.2]{Dig99}. For (3), note that by \cite[Proposition~7.4]{DL76}, the virtual representation $(-1)^{\rk_{\F_q}(\ol G) - \rk_{\F_q}(\ol T)}R_{\ol T}^{\ol G}(\theta)$ (resp.\ $(-1)^{\rk_{\F_{q^\ell}}(\ol G_{\F_{q^\ell}}) - \rk_{\F_{q^\ell}}(\ol T_{\F_{q^\ell}})} R_{\ol T_{\F_{q^\ell}}}^{\ol G_{\F_{q^\ell}}}(\theta_\ell)$) is an actual $\ol G(\F_q)$-representation (resp.\ $\ol G(\F_{q^\ell})$-representation). Therefore it suffices to observe that $\rk_{\F_q}(\ol G) = \rk_{\F_{q^\ell}}(\ol G_{\F_{q^\ell}})$ and $\rk_{\F_q}(\ol T) = \rk_{\F_{q^\ell}}(\ol T_{\F_{q^\ell}})$; these equalities follow from Lemma~\ref{lemma:split-rank-stays-same}. The final claim follows from \cite[Theorem 1.2]{Cot26a}.\footnote{We remark that, although the proof of \cite[Theorem 1.2]{Cot26a} uses \cite[Theorem~8.7.2]{GRV26} for general twisted Levi subgroups $\ol L$, this input is not needed when $\ol L$ is a torus.}
\end{proof}

\subsection{Lusztig series}\label{ssec:lusztig-series}

We saw in Proposition~\ref{prop:digne-3.5}(2) that, if $\ell > \max(|\Omega|^2, \rk \ol G^\circ + 1)$, where $\Omega$ is the absolute Weyl group of $\ol G^\circ$, then the Glauberman correspondence ``preserves semi-rational Lusztig series'' in an appropriate sense.

\begin{cor}\label{cor:glauberman-lusztig-bijection}
    There exists a constant $C$ such that\footnote{We have not attempted to optimize the constant $C$.} for every $\ell > C$, Hypothesis~\ref{hypothesis:glauberman-finite-reductive} is satisfied and if $\ol T \subset \ol G$ is a generalized maximal $\F_q$-torus and $\theta\co \ol T(\F_q) \to \ol\Q_\ell^\times$ is a character with $\Gal(\F_{q^\ell}/\F_q)$-stable extension $\theta_\ell\co \ol T(\F_{q^\ell}) \to \ol\Q_\ell^\times$, then the Glauberman correspondence
    \[
    \Gla\co \cE(\ol G, [\ol T, \theta]) \to \cE(\ol G_{\F_{q^\ell}}, [\ol T_{\F_{q^\ell}}, \theta_\ell])
    \]
    is bijective.
\end{cor}

\begin{proof}
    Since $\Gla$ is injective, it suffices to prove surjectivity for large enough $\ell$. First, we show that $\cE(\ol G_{\F_{q^\ell}}, [\ol T_{\F_{q^\ell}}, \theta_\ell])$ is finite of order which is bounded above independently of $\ell$.

    By twisting, we may assume that $\theta$ is of finite order (of order bounded independently of $\ol T$). Thus by passing to a central quotient of $\ol G$, we may and do assume that $\ol G$ is of finite type. In this case, observe first that if $\ell > \rk \ol G^\circ + 1$, then every maximal $\F_{q^\ell}$-torus of $\ol G_{\F_{q^\ell}}$ is $\ol G(\F_{q^\ell})$-conjugate to the base change of a maximal $\F_q$-torus of $\ol G$ by Lemma~\ref{lemma:split-rank-stays-same}; let $M$ be the number of such maximal $\F_q$-tori. By Lemma~\ref{lemma:dl-6.8-variant}, there exists an integer $N$ (which is independent of $\ell$) such that if $\ol T' \subset \ol G$ is a maximal $\F_q$-torus and $\theta'_0\co \ol T'(\F_{q^\ell}) \to \ol\Q_\ell^\times$ is a character, then the number of irreducible constituents of $R_{\ol T'_{\F_{q^\ell}}}^{\ol G_{\F_{q^\ell}}}(\theta'_0)$ is at most $N$. If $\pi_0(\ol T)(\ol\F_q)$ is of order $R_0$ and $\theta|_{\ol T^\circ(\F_q)}$ is of order $R_1$, then any such $\theta'_0$ above must restrict to an order $R_1$ character of $\ol T^\circ(\F_{q^\ell})$, of which there are at most $R_0 \cdot R_1^{\dim \ol T^\circ}$. This shows that $\cE(\ol G_{\F_{q^\ell}}, [\ol T_{\F_{q^\ell}}, \theta_\ell])$ is of order at most $MNR_0R_1^{\dim \ol T^\circ}$, which is indeed independent of $\ell$.

    We now let $C$ be large enough so that: 
    \begin{enumerate}
        \item Hypothesis~\ref{hypothesis:glauberman-finite-reductive} holds for all $\ell > C$,
        \item $C > MNR_0R_1^{\dim\ol T^\circ}$ for all $\ol T$,
        \item $C > \max(|\Omega|^2, \rk \ol G^\circ + 1)$.
    \end{enumerate}
    Since $\theta_\ell$ is $\Gal(\F_{q^\ell}/\F_q)$-stable, the group $\Gal(\F_{q^\ell}/\F_q)$ acts on $\cE(\ol G_{\F_{q^\ell}}, [\ol T_{\F_{q^\ell}},\theta_\ell])$. If $\ell > C$, then the preceding paragraph shows that $\ell$ is larger than the order of $\cE(\ol G_{\F_{q^\ell}}, [\ol T_{\F_{q^\ell}},\theta_\ell])$, so every element of $\cE(\ol G_{\F_{q^\ell}}, [\ol T_{\F_{q^\ell}},\theta_\ell])$ is necessarily $\Gal(\F_{q^\ell}/\F_q)$-stable. By Lemma~\ref{lemma:infinite-glauberman}, every such element arises from $\Gla$, as desired.
\end{proof}

The following technical corollary will appear at a crucial point later.

\begin{cor}\label{cor:glauberman-torus-character-conjugate}\emergencystretch=3em
    Suppose that $\ell$ satisfies Hypothesis~\ref{hypothesis:glauberman-finite-reductive} and $\ell > \max(\rk \ol G^\circ + 1, |\Omega|^2)$. Let $\ol T \subset \ol G$ (resp.\ $\ol S \subset \ol G_{\F_{q^\ell}}$) be a generalized maximal $\F_q$-torus (resp.\ generalized maximal $\F_{q^\ell}$-torus), let $\theta\co \ol T(\F_q) \to \ol\Q_\ell^\times$ (resp.\ $\eta\co \ol S(\F_{q^\ell}) \to \ol\Q_\ell^\times$) be a character, and let $\theta_\ell\co \ol T(\F_{q^\ell}) \to \ol\Q_\ell^\times$ be the unique $\Gal(\F_{q^\ell}/\F_q)$-stable character extending $\theta$. If $\cE(\ol G_{\F_{q^\ell}}, [\ol S, \eta]) = \cE(\ol G_{\F_{q^\ell}}, [\ol T_{\F_{q^\ell}}, \theta_\ell])$ are geometrically conjugate, then there exists a generalized maximal $\F_q$-torus $\ol S_0 \subset \ol G$ and a character $\eta_0\co \ol S_0(\F_q) \to \ol\Q_\ell^\times$ such that $(\ol S, \eta)$ is $\ol G(\F_{q^\ell})$-conjugate to $((\ol S_0)_{\F_{q^\ell}}, \eta_{0,\ell})$, where $\eta_{0,\ell}$ is the unique $\Gal(\F_{q^\ell}/\F_q)$-stable extension of $\eta_0$ to $\ol S_0(\F_{q^\ell})$.
\end{cor}

\begin{proof}
    By Lemma~\ref{lemma:split-rank-stays-same}, since $\ell > \rk \ol G^\circ + 1$ there is a generalized maximal $\F_q$-torus $\ol S_0 \subset \ol G$ such that $\ol S$ is $\ol G(\F_{q^\ell})$-conjugate to $(\ol S_0)_{\F_{q^\ell}}$. By conjugacy, we may assume $\ol S = (\ol S_0)_{\F_{q^\ell}}$. It is now enough to show that $\eta$ is $\Gal(\F_{q^\ell}/\F_q)$-stable, as we may then (by the uniqueness aspect of Proposition~\ref{prop:digne-3.5}(1)) let $\eta_0 = \eta|_{\ol S_0(\F_q)}$.

    Finally, we show that $\eta$ is $\Gal(\F_{q^\ell}/\F_q)$-stable. Since $\ell \nmid n$, we have
    \[
    \Gal(\F_{q^{\ell n}}/\F_q) \cong \Gal(\F_{q^\ell}/\F_q) \times \Gal(\F_{q^n}/\F_q).
    \]
    Let $s \in \ol S_0^\circ(\F_{q^\ell})$ and $\gamma \in \Gal(\F_{q^\ell}/\F_q)$, and let $t \in \ol T^\circ(\F_{q^{\ell n}})$ such that $s = \Nm_{\F_{q^{\ell n}}/\F_{q^\ell}}(gtg^{-1})$. We have now
    \begin{align*}
    \eta(\gamma(s)) &= \eta(\gamma(\Nm_{\F_{q^{\ell n}}/\F_{q^\ell}}(gtg^{-1}))) = \eta(\Nm_{\F_{q^{\ell n}}/\F_{q^\ell}}(\gamma(gtg^{-1}))) \\
        &= \eta(\Nm_{\F_{q^{\ell n}}/\F_{q^\ell}}(g\gamma(t)g^{-1})) = \theta_\ell(\Nm_{\F_{q^{\ell n}}/\F_{q^\ell}}(\gamma(t))) \\
        &= \theta_\ell(\gamma(\Nm_{\F_{q^{\ell n}}/\F_{q^\ell}}(t))) = \theta_\ell(\Nm_{\F_{q^{\ell n}}/\F_{q^\ell}}(t)) \\
        &= \eta(\Nm_{\F_{q^{\ell n}}/\F_{q^\ell}}(gtg^{-1})) = \eta(s),
    \end{align*}
    so indeed $\eta$ is $\Gal(\F_{q^\ell}/\F_q)$-stable, as desired.
\end{proof}

\section{Applications to the depth 0 Local Langlands Correspondence}\label{section:debacker-reeder}

Throughout this section, let $F$ be a non-archimedean local field with ring of integers $\cO_F$ and residue field $\F_q$. Let $W_F$ denote the Weil group of $F$, let $I_F$ denote the inertia subgroup of $W_F$, and let $P_F$ denote the wild inertia subgroup of $W_F$. Let $G$ be a connected reductive $F$-group, and let $\wh G$ denote the Langlands dual group of $G$ (over $\Z$, say). We will let $\ld G = \wh G \rtimes W_0$, where $W_0$ is a (sufficiently large) finite quotient of $W_F$ through which the action of $W_F$ on $\wh G$ factors. We will apply most of the preceding material to analyzing the Fargues--Scholze parameters of depth $0$ supercuspidal $\ol\Q_\ell$-representations of $G(F)$.

Recall from \cite[Proposition~6.8]{MP96} that if $\pi$ is a depth $0$ irreducible supercuspidal $\ol\Q_\ell$-representation of $G(F)$, then there is a point $x \in \cB(G)$ whose image $[x]$ in $\cB(G_{\der})$ is a vertex and an irreducible cuspidal $\ol\Q_\ell$-representation $\tau$ of $G(F)_{[x]}/G(F)_{x,0+}$ such that $\pi \cong \cInd_{G(F)_{[x]}}^{G(F)}(\tau)$, where $\cInd$ denotes the compact induction functor. We will study such representations through their reductions modulo $\ell$, using the Tate cohomology calculations given above.

Throughout the remainder of this section, we fix a point $x$ such that $[x]$ is a vertex. We let $\ol G_{[x]}$ denote the paraductive $\F_q$-group scheme described in Proposition~\ref{prop:main-paraductive-example}, so $\ol G_{[x]}(\F_q) = G(F)_{[x]}/G(F)_{x,0+}$.

\subsection{The DeBacker--Reeder parametrization}\label{ss:dr}

We now recall and extend the (partial) local Langlands parametrization of \cite{DR09} and \cite{Kal21b}. Let $k$ be a field among $\ol\Q_\ell$ and $\ol\F_\ell$, and let $\tau$ be a non-singular irreducible cuspidal $k$-representation of $\ol G_{[x]}(\F_q)$ in the sense of Definition~\ref{defn:finite-group-ns-reps}. Let $\pi = \cInd_{G(F)_{[x]}}^{G(F)}(\tau)$, so $\pi$ is an irreducible depth $0$ cuspidal $k$-representation of $G(F)$ by \cite[\S 7]{Vig01b}. By definition of non-singularity and Lemma~\ref{lemma:non-singular-conjugacy}, this means that there is a pair $(\ol T, \theta)$, unique up to conjugacy, such that
\begin{enumerate}
    \item $\ol T$ is a generalized maximal $\F_q$-torus of $\ol G_{[x]}$ in the sense of Definition~\ref{def:twisted-levi} such that $\ol T^\circ$ is an elliptic $\F_q$-subtorus of $\ol G_{[x]}^\circ$ (by Lemma~\ref{lemma:ns-dl-res-cuspidal}),
    \item $\theta\co \ol T(\F_q) \to k^\times$ is a non-singular character in the sense of Definition~\ref{defn:finite-group-non-singular},
    \item $\tau$ lies in the semi-rational Lusztig series $\cE(\ol G_{[x]}, [\ol T, \theta])$ in the sense of Definition~\ref{defn:lusztig-series}.
\end{enumerate}
From these data, one can extract a maximally unramified elliptic maximal $F$-torus $T \subset G$ as follows\footnote{Compare with the proof of Proposition~\ref{prop:main-paraductive-example}, which essentially gives this procedure in reverse.}: let $\cG_{[x]}$ denote the smooth separated $\cO_F$-group scheme such that $\cG_{[x]}(\cO_F) = G(F)_{[x]}$ as in \cite[Remark 8.3.4]{KP}, so by \cite[Proposition~11.14(1)]{Bor91} there exists a maximal $\F_q$-torus $\wt T \subset (\cG_{[x]})_{\F_q}$ whose image under the map $r\co (\cG_{[x]})_{\F_q} \to \ol G_{[x]}$ is $\ol T^\circ$. Note that $\wt T$ is unique up to $(\ker r)(\F_q)$-conjugacy because $\ker r$ is unipotent. By \cite[Expos\'e IX, Th\'eor\`eme 3.6, Th\'eor\`eme 7.1]{SGA3II}, there exists an $\cO_F$-subtorus $\cT_0$ of $\cG_{[x]}$ with special fiber $\wt T$, and this subtorus is unique up to $G(F)_{x,0+}$-conjugacy. The generic fiber $T_0 = (\cT_0)_F$ is therefore a maximal unramified $F$-subtorus of $G$, unique up to $G(F)$-conjugacy. If $T = Z_G(T_0)$, then $T$ is a maximally unramified maximal $F$-torus of $G$.

Note that $T$ is elliptic because $\ol T^\circ$ is elliptic. This implies that $T(F) = T(F)_{[x]}$, so $\theta$ induces a character of $T(F)$ which we will also denote by $\theta$. We will say that $\pi$ is \emph{non-singular} if, for a finite unramified extension $E/F$ such that $G_E$ is quasi-split and $T_E$ is maximally split in $G_E$ with maximal split $E$-subtorus $S$, then for each relative root $\alpha \in \Phi(G_E, S_E)$ we have
\begin{equation}\label{eqn:non-singularity-defn}
    \theta \circ \Nm_{E/F} \circ \alpha^\vee|_{\cO_E^\times} \neq 1.
\end{equation}
Observe that this condition is independent of the choice of $E/F$.

We assume from now on that $\pi$ is non-singular.\footnote{The construction we give will still make sense if we only assume that $\tau$ is non-singular, but it will not give the ``true'' semisimple L-parameter in general. Indeed, if $k = \ol\Q_\ell$ then the semisimple L-parameter associated to an irreducible supercuspidal $k$-representation of $G(F)$ should either be discrete or will have infinite image (mod center). The L-parameter we construct always has finite image (mod center).} Using the pair $(T, \theta)$, we will now construct an L-parameter $\rho^{\DR}(x,\tau)\co W_F \to \ld G(k)$ (with notation in recognition of the work of DeBacker--Reeder \cite{DR09}).\footnote{The notation $\rho^{\DR}(\pi)$ would perhaps be more natural, but we do not check directly that $\rho^{\DR}$ is independent of the pair $(x, \tau)$ defining $\pi$ when $k = \ol\F_\ell$. However, Theorem~\ref{thm:main-depth-0-comparison} will show that $\rho^{\DR}(x, \tau)$ does indeed only depend on $\pi$.}

Recall from \cite[\S 6]{Kal21a} that, if $H \subset G$ is a twisted Levi $F$-subgroup, then there is a canonical $W_F$-stable $\wh G(k)$-conjugacy class of embeddings $\wh H \to \wh G$, and by \cite[Remark 6.8]{Kal21a} one can extend any representative in this conjugacy class to an L-embedding $\ld H \to \ld G$ by choosing a set of $\chi$-data for the set of characters $\Phi((G/H)_{\ol F}, (Z(H)^\circ_{\red})_{\ol F})$. If $H$ is an unramified twisted Levi $F$-subgroup of $G$ (for example, a maximally unramified maximal torus), then all elements of $\Phi((G/H)_{\ol F}, (Z(H)^\circ_{\red})_{\ol F})$ are either asymmetric or symmetric unramified in the sense of \cite[\S 2]{Kal21a}; this is observed in \cite[Lemma~3.2.1]{CF26a}. Thus there is a canonical set of $\chi$-data: namely, one can take \emph{minimally ramified} $\chi$-data in the sense of \cite[Definition 4.6.1]{Kal19}. This leads to a (conjugacy class of) L-embedding(s)
\[
\ld j_{H,G}\co \ld H \to \ld G.
\]

Thus we can finally define
\[
\rho^{\DR}(x,\tau) = \ld j_{T,G} \circ \ld\theta\co W_F \to \ld G(k),
\]
where $(T, \theta)$ is the pair constructed above and $\ld\theta\co W_F \to \ld T(k)$ denotes the L-parameter deduced from the Local Langlands Correspondence for tori. Observe that \eqref{eqn:non-singularity-defn} implies that $Z_{\wh G}(\rho^{\DR}(x, \tau)|_{I_F})^\circ$ is a torus, and thus $Z_{\wh G}(\rho^{\DR}(x, \tau))/Z(\wh G)^{W_F}$ is finite since $T$ is elliptic.

We record the following two straightforward lemmas for ease of reference.

\begin{lemma}\label{lemma:depth-0-modular-reduction}
    Let $\tau$ be a $\ol\Z_\ell$-representation of $\ol G_{[x]}(\F_q)$ which is finite free as a $\ol\Z_\ell$-module, and assume that $\tau_{\ol\F_\ell}$ is an irreducible non-singular cuspidal $\ol\F_\ell$-representation. Then $\rho^{\DR}(x, \tau_{\ol\F_\ell})$ is the semisimplified $\ell$-modular reduction of $\rho^{\DR}(x, \tau_{\ol\Q_\ell})$.
\end{lemma}

\begin{lemma}\label{lemma:depth-0-independence-of-ell}
    Let $\tau$ be an irreducible non-singular cuspidal $\ol\Q_\ell$-representation of $\ol G_{[x]}(\F_q)$, let $\ell' \neq p$ be another prime number, and let $\iota\co \ol\Q_\ell \to \ol\Q_{\ell'}$ be a field isomorphism. Then
    \[
    \rho^{\DR}(x, \iota_*\tau) \sim \iota_*\rho^{\DR}(x, \tau).
    \]
\end{lemma}

\subsubsection{Functoriality}

In this section, we check two compatibility statements for $\rho^{\DR}$.

\begin{lemma}\label{lemma:depth-0-bc}
    Let $\tau$ be an irreducible non-singular cuspidal $\ol\F_\ell$-representation of $\ol G_{[x]}(\F_q)$. Let $E/F$ be an unramified field extension of degree $\ell$, and suppose that $\ell$ is banal for $G(F)$ and satisfies Hypothesis~\ref{hypothesis:infinite-glauberman} with $\Gamma = \ol G_{[x]}(\F_{q^\ell})$ and $\sigma$ a generator of $\Gal(E/F)$. If $\tau_\ell$ is the $\sigma$-stable irreducible $\ol\F_\ell$-representation of $\ol G_{[x]}(\F_{q^\ell})$ corresponding to $\tau$ as in \S\ref{sss:infinite-glauberman}, then
    \[
    \rho^{\DR}(x, \tau)|_{W_{F_\ell}} \sim \Fr_\ell \circ \rho^{\DR}(x, \tau_\ell).
    \]
\end{lemma}

\begin{proof}
    Let $(\ol T, \theta)$ be a generalized maximal torus-character pair in $\ol G_{[x]}$ such that $\tau \in \cE(\ol G_{[x]}, [\ol T, \theta])$. By Proposition~\ref{prop:digne-3.5}, there is a unique $\sigma$-stable character $\theta_\ell\co \ol T(\F_{q^\ell}) \to \ol\F_\ell^\times$ extending $\theta$, and this character has the property that $\tau_\ell$ lies in $\cE((\ol G_{[x]})_{\F_{q^\ell}}, [\ol T_{\F_{q^\ell}}, \theta_\ell])$. By Proposition~\ref{prop:banal-primes}, the groups $G$ and $G_E$ have the same split ranks, and the point $x$ has image in $\cB((G_E)_{\der})$ which is a vertex. Thus if $(T, \theta)$ is the pair extracted from $\tau$ as above, then $(T_E, \theta_\ell)$ is the pair extracted from $\tau_\ell$, and the claim follows from the definitions.
\end{proof}

\begin{lemma}\label{lemma:depth-0-unram-twisted-levi}
    Let $k \in \{\ol\Q_\ell, \ol\F_\ell\}$, let $H \subset G$ be an unramified twisted Levi $F$-subgroup such that $x \in \cB(H)$, and let $\tau$ be an irreducible non-singular cuspidal $k$-representation of $\ol G_{[x]}(\F_q)$. If $(\ol T, \theta)$ is a torus-character pair in $\ol H_{[x]}$ such that $\tau \in \cE(\ol G_{[x]}, [\ol T, \theta])$, and if $\tau_H \in \cE(\ol H_{[x]}, [\ol T, \theta])$ is any irreducible $k$-representation, then
    \[
    \rho^{\DR}(x, \tau) \sim \ld j_{H,G} \circ \rho^{\DR}(x, \tau_H).
    \]
\end{lemma}

\begin{proof}
    The pair $(T, \theta)$ extracted above from $(x, \tau)$ is the same as the one extracted from $(x, \tau_H)$. Thus the definitions reduce this to the claim that
    \[
    \ld j_{T, G} \circ \ld\theta \sim \ld j_{H,G} \circ \ld j_{T, H} \circ \ld\theta.
    \]
   This follows from \cite[Proposition~5.27, Proposition~6.9]{Kal21a} after unravelling the definitions of the L-embeddings in \cite[\S 6.1]{Kal21a}; for more details, see \cite[Proposition~5.3.2]{CF26a}.
\end{proof}

\subsection{The Fargues--Scholze parametrization}

Recall that $G$ is a connected reductive $F$-group and $x \in \cB(G)$ is a point whose image $[x]$ in $\cB(G_{\der})$ is a vertex. Let $\rho^{\FS}$ denote the Fargues--Scholze Local Langlands Correspondence, as in \cite{FS}.

The following lemma is extracted from the proof of \cite[Theorem~8.4.1]{F24}.

\begin{lemma}\label{lemma:tate-cohom-compact-induction}
    Let $\sigma$ be an $F$-automorphism of $G$ of order $\ell$, and let $H = G^\sigma$. Let $K$ be a $\sigma$-stable compact-mod-center open subgroup of $G(F)$, and let $K_H = K \cap H(F)$. If $\rho$ is a smooth representation of $K \rtimes \langle\sigma\rangle$ and $j \in \Z/2\Z$, then $\cInd_{K_H}^{H(F)} \rT^j(\sigma,\rho)$ is a direct summand of $\rT^j(\sigma,\cInd_K^{G(F)} \rho)$.
\end{lemma}

\begin{proof}
    By \cite[Proposition~3.3]{TV}, if $\cF_\rho$ is the sheaf on $G(F)/K$ corresponding to $\rho$ and $\Gamma_c$ denotes the functor of compactly supported global sections, then we have
    \[
    \rT^j(\sigma,\cInd_K^{G(F)} \rho) = \rT^j(\sigma,\Gamma_c(G(F)/K, \cF_\rho)) = \Gamma_c((G(F)/K)^\sigma, \rT^j(\sigma,\cF_\rho)).
    \]
    Consider the exact sequence of non-abelian cohomology \cite[\S I.5.4-I.5.5]{Se94}
    \[
    K_H \rightarrow H(F) \rightarrow (G(F)/K)^{\sigma} \rightarrow \rH^1(\sigma, K).
    \]
    Since $\rH^1(\sigma, K)$ is finite and the connecting map $(G(F)/K)^\sigma \to \rH^1(\sigma, K)$ is continuous, we see that $H(F)/K_H$ is an open and closed subspace of $(G(F)/K)^\sigma$. It follows that $\cInd_{K_H}^{H(F)} \rT^j(\sigma,\rho) = \Gamma_c(H(F)/K_H, \rT^j(\sigma,\cF_\rho))$ is a direct summand of $\rT^j(\sigma,\cInd_{K}^{G(F)} \rho)$.
\end{proof}

We will retain the notation of the introduction regarding modular functoriality. We admit the following two facts:
\begin{enumerate}
    \item If $E/F$ is a cyclic extension of degree $\ell$ and $G = \Res_{E/F}(H_E)$ and $\sigma$ is a generator of $\Gal(E/F)$, then the $\sigma$-dual L-homomorphism $\ld\psi\co \ld H \to \ld G \cong \wh H^\ell \rtimes W_F$ is the unique L-embedding extending the diagonal $\wh H \to \wh H^\ell$ which is the identity on the $W_F$-factor; this is verified in \cite[\S A.3.2]{CF26b}.
    \item If $\ell$ is sufficient large and $\sigma$ is induced by conjugation by an element $t \in G(F)$ of order $\ell$ such that $Z_G(t)$ is an unramified twisted Levi $F$-subgroup of $G$, then $\ld\psi|_{I_F} \sim \Fr_\ell \circ \ld j_{H,G}|_{I_F}$, with notation as in the previous section; this follows from \cite[Proposition~4.4.1]{CF26b}.
\end{enumerate}

\begin{prop}\label{prop:depth-0-bc-tate-cohom}
    Let $\tau$ be an irreducible cuspidal $\ol\F_\ell$-representation of $\ol G_{[x]}(\F_q)$. Let $E/F$ be an unramified field extension of degree $\ell$, and suppose that $\ell$ satisfies Hypothesis~\ref{hypothesis:infinite-glauberman} with $\Gamma = \ol G_{[x]}(\F_{q^\ell})$ and $\sigma$ a generator of $\Gal(E/F)$. If $\tau_\ell$ is the $\sigma$-stable irreducible $\ol\F_\ell$-representation of $\ol G_{[x]}(\F_{q^\ell})$ corresponding to $\tau$ as in \S\ref{sss:infinite-glauberman}, then $\cInd_{G(F)_{[x]}}^{G(F)}(\tau)$ is an irreducible subquotient of $\rT^i(\sigma, \cInd_{G(F_\ell)_{[x]}}^{G(F_\ell)}(\tau_\ell))$ for both $i \in \Z/2$.
\end{prop}

\begin{proof}
    By Lemma~\ref{lemma:tate-cohom-infinite-glauberman}, we have $\rT^i(\sigma, \tau_\ell) \cong \tau$. It is clear that $(G(F_\ell)_{[x]})^\sigma = G(F)_{[x]}$, so the claim follows from Lemma~\ref{lemma:tate-cohom-compact-induction} applied to $(\Res_{F_\ell/F}(G_{F_\ell}), G)$ in place of $(G, H)$.
\end{proof}

\begin{cor}\label{cor:fs-depth-0-bc}
    With notation and assumptions as in Proposition~\ref{prop:depth-0-bc-tate-cohom}, if $\ell$ is larger than the bound $b(\wh G)$ in \cite[Theorem~1.3.1]{F24} then
    \[
    \rho^{\FS}(\cInd_{G(F)_{[x]}}^{G(F)}(\tau)) \sim \Fr_\ell \circ \rho^{\FS}(\cInd_{G(E)_{[x]}}^{G(E)}(\tau_\ell))|_{W_E}.
    \]
\end{cor}

\begin{proof}
    This is immediate from Proposition~\ref{prop:depth-0-bc-tate-cohom}, \cite[Theorem~1.3.1]{F24}, and fact (1) above.
\end{proof}

\begin{prop}\label{prop:depth-0-unram-levi-tate-cohom}
    Suppose that $[\ol G_{[x]}(\F_q): (\ol G_{[x]})^\circ(\F_q) \cdot Z(\ol G)(\F_q)]$ is prime to $\ell$. Let $\tau$ be a $\ol\Z_\ell$-representation of $\ol G_{[x]}(\F_q)$ which is a finite free $\ol\Z_\ell$-module with the property that $\tau_{\ol\Q_\ell}$ is an irreducible cuspidal $\ol\Q_\ell$-representation. Let $t \in G(F)_{[x]}$ be an element of order $\ell$, and let $\sigma$ be the $F$-automorphism of $G$ induced by $t$-conjugation. Let $H = G^\sigma$, and assume:
    \begin{enumerate}
        \item $\ell$ is good for $(\ol G_{[x]})^\circ$,
        \item $\tau_{\ol\Q_\ell}$ is defined over a finite extension of $\Q_\ell^{\unr}$ of degree prime to $\ell-1$,
        \item there is a torus-character pair $(\ol T, \theta)$ in $\ol G_{[x]}$ such that $\tau_{\ol\Q_\ell}$ lies in the semi-rational Lusztig series $\cE(\ol G_{[x]}, [\ol T, \theta])$ and $\theta$ is of order prime to $\ell$ and $t \in \ol T_{[x]}(\F_q)$,
        \item $H \subset G$ is an unramified twisted Levi $F$-subgroup.
    \end{enumerate}
    Then there exists an irreducible cuspidal $\ol\F_\ell$-representation $\ol\tau_H$ of $\ol H_{[x]}(\F_q)$, whose Brauer character occurs with nonzero coefficient in the $\ell$-modular reduction of ${}^*R^{\ol G_{[x]}}_{\ol H_{[x]}}(\tau_{\ol\Q_\ell})$, such that for both $a \in \Z/2\Z$ the representation $\cInd_{H(F)_{[x]}}^{H(F)}(\ol\tau_H)$ is an irreducible subquotient of $\rT^a(\sigma, \cInd_{G(F)_{[x]}}^{G(F)}(\tau_{\ol\F_\ell}))$.
\end{prop}

\begin{proof}
    Assumption (4) ensures that the claims make sense. Assumptions (1) and (3) combine with Proposition~\ref{prop:tate-cohom-dl-restriction}(3) to show that $^*R_{\ol H_{[x]}}^{\ol G_{[x]}}(\tau_{\ol\Q_\ell})$ has nonzero $\ell$-modular reduction; let $\ol\tau_H$ be an irreducible $\ol\F_\ell$-representation of $\ol H_{[x]}(\F_q)$ occurring in this reduction. By Lemma~\ref{lemma:tate-cohom-compact-induction}, the Tate cohomology $\rT^a(\sigma, \cInd_{G(F)_{[x]}}^{G(F)}(\tau_{\ol\F_\ell}))$ admits $\cInd_{H(F)_{[x]}}^{H(F)}(\rT^a(\sigma, \tau_{\ol\F_\ell}))$ as a direct summand. On the other hand, assumption (2) and Proposition~\ref{prop:tate-cohom-dl-restriction}(2) imply that $\rT^a(\sigma, \tau_{\ol\F_\ell})$ admits $\ol\tau_H$ as an irreducible subquotient. Thus the claim follows from exactness of $\cInd_{H(F)_{[x]}}^{H(F)}$ \cite[Chapitre I, 5.10 i)]{Vig96}.
\end{proof}

\begin{cor}\label{cor:fs-depth-0-unram-twisted-levi}
    With notation and assumptions as in Proposition~\ref{prop:depth-0-bc-tate-cohom}, there exists a constant $C$ such that if $\ell > C$, then there is an irreducible constituent $\ol\pi$ of $\cInd_{G(F)_{[x]}}^{G(F)}(\tau_{\ol\F_\ell})$ such that
    \[
    \rho^{\FS}(\ol\pi)|_{I_F} \sim \ld j_{H,G} \circ \rho^{\FS}(\cInd_{H(F)_{[x]}}^{H(F)}(\ol\tau_H))|_{I_F}.
    \]
\end{cor}

\begin{proof}
    By Proposition~\ref{prop:depth-0-unram-levi-tate-cohom}, \cite[Theorem~1.3.1]{F24}, fact (2) above, and the fact that semisimplicity is preserved by restriction to a normal subgroup \cite[Theorem 3.10, \S 6.3]{BMR05}, as long as $\ell > b(\wh G)$ one can find $\ol\pi$ such that
    \[
    \rho^{\FS}(\ol\pi)|_{I_F} \sim (\ld j_{H,G} \circ \rho^{\FS}(\cInd_{H(F)_{[x]}}^{H(F)}(\ol\tau_H))|_{I_F})^{\ss}.
    \]
    Let $C$ be larger than $|W_0|$, the constant $b(\wh G)$ in \cite[Theorem 1.3.1]{F24}, and the order of the component group of $\pi_0(N_{\ld H}(\rho^{\FS}(\cInd_{H(F)_{[x]}}^{H(F)}(\tau_H))|_{P_F}))$. Under these hypotheses, if $s \in I_F$ lifts a generator of $I_F/P_F$ then $\ld j_{H,G} \circ \rho^{\FS}(\cInd_{H(F)_{[x]}}^{H(F)}(\ol\tau_H))(s)$ is semisimple and thus the inertial L-parameter is already semisimple by \cite[Lemma 2.6]{BMR05}.
\end{proof}

\subsection{The comparison theorem}

Finally, we prove the following theorem.

\begin{thm}\label{thm:main-depth-0-comparison}
    Let $k$ be a field among $\ol\Q_\ell$ and $\ol\F_\ell$, let $\tau$ be an irreducible non-singular cuspidal $k$-representation of $\ol G_{[x]}(\F_q)$, and let $\pi = \cInd_{G(F)_{[x]}}^{G(F)}(\tau)$. Then
    \[
    \rho^{\FS}(\pi) \sim \rho^{\DR}(x,\tau).
    \]
\end{thm}

To begin, we need a few group-theoretic lemmas.

\begin{lemma}\label{lemma:finite-groups-mod-ell}
    Let $\Gamma$ be a finite group of order not divisible by $\ell$, let $H$ be a smooth affine $\ol{\Z}_\ell$-group scheme with reductive fibers, and let $\rho_1, \rho_2\colon \Gamma \to H(\ol{\Z}_\ell)$ be two homomorphisms. The following are equivalent:
    \begin{enumerate}
        \item $(\rho_1)_{\ol{\F}_\ell}$ and $(\rho_2)_{\ol{\F}_\ell}$ are $H^\circ({\ol{\F}_\ell})$-conjugate,
        \item $(\rho_1)_{\ol{\Q}_\ell}$ and $(\rho_2)_{\ol{\Q}_\ell}$ are $H^\circ(\ol{\Q}_\ell)$-conjugate,
        \item $\rho_1$ and $\rho_2$ are $H^\circ(\ol{\Z}_\ell)$-conjugate.
    \end{enumerate}
\end{lemma}

\begin{proof}
    Let $\sH$ be the finitely presented affine $\ol{\Z}_\ell$-scheme parameterizing homomorphisms $\Gamma \to H$. For either residue field $\kappa$ of $\ol{\Z}_\ell$, every orbit map $H^\circ_\kappa \to \sH_\kappa$ is smooth: indeed, the cokernel of the map $\operatorname{Lie} H^\circ_\kappa \to \operatorname{Tan}_f \sH_\kappa$ is isomorphic to $\rH^1(\Gamma, \operatorname{Lie} H^\circ_\kappa)$ by \cite[Expos\'e III, 2.1(ii), 2.3]{SGA3I}. Since $\Gamma$ is finite of order invertible in $\kappa$, this cohomology group vanishes. Thus each orbit map $H^\circ \to \sH$ is smooth by the fibral flatness criterion; this shows the equivalence of (1) and (3). Moreover, the GIT quotient $\sH/\!/H^\circ$ is reduced and has discrete fibers over $\ol{\Z}_\ell$ (since the natural map $\sH_{\ol\F_\ell}/\!/H^\circ_{\ol\F_\ell} \to (\sH/H^\circ)_{\ol\F_\ell}$ is a universal homeomorphism by \cite[Proposition~5.2.9(3), Theorem~9.1.4, Theorem~9.7.5]{Alp14}\footnote{We note that the necessary word ``affine'' is missing from the published version of \cite[Theorem~9.7.5]{Alp14}, and the unnecessary word ``separated'' is included in the arXiv version. The fact that ``separated'' is unnecessary follows from \cite[Proposition~3.1.3]{Con14}.}), so it is quasi-finite and we conclude the equivalence of (2) and (3) by Zariski's main theorem.
\end{proof}

The following lemma is standard, but we are not aware of a precise reference.

\begin{lemma}\label{lemma:pinning-preserving-fundamental-group}
    Let $M$ be a connected reductive group over a field $k$ of characteristic $\neq p$, and suppose that $\pi_1(M_{\der})$ is of $p$-power order. 
    \begin{enumerate}
        \item If $\alpha$ is a pinning-preserving $k$-automorphism of $M$ of order $p$, then $\pi_1(((M^\alpha)^\circ)_{\der})$ is of $p$-power order.
        \item If $\beta$ is a $k$-automorphism of $M$ of finite order prime to $p$, then $(M_{\der})^\Gamma$ is connected.
    \end{enumerate}
\end{lemma}

\begin{proof}
    We may and do assume that $M$ is semisimple and $k$ is algebraically closed. Let $\pi\co M_{\mathrm{sc}} \to M$ denote the universal cover of $M$, so $\alpha$ induces an automorphism of $M_{\mathrm{sc}}$, and note that the map $(M_{\mathrm{sc}})^\alpha \to (M^\alpha)^\circ$ is surjective with kernel of $p$-power order. Thus for (1) we may and do pass from $M$ to $M_{\mathrm{sc}}$ to assume that $M$ is simply connected. In this case, $M = \prod_{i=1}^n M_i$ is a product of simple $k$-groups $M_i$. We may and do assume that $\alpha$ permutes the $M_i$ transitively, so $n \in \{1, p\}$. If $n = p$, then $M^\alpha \cong M_1$, and the result is clear. If $n = 1$, then $M$ is simple and (1) is standard from the classification of pinning-preserving automorphisms; see for instance \cite[Lemma 5.5]{Cot22} (where $p$ in \textit{loc.\ cit.}\ plays the role of $\operatorname{char} k$ here).

    For (2), recall that $(M_{\mathrm{sc}})^\beta$ is connected by \cite[Theorem 8.1]{St68}. If $m \in M^\beta(k)$ and $\wt m \in M_{\mathrm{sc}}(k)$ lifts $m$, then $\wt m \beta(\wt m)^{-1} \in (\ker \pi)(k)$, so $(\wt m \beta(\wt m)^{-1})^n = \wt m^n \beta(\wt m^n)^{-1}$ since $\wt m \beta(\wt m)$ is central in $M_{\mathrm{sc}}(k)$. If $p^a$ kills $(\ker \pi)(k)$, then it follows that $\wt m^{p^a} \in M_{\mathrm{sc}}(k)^\beta$, hence $m^{p^a} \in (M^\beta)^\circ(k)$. On the other hand, if $\beta^n = 1$, then $\prod_{i=0}^{n-1} \beta^i(\wt m \beta(\wt m)^{-1}) = \wt m \beta^n(\wt m)^{-1} = 1$, so again $m^n \in (M^\beta)^\circ(k)$. Since $n$ and $p$ are relatively prime, it follows that $m \in (M^\beta)^\circ(k)$.
\end{proof}

\begin{lemma}\label{lemma:conjugate-weyl-elements}
    Let $H$ be an algebraic group over an algebraically closed field $k$ such that $H^\circ$ is commutative, let $S \subset H^\circ$ be a connected closed $k$-subgroup, and let $h \in N_H(S)(k)$ be such that $S^h$ is finite. If $h' \in N_H(S)(k)$ has image in $(N_H(S)/S)(k)$ which is $H^\circ(k)$-conjugate to $h$, then $h$ and $h'$ are $H^\circ(k)$-conjugate.
\end{lemma}

\begin{proof}
    We may and do assume $S \subset H$ is normal. By conjugacy, we may and do assume that the images of $h$ and $h'$ in $(H/S)(k)$ are equal; let $\ol h$ be their common image. In this case, it is enough to show that $h$ and $h'$ are $S(k)$-conjugate. Since $S$ is connected, it is clear that $h$ and $h'$ lie in the same connected component of $\pi^{-1}(\ol h) = S \cdot h$, where $\pi\co H \to H/S$ is the natural quotient map. If $t \in S(k)$, then we have $tht^{-1} = (t \cdot \prescript{h}{}{t}^{-1}) h$, so it is equivalent to show that the $k$-homomorphism $f\colon S \to S$ defined by $f(t) = t \cdot \prescript{h}{}{t}^{-1}$ is surjective. But $\ker f = S^h$, which is finite by assumption, so for dimension reasons $f$ must be an isogeny.
\end{proof}

For an integer $n \geq 1$, let $F_n$ denote the unramified extension of $F$ of degree $n$.

\begin{lemma}\label{lemma:torus-torsion-grows} 
    If $T$ is a nontrivial unramified $F$-torus, then for all $n > 1$, we have $T(F)_{p'\textrm{-}\rm{tors}} \subsetneq T(F_n)_{p'\textrm{-}\rm{tors}}$ unless $q = 2$, $n = 2$, and $T$ is a product of norm-one tori corresponding to the quadratic extension $F_2/F$. In particular, for any positive integer $N$, there exists a positive integer $M$ such that for every prime number $\ell \geq M$, the group $T(F_\ell)$ contains an element of prime order $\geq N$ and $\neq p$.
\end{lemma}

\begin{proof}
Since $T$ is unramified, there is an $\cO_F$-torus $\cT$ with generic fiber $T$; if $T_0$ is the special fiber of $\cT$, then it is enough to show that $T_0(\F_q) \subsetneq T_0(\F_{q^n})$. Note that $|T_0(\F_{q^n})|=\det(q^n-\Phi^n\mid X^*(T_{\ol\F_q})\otimes\Q)$, where $\Phi$ is the Frobenius automorphism of $X^*(T_{\ol\F_q})$. Thus, if $\alpha_1,\dots,\alpha_m$ are the eigenvalues of $\Phi$, then the $\alpha_i$ are roots of unity and
\[
|T_0(\F_{q^n})| = \prod_{i=1}^m (q^n - \alpha_i^n) \text{ for all $n$}.
\]
Note that 
\[
|q^n - \alpha_i^n| = |q - \alpha_i| \cdot \Big|\sum_{j=0}^{n-1} q^j\alpha_i^{n-1-j}\Big|.
\]
The second factor can be bounded below by
\begin{equation}\label{eq:torus-points-lower-bound}
\Big|\sum_{j=0}^{n-1} q^j\alpha_i^{n-1-j} \Big| \geq q^{n-1} - \sum_{j=0}^{n-2} q^j,
\end{equation}
with equality if and only if $\alpha_i^{n-1-j}=-1$ for $0\leq j\leq n-2$. The right side of \eqref{eq:torus-points-lower-bound} is $>1$ unless $q=2$, in which case it is $1$; moreover, the equality $\alpha_i^{n-1-j}=-1$ for $0\leq j\leq n-2$ implies $n-1=1$ and $\alpha_i=-1$. If $T_0(\F_q) = T_0(\F_{q^n})$, then$|q^n - \alpha_i^n| = |q - \alpha_i|$ for all $i$, so $q = 2$ and $n = 2$ and $\alpha_i = -1$ for all $i$. Hence $T$ is a product of norm-one tori corresponding to $F_2/F$, as desired.

Now we prove the final claim. Fix a prime number $\ell_0 \leq N$, and note that $T(F)[\ell_0^\infty]$ is finite since the residue field of $F$ is finite\footnote{If $\ell_0$ divides $q$, then $T(F)[\ell_0^\infty]$ can be seen to be finite by observing that this is the case for $\G_{\mathrm{m}}$ and that a finite extension of $F$ splits $T$.}; suppose it is killed by $\ell_0^{k-1}$. If $n$ is a positive integer such that $T(F)[\ell_0^\infty] \subsetneq T(F_n)[\ell_0^\infty]$, then it follows that $T(F)[\ell_0^k] \subsetneq T(F_n)[\ell_0^k]$. Observe that if $\ell$ and $\ell'$ are distinct primes, then $T(F_\ell) \cap T(F_{\ell'}) = T(F)$, so there are only finitely many primes $\ell$ such that $T(F)[\ell_0^k] \subsetneq T(F_\ell)[\ell_0^k]$. Applying this reasoning to the finitely many primes $\ell_0 \leq N$ shows that we can find some integer $M \geq 1$ such that for any $\ell \geq M$, we have $T(F)[(N!)^\infty] = T(F_\ell)[(N!)^\infty]$. By the first claim of this lemma, it follows that there is some prime $\ell' > N$ such that $T(F)[(\ell')^\infty] \subsetneq T(F_\ell)[(\ell')^\infty]$, as desired. 
\end{proof}

\begin{proof}[Proof of Theorem~\ref{thm:main-depth-0-comparison}]
    We will prove this by induction on the semisimple rank of $G$, the case that $G$ is a torus being due to compatibility of $\rho^{\FS}$ and $\rho^{\DR}$ with the usual Local Langlands Correspondence for tori (for $\rho^{\FS}$, this is \cite[Theorem~I.9.6(i)]{FS}). By Lemma~\ref{lemma:ns-cuspidal-lift}, Lemma~\ref{lemma:depth-0-modular-reduction}, and compatibility of $\rho^{\FS}$ with $\ell$-modular reduction, the result for $k = \ol\F_\ell$ follows from the result for $k = \ol\Q_\ell$; thus we may and do assume $k = \ol\Q_\ell$. By twisting by a character of $G(F)$ of depth $0$ at $x$ (using \cite[Theorem~I.9.6(ii)]{FS}), we may and do further assume that $\tau$ has finite order central character. Further, passing to a z-embedding and applying Lemma~\ref{lemma:paraductive-series-extension} and \cite[Theorem~I.9.6(v)]{FS}, we may assume that $G$ has center which is an induced torus. In this case, $\wh G$ has simply connected derived group.

    We claim that we need only show that $\rho^{\DR}(x, \tau)|_{I_F}$ and $\rho^{\FS}(\pi)|_{I_F}$ are $\wh G(\ol\Q_\ell)$-conjugate. For this, let $s \in I_F$ lifting a pro-generator of $I_F/P_F$. Note that $\rho^{\DR}(x, \tau)|_{P_F}$ preserves a common pinning $(\wh B, \wh T, \{X_\alpha\})$ of $\wh G$, and the action of $P_F$ on $X^*(\wh T)$ through $\rho^{\DR}(x, \tau)$ permutes a basis since $\wh G_{\der}$ is simply connected and $Z(G)$ is an induced torus. Thus $X^*(\wh T)_{P_F}$ is torsion-free and so $Z_{\wh G}(\rho^{\DR}(x, \tau)|_{P_F})$ is connected by \cite[Proposition 4.1(d)]{Hai15}. Moreover, the fundamental group of $Z_{\wh G}(\rho^{\DR}(x, \tau)|_{P_F})_{\der}$ is of $p$-power order: by induction on a composition series of the image of $P_F$, this reduces to Lemma~\ref{lemma:pinning-preserving-fundamental-group}(1). Note that the $\ol\Q_\ell$-automorphism $\rho^{\DR}(x,\tau)(s)$ of $Z_{\wh G}(\rho^{\DR}(x, \tau)|_{P_F})$ is of order prime to $p$, so it follows from Lemma~\ref{lemma:pinning-preserving-fundamental-group}(2) that $\wh S \coloneqq Z_{\wh G}(\rho^{\DR}(x, \tau)|_{I_F})$ is connected. Moreover, by construction and non-singularity of $\tau$, it follows that $\wh S$ is a $\ol\Q_\ell$-torus.
    
    If we pass to conjugates so that $\rho^{\DR}(x, \tau)|_{I_F} = \rho^{\FS}(\pi)|_{I_F}$, then for any lift $\Fr \in W_F$ of Frobenius the elements $\rho^{\DR}(x,\tau)(\Fr)$ and $\rho^{\FS}(\pi)(\Fr)$ differ by an element of $\wh S$. Since $\wh S$ is commutative, this implies that these two elements have the same action on $\wh S$. Since the centralizer of $\rho^{\DR}(x, \tau)$ is finite modulo $Z(\wh G)^{W_F}$ by construction, the element $\rho^{\DR}(x,\tau)(\Fr)$ is determined up to $\wh S(\ol\Q_\ell)$-conjugacy by its image in $(\wh G/\wh G_{\der})(k)$ and its action on $\wh S$; this follows from Lemma~\ref{lemma:conjugate-weyl-elements}. Since $\rho^{\DR}$ and $\rho^{\FS}$ respect central characters, the former by construction and the latter by \cite[Theorem~I.9.6(iii)]{FS}, the claim is proven.

    We now begin the argument described in the introduction. By \cite[Proposition~2.1]{GH91} and \cite[Corollary~2.16(c)]{St75}, if $\ell_0$ is prime number which is good for $G$ and does not divide $|\pi_1(G_{\der})|$, then for any element $t \in G(\ol F)$ of order $\ell_0$ the centralizer $Z_{G_{\ol F}}(t)$ is a Levi $\ol F$-subgroup of $G_{\ol F}$. If moreover $E/F$ is an unramified extension such that $t \in G(E)$ and $\ell_0$ does not divide the order of $S(E)_{\tors}$ for any maximal totally ramified $E$-torus $S \subset G_E$ (which excludes only finitely many primes, independently of $E$, namely those which are at most $\rk G + 1$), then $Z_{G_E}(t)$ is an unramified twisted Levi $F$-subgroup of $G_E$. Let $M$ be an integer larger than any of these quantities, as well as $|\ol G_{[x]}^\circ(\F_q)|$, the index $[\ol G_{[x]}(\F_{q^n}): \ol G_{[x]}^\circ(\F_{q^n}) \cdot Z(\ol G_{[x]})(\F_{q^n})]$ for all $n \geq 1$,\footnote{It is easy to check that this index is bounded independently of $n$.} the order of $Z_{G_{\der}}(\ol F)$, the orders of $\rho^{\FS}(s)$ and $\rho^{\DR}(s)$, and the constant $C$ from Corollary~\ref{cor:fs-depth-0-unram-twisted-levi}.

    Let $(T, \theta)$ be a pair associated to $\pi$ as in \S\ref{ss:dr}, and let $T_0 = T \cap G_{\der}$. Let $N$ be an integer larger than the orders of $\rho^{\FS}(s)$ and $\rho^{\DR}(s)$ and large enough that any prime $\ell_1 > N$ is banal for $G(F)$. By Lemma~\ref{lemma:torus-torsion-grows}, if $\ell_1 > N$ is a large enough prime number then there exists a non-central element $t \in T_0(E)$ (where $E$ denotes the unramified extension of $F$ of degree $\ell_1$) of prime order $\ell_0 > M$. Let $H = Z_{G_E}(t)$, so $H$ is an unramified twisted Levi $E$-subgroup of $G_E$ by the previous paragraph.
    
    By independence of $\ell$, i.e., Lemma~\ref{lemma:depth-0-independence-of-ell} and \cite[Theorem~1.1]{Sch25}, we may pass from $\ol\Q_\ell$ to $\ol\Q_{\ell_1}$ to assume $\ell = \ell_1$. Let $\tau_E$ denote the irreducible $\ol\Q_\ell$-representation of $\ol G_{[x]}(\F_{q^\ell})$ corresponding to $\tau$ under the Glauberman correspondence of \S\ref{sss:infinite-glauberman}; this makes sense by Proposition~\ref{prop:banal-primes}(4). By Proposition~\ref{prop:banal-primes}(3), the image of the point $x$ in $\cB((G_{\der})_E)$ is a vertex. By Proposition~\ref{prop:glauberman-cuspidal}, the representation $\tau_E$ is cuspidal; it is non-singular by Proposition~\ref{prop:digne-3.5}. Let $\pi_E = \cInd_{G(E)_{[x]}}^{G(E)}(\tau_E)$. By Lemma~\ref{lemma:depth-0-bc}, we have
    \[
    \rho^{\DR}(x, \ol\tau)|_{W_E} \sim \Fr_\ell \circ \rho^{\DR}(x, \ol\tau_E)
    \]
    and by Corollary~\ref{cor:fs-depth-0-bc} we have
    \[
    \rho^{\FS}(\ol\pi)|_{W_E} \sim \Fr_\ell \circ \rho^{\FS}(\ol\pi_E).
    \]
    Since $\ell \neq p$ is larger than the orders of $\rho^{\DR}(x, \tau)(s)$ and $\rho^{\FS}(\pi)(s)$ by hypothesis and the orders of $\rho^{\DR}(x, \tau_E)(s)$ and $\rho^{\DR}(x, \tau)(s)$ are the same by construction, it suffices by Lemma~\ref{lemma:finite-groups-mod-ell} to show that $\rho^{\DR}(x, \tau_E)|_{I_F}$ and $\rho^{\FS}(\pi_E)|_{I_F}$ are $\wh G(\ol\Q_\ell)$-conjugate. By all the choices above, we may therefore pass from $F$ to $E$ to assume
    \begin{enumerate}
        \item there exists a non-central element $t \in T_0(\ol\Q_\ell)$ of prime order $\ell_0$ larger than the orders of $\rho^{\DR}(x, \tau)(s)$ and $\rho^{\FS}(\pi)(s)$,
        \item the representation $\tau$ lies in $\cE(\ol G_{[x]}, [\ol T, \theta])$ for some generalized maximal torus-character pair $(\ol T, \theta)$ such that $\theta$ is of finite order prime to $\ell_0$,
        \item $[\ol G_{[x]}(\F_q): \ol G_{[x]}^\circ(\F_q) \cdot Z(\ol G_{[x]})(\F_q)]$ is prime to $\ell$,
        \item the representation $\tau$ is defined over $\Q_{\ell_0}^{\unr}$ (using Lemma~\ref{lemma:integer-valued-characters} and Lemma~\ref{lemma:character-values-irreducible-factors}).
    \end{enumerate}

    Let $H = Z_G(t)$, so $H$ is an unramified twisted Levi $F$-subgroup of $G$ by the above, and $Z(H)/Z(G)$ is anisotropic since $H$ contains the elliptic $F$-torus $T$. By independence of $\ell$ again, we may now assume $\ell = \ell_0$. By Lemma~\ref{lemma:depth-0-unram-twisted-levi} and Corollary~\ref{cor:fs-depth-0-unram-twisted-levi} (whose hypotheses hold by (1)-(4) above), as well as Corollary~\ref{cor:bonnafe-11.11}, there is a $\ol\Z_\ell$-representation $\tau_H$ of $\ol H_{[x]}$ such that if $\pi_H = \cInd_{H(F)_{[x]}}^{H(F)}(\tau_H)$ then
    \[
    \rho^{\DR}(x, \ol\tau) \sim \ld j_{H,G} \circ \rho^{\DR}(x, (\tau_H)_{\ol\F_\ell})
    \]
    and
    \[
    \rho^{\FS}(\ol\pi)|_{I_F} \sim \ld j_{H,G} \circ \rho^{\FS}((\pi_H)_{\ol\F_\ell})|_{I_F}.
    \]
    By induction on the semisimple rank of $G$, we have
    \[
    \rho^{\DR}(x, (\tau_H)_{\ol\F_\ell}) \sim \rho^{\FS}((\pi_H)_{\ol\F_\ell}).
    \]
    Since $\ell = \ell_0$ was chosen to be larger than the order of $\rho^{\DR}(x, \tau)(s)$, it is also larger than the order of $\rho^{\DR}(x, \tau_H)(s)$. Thus Lemma~\ref{lemma:finite-groups-mod-ell} shows that $\ell_0$ is the only possible prime number dividing the order of $\rho^{\FS}(\pi)(s)$ but not $\rho^{\DR}(x, \tau)(s)$. Running the same argument again for a different choice of pair $(\ell_1, \ell_0)$ shows that in fact $\rho^{\FS}(\pi)(s)$ and $\rho^{\DR}(x, \tau)(s)$ have the same order, and the three above displayed equations show (as before, using Lemma~\ref{lemma:finite-groups-mod-ell}) that $\rho^{\DR}(x, \tau)|_{I_F}$ and $\rho^{\FS}(\pi)|_{I_F}$ are $\wh G(\ol\Q_\ell)$-conjugate, as desired.
\end{proof}

\begin{remark}\label{rmk:singular-generalization}
    If the answer to Question~\ref{question:cuspidal-lifting} is positive, then the above proof extends immediately to yield a computation of $\rho^{\FS}(\pi)|_{I_F}$ for any cuspidal $k$-representation $\pi$ of depth $0$. In \cite{CF26b}, we will describe a method to get around Question~\ref{question:cuspidal-lifting}, and we will use this method to prove a positive depth generalization of Theorem~\ref{thm:main-depth-0-comparison} when $p \neq 2$ and $G$ is tamely ramified. This same method seems to apply in the depth $0$ case in general (and thus to compute $\rho^{\FS}(\pi)|_{I_F}$ as above), but it would massively complicate matters to use this here. We plan to return to this question, as well as the subtler question of describing the full parameter $\rho^{\FS}(\pi)$, in future work.
\end{remark}


\bibliographystyle{amsalpha-with-labels}
\bibliography{Bibliography}

\end{document}